\documentclass[11pt]{article}

\usepackage{fancyhdr}
\usepackage{amsfonts}
\usepackage{amsmath}
\usepackage{amssymb}
\usepackage{amsthm}
\usepackage{tikz}

\usepackage{amscd}
\usepackage{amstext}

\usepackage{verbatim}

\usepackage[draft=false,colorlinks,bookmarksnumbered,linkcolor=black, citecolor=black]{hyperref}

\usepackage{oldgerm} 

\usepackage{fancyhdr}
\usepackage{amsfonts}
\usepackage{amsmath}
\usepackage{amssymb}
\usepackage{amsthm}

\usepackage{amscd}
\usepackage{amstext}

\usepackage{graphics}
\usepackage{graphicx}

\newlength{\fixboxwidth}
\newcommand*{\distr}[2]{\left\langle #1, #2 \right\rangle}              % application of distributions: \distr{f}{g} -> <f,g>
\newcommand*{\abs}[1]{\left| #1 \right|}                                % | . |
\newcommand*{\norm}[1]{\left\| #1 \right\|}                             % || . ||
\newcommand*{\sep}{\; \vrule \;}                                        % set seperator: { x \in X \sep x=ay }
\renewcommand{\d}{\,\mathrm{d}}											% Differential
\newcommand{\osc}{\mathrm{osc}}											% oscillation
\newcommand{\loc}{\mathrm{loc}}											% loc
\newcommand{\esssup}{ \mathop{\mathrm{ess}\text{-}\mathrm{sup}}\limits }                 % essential supremum

\newcommand{\re}{\mathbb{R}}\newcommand{\N}{\mathbb{N}}
\newcommand{\zz}{\mathbb{Z}}

\newcommand{\com}{\mathbb{C}}
\newcommand{\Z}{{\zz}^d}

\newcommand{\R}{{\re}^d}

\newcommand{\cs}{{\mathcal S}}
\newcommand{\cd}{{\mathcal D}}
\newcommand{\cl}{{\mathcal L}}

\newcommand{\cf}{{\mathcal F}}
\newcommand{\cfi}{{\cf}^{-1}}

\newcommand{\supp}{{\rm supp \, }}

\newcommand{\dist}{{\rm dist \, }}

\newcommand{\be}{\begin{equation}}
\newcommand{\ee}{\end{equation}}
\newcommand{\beq}{\begin{eqnarray}}
\newcommand{\beqq}{\begin{eqnarray*}}
\newcommand{\eeq}{\end{eqnarray}}
\newcommand{\eeqq}{\end{eqnarray*}}

\usepackage[draft=false,colorlinks,bookmarksnumbered,linkcolor=black, citecolor=black]{hyperref}

\usepackage{aliascnt}

\newtheorem{theorem}{Theorem}[section]

\newaliascnt{lem}{theorem}
\newtheorem{lemma}[lem]{Lemma}

\aliascntresetthe{lem}

\newaliascnt{ass}{theorem}

\aliascntresetthe{ass}

\newaliascnt{prop}{theorem}
\newtheorem{prop}[prop]{Proposition}

\aliascntresetthe{prop}

\newaliascnt{cor}{theorem}

\aliascntresetthe{cor}

\newaliascnt{defi}{theorem}
\newtheorem{defi}[defi]{Definition}

\aliascntresetthe{defi}

\theoremstyle{definition}
\newaliascnt{ex}{theorem}

\aliascntresetthe{ex}

\newaliascnt{rem}{theorem}
\newtheorem{remark}[rem]{Remark}

\aliascntresetthe{rem}

\begin{document}

%%%%%%%%%%%%%%%%%%%%%%%%%%%%%%%%%%%%%%%%%%%%%%%%%%%%%%%%%%%%%%%%%%%%%%%%%%%%%%%%%%%%%
%%%%%%%%%%%%%%%%%%%%%%%%%%%%%%%%%%%%%%%%%%%%%%%%%%%%%%%%%%%%%%%%%%%%%%%%%%%%%%%%%%%%%

\title{Oscillations and Differences in\\ Besov-Morrey and Besov-Type Spaces}
\author{Marc Hovemann\footnote{Philipps-Universit\"at Marburg, Institute of Mathematics, Hans-Meerwein-Stra{\ss}e 6, 35043 Marburg, Germany. 
Email: \href{mailto:hovemann@mathematik.uni-marburg.de}{hovemann@mathematik.uni-marburg.de} } $^{,}$\thanks{Marc Hovemann has been supported by Deutsche Forschungsgemeinschaft (DFG), grant HO 7444/1-1 with project number 528343051.
} $^{,}$\footnote{Corresponding author.}
\quad\ and \quad Markus Weimar\footnote{Julius-Maximilians-Universit\"at W\"urzburg (JMU), 
Institute of Mathematics, %Chair of Mathematics IX (Scientific Computing), 
Emil-Fischer-Stra{\ss}e 30, 97074 W\"urzburg, Germany. 
Email: \href{mailto:markus.weimar@uni-wuerzburg.de}{markus.weimar@uni-wuerzburg.de} }
}
\date{\today}

\maketitle

\noindent\textbf{Abstract:} In this paper we investigate Besov-Morrey spaces $\mathcal{N}^{s}_{u,p,q}(\Omega)$ and Besov-type spaces $B^{s,\tau}_{p,q}(\Omega)$ of positive smoothness defined on Lipschitz domains $\Omega \subset \mathbb{R}^d$ as well as on~$\mathbb{R}^d$. We combine the Hedberg-Netrusov approach to function spaces with distinguished kernel representations due to Triebel, in order to derive novel characterizations of these scales in terms of local oscillations provided that some standard conditions concerning the parameters are fulfilled. In connection with that we also obtain new characterizations of $\mathcal{N}^{s}_{u,p,q}(\Omega)$ and~$B^{s,\tau}_{p,q}(\Omega)$ via differences of higher order. By the way we recover and extend corresponding results for the scale of classical Besov spaces $B^{s}_{p,q}(\Omega)$. 

\vspace{0,2 cm}

\noindent\textbf{Key words:} Besov-Morrey space, Besov-type space, Morrey space, Lipschitz domain, oscillations, higher order differences

\vspace{0,2 cm}

\noindent\textbf{Mathematics Subject Classification (2010):} 46E35, 46E30, 41A30, 26B35

\section{Introduction and Main Results}

Nowadays the Besov spaces $B^{s}_{p,q} (\mathbb{R}^d)$ are a well-established tool to precisely describe the regularity of functions and distributions. 
These function spaces have been introduced by Nikol'skij and Besov between 1951 and 1961; see \cite{Ni2}, \cite{Be1959}, and \cite{Be1961}. 
Later on, a Fourier analytic approach towards $B^{s}_{p,q} (\mathbb{R}^d) $ has been investigated in detail in the books of Triebel; see at least \cite{Tr83}, \cite{Tr92}, and \cite{Tr06}. 
In recent years a growing number of authors worked with generalizations of Besov spaces where the $ L_{p}$-quasi-norm is replaced by a Morrey quasi-norm. 
In connection with that Besov-Morrey spaces $  \mathcal{N}^{s}_{u,p,q}(\mathbb{R}^d) $ with $ 0 < p \leq u < \infty$, $ 0 < q \leq \infty $ and $ s \in \mathbb{R} $ showed up. They have been introduced by Kozono and Yamazaki~\cite{KoYa} in 1994. 
Later those Besov-Morrey spaces were studied by Mazzucato~\cite{Maz2003} in the context of Navier-Stokes equations. 
A modification of the definition of Besov-Morrey spaces leads to so-called Besov-type spaces $B^{s , \tau}_{p,q}(\mathbb{R}^d)$ with $s \in \mathbb{R}$, $ 0 \leq  \tau < \infty$, and $0 < p,q \leq \infty$. These function spaces have been introduced by El Baraka~\cite{EB1} around 2002. In the sequel Besov-type spaces also have been investigated by Yang and Yuan~\cite{yy2}. 
However, both the scales $\mathcal{N}^{s}_{u,p,q}(\mathbb{R}^d) $ and $B^{s,\tau}_{p,q}(\mathbb{R}^d)$ also include the original Besov spaces, since we have $\mathcal{N}^{s}_{p,p,q}(\mathbb{R}^d) = B^{s}_{p,q}(\mathbb{R}^d)=B^{s, 0}_{p,q}(\mathbb{R}^d)$. On the other hand, let us remark that for $u \neq p$ or $\tau \neq 0$ and many parameter constellations $  \mathcal{N}^{s}_{u,p,q}(\mathbb{R}^d) $ and $B^{s , \tau}_{p,q}(\mathbb{R}^d)$ do not coincide. We refer to \cite[Corollary 3.3(ii)]{ysy} and \cite[Theorem 1.1]{SawYaYu} for some more details.

When it comes to applications, in the theory of quasi-linear partial differential equations (PDEs) locally weighted $L_p$-averages (such as norms in Morrey spaces) of derivatives constitute a standard tool from the early beginning~\cite{Mor}; see also \cite{DahDieHar+} and the references therein.
In this context it seems advantageous to deal with Besov-Morrey spaces~$\mathcal{N}^{s}_{u,p,q}(\Omega)$ and Besov-type spaces $B^{s, \tau}_{p,q}(\Omega)$ defined on domains $\Omega \subseteq \mathbb{R}^d$ with preferably mild restrictions on the regularity of their boundary. 
Therefore, we will (mostly) concentrate on special or bounded \emph{Lipschitz domains} $\Omega$ in what follows. 

It is the main aim of this paper to prove intrinsic characterizations of the spaces~$\mathcal{N}^{s}_{u,p,q}(\Omega) $ and $B^{s, \tau}_{p,q}(\Omega)$ in terms of oscillations $\osc_{v, \Omega}^{N}f$. 
Here for a function $f \in L_v^{\loc}(\Omega)$ with $0< v \leq \infty$ its local $v$-oscillation of order $N\in\N_0$ is given by
\begin{align*}
\osc_{v, \Omega}^{N}f(x,t) := \inf_{\pi \in \mathcal{P}_{N}} \Big( t^{-d} \int_{B(x,t) \cap \Omega} \abs{f(y) -  \pi (y)}^{v} \d y \Big)^{\frac{1}{v}},
\qquad x\in\Omega,\; t>0,
\end{align*}
whereby $ \mathcal{P}_{N} $ denotes the set of all polynomials with degree at most $N$. For $v = \infty$ the usual modifications have to be made. If $\Omega=\R$, we write $\osc_{v}^{N}f:=\osc_{v,\R}^{N}f$. 
We will investigate under which conditions on the parameters the membership of functions~$f$ in~$\mathcal{N}^{s}_{u,p,q}(\Omega)$ or~$B^{s, \tau}_{p,q}(\Omega)$ can be described by using $\osc_{v, \Omega}^{N}f$ only. 
Such characterizations allow to describe Besov-Morrey and Besov-type spaces in terms of equivalent quasi-norms which measure the decay of bestapproximation errors with regard to polynomials. 
This illustrates the strong relationship between the scales $\mathcal{N}^{s}_{u,p,q}(\Omega)$ and $ B^{s, \tau}_{p,q}(\Omega)$ and abstract approximation theory which is well-known for classical Sobolev and Besov spaces; see, e.g.,~\cite{CioWei} and \cite{Tr89}. 
As an application, characterizations via oscillations could be used to derive sharp assertions on the regularity of PDE solutions as it was done for instance in \cite{BalDieWei} for the original Besov spaces $B^s_{p,q}(\Omega)$. In a sense this paper can be seen as a continuation of~\cite{Ho1} and \cite{HoWe23}, where corresponding descriptions of Triebel-Lizorkin-Morrey spaces~$\mathcal{E}^s_{u,p,q}(\Omega)$ in terms of oscillations have been proven. 

For the classical Besov spaces $B^{s}_{p,q}(\Omega)$, at which $\Omega \subseteq \mathbb{R}^d$ is either $\mathbb{R}^d$ or a bounded $C^{\infty}$-domain, characterizations in terms of oscillations are known since decades. 
We may refer e.g.\ to Triebel \cite{Tr92}, where the following (slightly modified) result can be found. 

\begin{theorem}[{Triebel~\cite[Theorems 3.5.1 and 5.2.1]{Tr92}}]\label{thm_hist1}
    For $d\in\N$ let $\Omega\subseteq\R$ be either~$\mathbb{R}^d$ or a bounded $C^{\infty}$-domain. 
    Let $0 < p \leq \infty$, $0 < q \leq \infty $, $ 1 \leq v \leq \infty$, $N \in \mathbb{N}$, and $s \in \mathbb{R}$ with
    \begin{align*}
        d \max\!\left\{ 0, \frac{1}{p} - \frac{1}{v}  \right\} < s < N. 
    \end{align*}
    Then $f \in L_{\max\{p,v\}}(\Omega)$ belongs to $B^{s}_{p,q}(\Omega)$ if and only if the equivalent quasi-norm
    \begin{align*}
        \norm{ f \sep L_{p}(\Omega) } + \Big (  \int_{0}^{1} t^{- sq} \norm{\osc^{N}_{v, \Omega} f(\cdot , t) \sep L_{p}(\Omega) }^{q} \frac{\d t}{t} \Big )^{\frac{1}{q}}
    \end{align*}
    is finite, where for $q = \infty$ the usual modification has to be made.
\end{theorem}

Further results concerning the original Besov spaces and local oscillations also can be found in DeVore and Popov \cite{DePo}, Dorronsoro \cite[Theorem 1]{Dor1985}, Sawano \cite[Theorem 2.40]{SawBook}, Shvartsman \cite{Shv06}, Triebel \cite[Theorem 2.2.2]{Tr89} and Wallin \cite[Theorem 7]{Wal}. For $\Omega = \mathbb{R}^d$ in the literature one can also find versions of \autoref{thm_hist1} with $ 0 < v \leq \infty$, see for example Hedberg, Netrusov \cite[Theorem 1.1.14]{HN}.

When we turn to Besov-Morrey spaces $\mathcal{N}^{s}_{u,p,q}(\Omega)$ some first characterizations in terms of oscillations for the special case $\Omega = \mathbb{R}^d$ can be found in \cite[Chapter~4.5.3]{ysy}. Some of these results will be recalled in \autoref{sec_char_Rd_1} below. 
Our first main result of this paper extends these assertions. It provides new equivalent quasi-norms in terms of local oscillations for the spaces $\mathcal{N}^{s}_{u,p,q}(\Omega)$, whereby $\Omega$ is either $\mathbb{R}^d$ or a special or bounded Lipschitz domain.

\begin{theorem}[Oscillation Characterizations for Besov-Morrey Spaces]\label{mainresult1}
    Let $d,N \in \mathbb{N}$, $0< p \leq u < \infty$, $0 < q,T,v \leq \infty$, and $0<R<\infty$. Moreover, assume that $s\in\re$ satisfies
    \begin{align*}
		d\, \max\! \left\{ 0, \frac{1}{p} - 1,  \frac{1}{p} - \frac{1}{v} \right\} < s < N.
    \end{align*}
    Then the following statements hold:
    \begin{enumerate}
        \item $\mathcal{N}^s_{u,p,q}(\R)$ is the collection of all $f \in L^{\loc}_{\max\{1,p,v\}}(\R)$ for which
        \begin{align*}
            \norm{ \bigg( \int_{B(\,\cdot\,, R)} \abs{f(y)}^v \d y \bigg)^{\frac{1}{v}} \sep \mathcal{M}^{u}_{p}(\R)}  + \bigg( \int_{0}^{T} t^{-sq} \norm{ \osc_{v}^{N-1} f(\cdot,t) \sep \mathcal{M}^{u}_{p}(\mathbb{R}^d)}^{q} \frac{\d t}{t} \bigg)^{\frac{1}{q}} 
        \end{align*}
        is finite (equivalent quasi-norm).
        Furthermore, the assertion remains valid when $\norm{ \big( \int_{B(\,\cdot\,, R)} \abs{f(y)}^v \d y \big)^{\frac{1}{v}} \sep \mathcal{M}^{u}_{p}(\R)}$ is replaced by $\norm{f \sep \mathcal{M}^{u}_{p}(\R)}$.

        \item In addition assume that $v\geq 1$ and let $\Omega$ be either a special or a bounded Lipschitz domain in $\R$. 
        Then $\mathcal{N}^s_{u,p,q}(\Omega)$ consists of all $f \in L^{\loc}_{\max\{p,v\}}(\Omega)$ for which
        \begin{align*}
            \norm{ \bigg( \int_{B(\,\cdot\,, R)\cap\Omega} \abs{f(y)}^v \d y \bigg)^{\frac{1}{v}} \sep \mathcal{M}^{u}_{p}(\Omega)} + \bigg( \int_{0}^{T} t^{-sq} \norm{ \osc_{v,\Omega}^{N-1} f(\cdot,t) \sep \mathcal{M}^{u}_{p}(\Omega)}^{q} \frac{\d t}{t} \bigg)^{\frac{1}{q}}
        \end{align*}
        is finite (equivalent quasi-norm).
        If additionally $p\geq 1$, this statement remains true when $\norm{ \big( \int_{B(\,\cdot\,, R)\cap\Omega} \abs{f(y)}^v \d y \big)^{\frac{1}{v}} \sep \mathcal{M}^{u}_{p}(\Omega)}$ is replaced by $\norm{f \sep \mathcal{M}^{u}_{p}(\Omega)}$.
      \end{enumerate}
      In all cases, the usual modifications have to be made if $q=\infty$ and/or $v=\infty$.
\end{theorem}
Looking at Besov-type spaces $B^{s, \tau}_{p,q}(\Omega)$, for the special case $\Omega = \mathbb{R}^d$ some first results concerning oscillations can be found in \cite[Chapter 4.4.2]{ysy}. 
Our second main result extends these characterizations in the spirit of \autoref{mainresult1} above.
%It contains new equivalent quasi-norms via local oscillations for the spaces $B^{s, \tau}_{p,q}(\Omega)$, whereat $\Omega$ is either $\mathbb{R}^d$ or a special or bounded Lipschitz domain.

\begin{theorem}[Oscillation Characterizations for Besov-Type Spaces]\label{mainresult2}
    Let $d,N \in \mathbb{N}$ as well as $0< p  < \infty$, $0 \leq \tau < \frac{1}{p}$,  $0 < q,T,v \leq \infty$, and $0<R<\infty$. Further, assume that $s\in\re$ satisfies
    \begin{align*}
		d\, \max\! \left\{ 0, \frac{1}{p} - 1,  \frac{1}{p} - \frac{1}{v} \right\} < s < N.
    \end{align*}
    Then the following statements hold:
    \begin{enumerate}
        \item $B^{s, \tau}_{p,q}(\R)$ is the collection of all $f \in L^{\loc}_{\max\{1,p,v\}}(\R)$ for which
        \begin{align*}
        & \sup_{P\in\mathcal{Q}} \frac{1}{\abs{P}^\tau} \norm{ \bigg( \int_{B(\,\cdot\,,R)} \abs{f(y)}^v\d y \bigg)^{\frac{1}{v}} \sep L_p(P )} \\
        & \qquad \qquad \qquad + \sup_{P\in\mathcal{Q}} \frac{1}{\abs{P}^\tau} \bigg(\int_{0}^T t^{-sq} \norm{ \osc_{v}^{N-1} f(\cdot, t) \sep L_p(P) }^{q} \frac{\d t}{t} \bigg)^{\frac{1}{q}}
        \end{align*}
        is finite (equivalent quasi-norm).
        Moreover, the assertion remains valid when $ \sup_{P\in\mathcal{Q}}\limits \frac{1}{\abs{P}^\tau} \norm{ \Big( \int_{B(\,\cdot\,,R)} \abs{f(y)}^v\d y \Big)^{\frac{1}{v}} \sep L_p(P )}  $ is replaced by $ \sup_{P\in\mathcal{Q}}\limits \frac{1}{\abs{P}^\tau} \norm{ f \sep L_p(P)}   $.

        \item In addition assume that $v\geq 1$ and let $\Omega$ be either a special or a bounded Lipschitz domain in $\R$. 
        Then $B^{s, \tau}_{p,q}(\Omega)$ is the set of all $f \in L^{\loc}_{\max\{p,v\}}(\Omega)$ for which
        \begin{align*}
           & \sup_{P\in\mathcal{Q}} \frac{1}{\abs{P}^\tau} \norm{ \bigg( \int_{B(\,\cdot\,,R)\cap\Omega} \abs{f(y)}^v\d y \bigg)^{\frac{1}{v}} \sep L_p(P \cap \Omega )} \\
           & \qquad \qquad \qquad + \sup_{P\in\mathcal{Q}} \frac{1}{\abs{P}^\tau} \bigg(\int_{0}^T t^{-sq} \norm{ \osc_{v,\Omega}^{N-1} f(\cdot, t) \sep L_p(P\cap \Omega) }^{q} \frac{\d t}{t} \bigg)^{\frac{1}{q}}
        \end{align*}
        is finite (equivalent quasi-norm).
        If additionally $p\geq 1$, this statement remains true when $\sup_{P\in\mathcal{Q}}\limits \frac{1}{\abs{P}^\tau} \norm{ \Big( \int_{B(\,\cdot\,,R)\cap\Omega} \abs{f(y)}^v\d y \Big)^{\frac{1}{v}} \sep L_p(P \cap \Omega )}$ is replaced by $\sup_{P\in\mathcal{Q}}\limits \frac{1}{\abs{P}^\tau} \norm{ f \sep L_p(P\cap \Omega)}$.
      \end{enumerate}
      In all cases, the usual modifications have to be made if $q=\infty$ and/or $v=\infty$.
\end{theorem}

When we concentrate on the case $\Omega = \mathbb{R}^d$ our Theorems \ref{mainresult1} and \ref{mainresult2} can be seen as continuations of \cite[Chapters 4.4.2 and 4.5.3]{ysy} and \cite[Section 8.2]{LiYYSaU}, to cover a larger range of the parameters and to provide  alternative quasi-norms. 
If, in contrast, $\Omega$ is a special or bounded Lipschitz domain and we do either have $p \neq u$ (in \autoref{mainresult1}) or $\tau \neq 0$ (in \autoref{mainresult2}), to the best of our knowledge there are no counterparts of Theorems~\ref{mainresult1} and~\ref{mainresult2} in the literature up to now. 
Let us stress that for $p = u$ and $\tau = 0$, respectively, we particularly recover and extend the known oscillation characterizations of classical Besov spaces~$B^{s}_{p,q}(\Omega)$ given, e.g., in \autoref{thm_hist1} above.

Another important topic of this paper are differences of higher order $\Delta^{N}_{h}f(x)$, see~\eqref{eq_diff_def} below for a precise definition. It is known since many years that local oscillations and differences are closely related to each other. Actually, differences are an important tool for some of our proofs in this paper. Indeed, as a byproduct we obtain new characterizations in terms of higher order differences for the scales of Besov-Morrey and Besov-type spaces. 
For the spaces $\mathcal{N}^{s}_{u,p,q}(\Omega)$ the corresponding result reads as follows.

\begin{theorem}[Difference Characterizations for Besov-Morrey Spaces]\label{thm_main_diff_1}
    Let $d,N \in \mathbb{N}$, $0< p \leq u < \infty$, $0 < q,T,v \leq \infty$, and $0<R<\infty$. Moreover, assume that $s\in\re$ satisfies
    \begin{align}\label{eq:parameter}
		d\, \max\! \left\{ 0, \frac{1}{p}-1,  \frac{1}{p} - \frac{1}{v} \right\} < s < N.
    \end{align}
    Then the following assertions hold true:
    \begin{enumerate}
        \item  $\mathcal{N}^s_{u,p,q}(\R)$ is the set of all $f \in L^{\loc}_{\max\{1,p,v\}}(\mathbb{R}^d)$ for which
        \begin{align*} 
            & \norm{\bigg(  \int_{B(\,\cdot\,, R)}     \abs{f(y)}^{v} \d y  \bigg)^{\frac{1}{v}}  \sep \mathcal{M}^{u}_{p}(\R)} \\
            & \qquad \qquad + \bigg( \int_{0}^{T} t^{-sq} \norm{ \Big( t^{-d} \int_{|h| < t} \abs{\Delta^{N}_{h}f(\cdot)}^v \d h \Big)^{\frac{1}{v}} \sep \mathcal{M}^{u}_{p}(\mathbb{R}^d) }^q \frac{\d t}{t} \bigg)^{\frac{1}{q}}
        \end{align*}
        is finite (equivalent quasi-norm). Furthermore, the assertion remains valid when $\norm{ \big( \int_{B(\,\cdot\,, R)} \abs{f(y)}^v \d y \big)^{\frac{1}{v}} \sep \mathcal{M}^{u}_{p}(\R)}$ is replaced by $\norm{f \sep \mathcal{M}^{u}_{p}(\R)}$.

        \item In addition, assume that $v\geq 1$ and let $\Omega$ be a special Lipschitz domain in $\R$. 
        Then $\mathcal{N}^s_{u,p,q}(\Omega)$ is the collection of all $f \in L^{\loc}_{\max\{p,v\}}(\Omega)$ for which
        \begin{align*} 
            & \norm{\bigg( \int_{B(\,\cdot\,, R) \cap \Omega}     \abs{f(y)}^{v} \d y  \bigg)^{\frac{1}{v}}  \sep \mathcal{M}^{u}_{p}(\Omega)} \\
            & \qquad \qquad +  \bigg( \int_{0}^{T} t^{-sq} \norm{ \Big( t^{-d} \int_{V^{N}(\,\cdot\,,t)} \abs{\Delta^{N}_{h}f(\cdot)}^v \d h \Big)^{\frac{1}{v}} \sep \mathcal{M}^{u}_{p}(\Omega) }^q \frac{\d t}{t} \bigg)^{\frac{1}{q}}
        \end{align*}
        is finite (equivalent quasi-norm).
        If additionally $p\geq 1$, this statement remains true when $\norm{ \big( \int_{B(\,\cdot\,, R)\cap\Omega} \abs{f(y)}^v \d y \big)^{\frac{1}{v}} \sep \mathcal{M}^{u}_{p}(\Omega)}$ is replaced by $\norm{f \sep \mathcal{M}^{u}_{p}(\Omega)}$.

        \item  Additionally, assume that $\Omega$ is a bounded convex Lipschitz domain in $\R$ and let $p>1$ as well as $v=\infty$ such that \eqref{eq:parameter} reduces to $ \frac{d}{p} < s < N$.
        Then $\mathcal{N}^s_{u,p,q}(\Omega)$ consists of all $f \in L^\loc_\infty(\Omega)$ for which
        \begin{align*}
            \norm{ \esssup_{y\in B(\,\cdot\,, R) \cap \Omega} \abs{f(y)} \sep  \mathcal{M}^{u}_{p}( \Omega)} + \bigg( \int_{0}^{T} t^{-sq} \norm{ \esssup_{h\in V^N(\,\cdot\,, t)} \abs{\Delta_h^N f(\cdot)} \sep \mathcal{M}^{u}_{p}(\Omega)}^{q} \frac{\d t}{t} \bigg)^{\frac{1}{q}} 
        \end{align*}
        is finite (equivalent quasi-norm).
        Moreover, the assertion remains valid when 
        $\norm{ \esssup_{y\in B(\,\cdot\,, R) \cap \Omega} \abs{f(y)} \sep  \mathcal{M}^{u}_{p}( \Omega)}$ is replaced by $\norm{ f \sep  \mathcal{M}^{u}_{p}( \Omega)}$.
    \end{enumerate}
    Therein we set $V^{N}(x,t) := \{ h \in \mathbb{R}^d \,:\,  \abs{h} < t \ \mbox{and} \ x + \ell h \in \Omega \ \mbox{for all} \ 0 \leq \ell \leq N\}$ for $t>0$.
    In all cases, the usual modifications have to be made if $q=\infty$ and/or $v=\infty$.
\end{theorem}

\autoref{thm_main_diff_1} above can be seen as a continuation of \cite{HoN}, see also \cite[Chapter~5.6]{H21}. Earlier results concerning Besov-Morrey spaces and differences can be found in \cite[Section~4.5.2]{ysy}. 
In the course of this paper we also provide corresponding new characterizations in terms of differences for Besov-type spaces~$B^{s, \tau}_{p,q}(\Omega)$. 
There is the following result.

\begin{theorem}[Difference Characterizations for Besov-Type Spaces]\label{thm_main_diff_2}
    Let $d,N \in \mathbb{N}$ as well as $0< p  < \infty$, $ 0 \leq \tau < \frac{1}{p} $, $0 < q,T,v \leq \infty$, and $0<R<\infty$. Further, assume that $s\in\re$ satisfies
    \begin{align}\label{eq:parameter2}
		d\, \max\! \left\{ 0, \frac{1}{p}-1,  \frac{1}{p} - \frac{1}{v} \right\} < s < N.
    \end{align}
    Then the following assertions hold true:
    \begin{enumerate}
        \item  $B^{s, \tau}_{p,q}(\R)$ is the set of all $f \in L^{\loc}_{\max\{1,p,v\}}(\mathbb{R}^d)$ for which
        \begin{align*} 
            &  \sup_{P\in\mathcal{Q}} \frac{1}{\abs{P}^\tau} \norm{ \bigg( \int_{B(\,\cdot\,,R)} \abs{f(y)}^v\d y \bigg)^{\frac{1}{v}} \sep L_p(P )}    \\
            & \qquad \qquad + \sup_{P\in\mathcal{Q}} \frac{1}{\abs{P}^\tau} \bigg( \int_{0}^{T} t^{-sq} \norm{ \Big( t^{-d} \int_{|h| < t} \abs{\Delta^{N}_{h}f(\cdot)}^v \d h \Big)^{\frac{1}{v}} \sep L_{p}(P) }^q \frac{\d t}{t} \bigg)^{\frac{1}{q}}
        \end{align*}
        is finite (equivalent quasi-norm). Moreover, the assertion remains valid when $ \sup_{P\in\mathcal{Q}}\limits \frac{1}{\abs{P}^\tau} \norm{ \Big( \int_{B(\,\cdot\,,R)} \abs{f(y)}^v\d y \Big)^{\frac{1}{v}} \sep L_p(P )}$ is replaced by $ \sup_{P\in\mathcal{Q}}\limits \frac{1}{\abs{P}^\tau} \norm{ f \sep L_p(P)}$.

        \item In addition, assume that $v\geq 1$ and let $\Omega$ be a special Lipschitz domain in $\R$. 
        Then $B^{s, \tau}_{p,q}(\Omega)$ is the collection of all $f \in L^{\loc}_{\max\{p,v\}}(\Omega)$ for which
        \begin{align*} 
            &\sup_{P\in\mathcal{Q}} \frac{1}{\abs{P}^\tau} \norm{ \bigg( \int_{B(\,\cdot\,,R)\cap\Omega} \abs{f(y)}^v\d y \bigg)^{\frac{1}{v}} \sep L_p(P \cap \Omega )}    \\
            & \qquad \qquad +  \sup_{P\in\mathcal{Q}} \frac{1}{\abs{P}^\tau} \bigg( \int_{0}^{T} t^{-sq} \norm{ \Big( t^{-d} \int_{V^{N}(\,\cdot\,,t)} \abs{\Delta^{N}_{h}f(\cdot)}^v \d h \Big)^{\frac{1}{v}} \sep L_{p}(P\cap \Omega) }^q \frac{\d t}{t} \bigg)^{\frac{1}{q}}
        \end{align*}
        is finite (equivalent quasi-norm).
        If additionally $p\geq 1$, this statement remains true when $\sup_{P\in\mathcal{Q}}\limits\frac{1}{\abs{P}^\tau} \norm{ \Big( \int_{B(\,\cdot\,,R)\cap\Omega} \abs{f(y)}^v\d y \Big)^{\frac{1}{v}} \sep L_p(P \cap \Omega )}   $ is replaced by $\sup_{P\in\mathcal{Q}}\limits \frac{1}{\abs{P}^\tau} \norm{ f \sep L_p(P\cap \Omega)}$.

        \item  Additionally, assume that $\Omega$ is a bounded convex Lipschitz domain in $\R$ and let $p>1$ as well as $v=\infty$ such that \eqref{eq:parameter2} reduces to $\frac{d}{p} < s < N$.
        Then $B^{s, \tau}_{p,q}(\Omega)$ consists of all $f \in L^\loc_\infty(\Omega)$ for which
        \begin{align*}
            &\sup_{P\in\mathcal{Q}} \frac{1}{\abs{P}^\tau} \norm{ \esssup_{y\in B(\,\cdot\,, R) \cap \Omega} \abs{f(y)} \sep  L_{p}( P\cap \Omega)} \\
            &\qquad\qquad + \sup_{P\in\mathcal{Q}} \frac{1}{\abs{P}^\tau} \bigg( \int_{0}^{T} t^{-sq} \norm{ \esssup_{h\in V^N(\,\cdot\,, t)} \abs{\Delta_h^N f(\cdot)} \sep L_{p}(P\cap \Omega)}^{q} \frac{\d t}{t} \bigg)^{\frac{1}{q}} 
        \end{align*}
        is finite (equivalent quasi-norm).
        Moreover, the assertion remains valid when 
        $\sup_{P\in\mathcal{Q}}\limits \frac{1}{\abs{P}^\tau} \norm{ \esssup_{y\in B(\,\cdot\,, R) \cap \Omega} \abs{f(y)} \sep  L_{p}(P\cap \Omega)}$ is replaced by $\sup_{P\in\mathcal{Q}}\limits \frac{1}{\abs{P}^\tau} \norm{ f \sep L_p(P\cap \Omega)}$.
    \end{enumerate}
    In all cases, the usual modifications have to be made if $q=\infty$ and/or $v=\infty$.
\end{theorem}
Note that \autoref{thm_main_diff_2} extends some results of \cite{HoSi20}; see also \cite[Chapter 5.8]{H21}. Earlier results concerning Besov-type spaces and differences can be found in \cite[Chapter~4.3.2]{ysy}, \cite{ZSYY1} and in \cite{Dri1}.

This paper is organized as follows. In \autoref{sec_2_definitions} we recall the definitions of Besov-Morrey and Besov-type spaces defined on both $\mathbb{R}^d$ and domains. Moreover, we collect some useful properties of these spaces that are used in our proofs later. 
\autoref{sec_3_norms} contains the definitions of our new quasi-(semi-)norms which are frequently used later on. Here also local oscillations and differences of higher order are defined properly. 
Then \autoref{sec_char_Rd} is devoted to characterizations of $\mathcal{N}^{s}_{u,p,q}(\mathbb{R}^d)$ and $B^{s, \tau}_{p,q}(\mathbb{R}^d)$ in terms of local oscillations and higher order differences. The core of the paper is given by \autoref{sect:characterizations_domains}, where we derive equivalent descriptions via oscillations and higher order differences for the spaces $\mathcal{N}^{s}_{u,p,q}(\Omega)$ and $B^{s, \tau}_{p,q}(\Omega )$ defined on special or bounded Lipschitz domains $\Omega$. 
Finally, the proofs of Theorems \ref{mainresult1}--\ref{thm_main_diff_2} are given in \autoref{sect:proofs}.
However, first of all we shall fix some notation.

\medskip

\noindent\textbf{Notation.} As usual, $\N$ denotes the natural numbers, $\N_0:=\N\cup\{0\}$, $\zz$ describes the integers and~$\re$ the real numbers. Further, $\R$ with $d\in\N$ denotes the $d$-dimensional Euclidean space and we put
$$
    B(x,t) := \left\{y\in \R \,: \abs{x-y}< t \right\}\, , \qquad x \in \R,\,\; t>0.
$$
All functions are assumed to be complex-valued, i.e.\ we consider functions $f\colon \R \to \com$. 
We let $\mathcal{S}(\R)$ be the collection of all Schwartz functions on $\R$ endowed with the usual topology and by $\mathcal{S}'(\R)$ we denote its topological dual, the space of all bounded linear functionals on~$\mathcal{S}(\R)$ equipped with the weak-$\ast$ topology. 
The symbol $\cf$ refers to the Fourier transform and $\cfi$ to its inverse, both defined on $\cs'(\R)$. 
For domains (open connected sets) $\Omega\subseteq\R$ and $0<v\leq \infty$, by $L_v^\loc(\Omega)$ we mean the set of locally $v$-integrable (or locally essentially bounded) functions on $\Omega$. Furthermore, $\mathcal{D}(\Omega)=C_0^\infty(\Omega)$ denotes the set of infinitely often differentiable functions with compact support on $\Omega$. 
Its topological dual, $\mathcal{D}'(\Omega)$, is the space of distributions on $\Omega$.
Almost all function spaces considered in this paper are subspaces of regular distributions from $\cs'(\R)$ or $\mathcal{D}'(\Omega)$, interpreted as spaces of equivalence classes with respect to almost everywhere equality. %However, if such an equivalence class contains a continuous representative, we usually work with this one and call the respective equivalence class a continuous function. 
Given two quasi-Banach spaces $X$ and $Y$, the norm of a linear operator $T\colon X\to Y$ is denoted by $\norm{T \sep \cl (X,Y)}$. 
Moreover, we write $X \hookrightarrow Y$ if the natural embedding of $X$ into $Y$ is continuous. 
For $0<p<\infty$ we shall use the well-established quantity
$$
   \sigma_p:= d\,  \max\!\left\{0, \frac 1p - 1\right\}.
$$
The symbols $C, C_1, c, c_{1}, \ldots$ denote positive constants depending only on the fixed para\-meters and probably on auxiliary functions. 
Unless otherwise stated their values may vary from line to line. 
With $A \lesssim B$ we mean $ A \leq C B  $ for a constant $ C > 0 $ independent of $A$ and $B$. The notation $ A \sim B $ stands for $A \lesssim B$ and $B \lesssim A$.

\section{Function Spaces}\label{sec_2_definitions}
\subsection{Morrey Spaces \texorpdfstring{$\mathcal{M}^{u}_{p}(\mathbb{R}^d)$}{Mup(Rd)} and Further Basic Definitions}

Besov-Morrey and Besov-type spaces are built upon Morrey spaces which can be defined as follows.

\begin{defi}[Morrey Spaces]\label{def_mor}
For $d\in\N$ and $ 0 < p \leq u < \infty$ the Morrey space~$ \mathcal{M}^{u}_{p}(\R)$ is the set of all $ f \in L_{p}^{\loc}(\R) $ such that 
\begin{align*}
\norm{ f \sep \mathcal{M}^{u}_{p}(\R) } := \sup_{y \in \R, r > 0} \abs{ B(y,r) }^{\frac{1}{u}-\frac{1}{p}} \Big ( \int_{B(y,r)} \abs{ f(x) }^{p} \d x      \Big )^{\frac{1}{p}} < \infty.
\end{align*} 
\end{defi}
Morrey spaces have been introduced already in 1938 by Morrey~\cite{Mor}. They are quasi-Banach spaces and even Banach spaces if $ p \geq 1$. Note that an equivalent quasi-norm can be obtained by replacing the balls in \autoref{def_mor} by dyadic cubes. For $ 0 < p_{2} \leq p_{1} \leq u < \infty $ we have
\begin{align*}%\label{f_mor_emb}
L_{u}(\R) = \mathcal{M}^{u}_{u}(\R) \hookrightarrow   \mathcal{M}^{u}_{p_{1}}(\R)  \hookrightarrow  \mathcal{M}^{u}_{p_{2}}(\R),
\end{align*}
such that Morrey spaces can be viewed as a refinement of the Lebesgue spaces.
However, for $ p \not = u   $ the spaces $ \mathcal{M}^{u}_{p}(\R)$ are much more complicated than $L_p(\R)$. 

To properly define Besov-Morrey and Besov-type spaces we require a so-called smooth dyadic decomposition of unity. 
For that purpose let $ \psi \in C_0^{\infty}(\mathbb{R}^d)$ be a non-negative function such that $\psi (x) = 1$ if $|x|\leq 1$ and $ \psi (x) = 0$ if $|x| \geq \frac{3}{2}$. 
Then we define a family of functions $ ( \varphi_k )_{k \in \mathbb{N}_{0}}  $ by $ \varphi_{0} := \psi$ and $ \varphi_k(x) := \varphi_0(2^{-k}x)-\varphi_0(2^{-k+1}x) $ for $k\in \mathbb{N}$ and $ x \in \R $. These functions have some special properties. For their supports we observe $ \supp \varphi_k \subset \{ x\in \R: \: 2^{k-1}\le |x|\le 3 \cdot 2^{k-1} \} $ if $ k \in \N $. Moreover, there is the identity $\sum_{k=0}^\infty \varphi_k(x) = 1 $ for $  x\in \R $.
Based on this every $f \in \mathcal{S}'(\mathbb{R}^{d})$ can be decomposed into a series of building blocks $ \cfi[\varphi_{k}\, \cf f]$, $k\in\N_0$, where by the Paley-Wiener-Schwartz Theorem each of them can be interpreted as a smooth function for any $ f \in \mathcal{S}'(\mathbb{R}^{d})$. 

\subsection{Besov-Morrey Spaces \texorpdfstring{$\mathcal{N}^{s}_{u,p,q}(\mathbb{R}^d)$}{Nsupq(Rd)}}

Now we are prepared to define the Besov-Morrey spaces $\mathcal{N}^{s}_{u,p,q}(\mathbb{R}^{d})$.

\begin{defi}[Besov-Morrey Spaces]\label{def_bms}
	Let $ (\varphi_{k})_{k\in \N_0 }$ be a smooth dyadic decomposition of unity in $\R$, $d\in\N$. Then, for $ s \in \mathbb{R}$, $ 0 < p \leq u < \infty $, and $ 0 < q \leq \infty $ the Besov-Morrey space $  \mathcal{N}^{s}_{u,p,q}(\mathbb{R}^{d}) $ is the collection of all distributions $ f \in \mathcal{S}'(\mathbb{R}^{d})$ such that
\begin{align*} 
    \norm{ f \sep \mathcal{N}^{s}_{u,p,q}(\mathbb{R}^{d}) } 
    :=  \bigg ( \sum_{k = 0}^{\infty} 2^{ksq}   \norm{ \mathcal{F}^{-1}[\varphi_{k} \mathcal{F}f]  \sep \mathcal{M}^{u}_{p}(\R)   }^{q} \bigg)^{\frac{1}{q}} < \infty .
\end{align*}
In the case $ q = \infty $ the usual modifications are made.
\end{defi}

In what follows we collect some basic properties of these Besov-Morrey spaces. 
Most of them are well-known, but will be important for our later considerations. First of all, for every $  s \in \mathbb{R}  $, $ 0 < p \leq u < \infty $, and $ 0 < q \leq \infty $ the space $  \mathcal{N}^{s}_{u,p,q}(\mathbb{R}^{d}) $ is independent of the chosen smooth dyadic decomposition of unity in the sense of equivalent quasi-norms; see \cite[Theorem 2.8]{TangXu}. 
Moreover, all spaces $  \mathcal{N}^{s}_{u,p,q}(\mathbb{R}^{d}) $ are quasi-Banach spaces. For $ p \geq 1 $ and $ q \geq 1  $ they are Banach spaces; see \cite[Corollary~2.6]{KoYa}. 
According to \cite[Theorem~3.2]{SawTan}
it holds $\mathcal{S}(\mathbb{R}^{d}) \hookrightarrow    \mathcal{N}^{s}_{u,p,q}(\mathbb{R}^{d}) \hookrightarrow   \mathcal{S}'(\mathbb{R}^{d})$. 
Furthermore, for $u = p$ we recover the original Besov spaces, namely we have $  \mathcal{N}^{s}_{p,p,q}(\mathbb{R}^{d}) = B^{s}_{p,q}(\R)$; cf.\ \cite[Proposition 3.6]{SawTan}. 

For us it will be important to know, whether a distribution that belongs to some~$\mathcal{N}^{s}_{u,p,q}(\mathbb{R}^{d})   $ can be interpreted as a locally integrable function. 
In this regard there is the following observation. For a somewhat sharper result that also covers the limiting case we refer to \cite[Theorems~3.2 and 3.4]{HaMoSk2}.

\begin{lemma}[{\cite[Theorem 3.3]{HaMoSk}}]\label{lem_si_nlc}
Let $  s \in \mathbb{R}  $, $ 0 < p \leq u < \infty$, and $ 0 < q \leq \infty   $. Then
\begin{align*}
\mathcal{N}^{s}_{u,p,q}(\mathbb{R}^{d}) \subset L_{1}^{\loc}(\R) \quad   \mbox{if} \quad  s > \frac{p}{u} \sigma_{p} \quad \mbox{and} \quad \mathcal{N}^{s}_{u,p,q}(\mathbb{R}^{d}) \not \subset L_{1}^{\loc}(\R)  \quad  \mbox{if}  \quad  s < \frac{p}{u} \sigma_{p}.
\end{align*}
\end{lemma} 

For the Besov-Morrey spaces the following elementary embedding holds. In this context let us also refer to \cite[Theorem 3.4]{HaMoSk2}, where a similar result can be found.

\begin{lemma}\label{l_el_em_EN}
Let $ s \in \mathbb{R}   $, $ 0 < p \leq u < \infty$, and $ 0 < q \leq \infty $. Then $s>\sigma_p$ implies  $\mathcal{N}^{s}_{u,p,q}(\R) \hookrightarrow  \mathcal{M}^{u}_{p}(\R)$. 
\end{lemma} 

\begin{proof}
    Choose $\widetilde{s}\in\re$ with $s>\widetilde{s}> \sigma_p$. 
    Then \cite[Proposition 3.6]{SawTan} and \cite[Lemma 8(iii)]{H21} yield
    \begin{align}\label{eq:embeddingNE}
        \mathcal{N}^{s}_{u,p,q}(\R)
        \hookrightarrow \mathcal{N}^{\widetilde{s}}_{u,p,\min\{p,q\}}(\R)
        \hookrightarrow \mathcal{E}^{\widetilde{s}}_{u,p,q}(\R)
    \end{align}
    with $\mathcal{E}^{\widetilde{s}}_{u,p,q}(\R)$ being a Triebel-Lizorkin-Morrey space; see for example \cite[Chapter 1.3.3]{ysy} for a definition. Hence, the assertion follows from \eqref{eq:embeddingNE} and \cite[Lemma 5(iv)]{HoWe23}.
\end{proof}

\subsection{Besov-Type Spaces \texorpdfstring{$B^{s,\tau}_{p,q}(\mathbb{R}^d)$}{Bstaupq(Rd)}}

Hereinafter we recall the definition of the Besov-type spaces $ B^{s, \tau}_{p,q}(\mathbb{R}^{d})  $. For that purpose we have to work with dyadic cubes. Let $ \mathcal{Q} $ be the collection of all dyadic cubes in $ \mathbb{R}^{d} $,
\begin{align*} 
\mathcal{Q} := \left\{  Q_{j,k}\subset\R \;: Q_{j,k} = 2^{-j}([0,1)^{d}+k) \ \mbox{with} \ j \in \mathbb{Z} \ \mbox{and} \ k \in \mathbb{Z}^{d} \right\}.
\end{align*} 
Further, let $\ell( P ) $ denote the side-length of a cube $ P \in \mathcal{Q} $ and write $ j_{P} := - \log_{2}(\ell(P))$.
\begin{defi}[Besov-Type Spaces $B^{s,\tau}_{p,q}(\R)$]\label{def_BTS}
Let $ (\varphi_{k})_{k\in \N_0 }$ be a smooth dyadic decomposition of unity in $\R$ with $d\in\N$. 
Then, for $ s \in \mathbb{R} $, $ 0 \le \tau <\infty$, and  $0 < p,q \le \infty$, the Besov-type space $  B^{s,\tau}_{p,q}(\R) $ is the set of all distributions $ f \in \mathcal{S}'(\mathbb{R}^{d})$ such that
\begin{align*}
    \norm{ f \sep B^{s,\tau}_{p,q}(\R) } 
    :=  \sup_{P \in \mathcal{Q} } \frac{1}{\abs{P}^{\tau}} \bigg ( \sum_{k = \max\{j_{P} , 0\}}^{\infty} 2^{ksq}  \Big ( \int_{P} \abs{  \cfi[\varphi_{k}\, \cf f](x) }^{p} \d x \Big )^{\frac{q}{p}} \bigg )^{\frac{1}{q}} < \infty .
\end{align*}
In the cases $ p = \infty $ and/or $ q = \infty $ the usual modifications have to be made.
\end{defi}

Below we collect some basic properties of the spaces $ B^{s,\tau}_{p,q}(\R)  $. 
Most of them are already well-known and will be important for us later. For that purpose, we let $ s \in \mathbb{R} $, $ \tau \geq 0 $, $ 0 < p < \infty $, and $ 0 < q \leq \infty$. 
Then $ B^{s,\tau}_{p,q}(\R) $ is independent of the chosen smooth dyadic decomposition of unity in the sense of equivalent quasi-norms; see \cite[Corollary 2.1]{ysy}. Moreover, the spaces $ B^{s,\tau}_{p,q}(\R)  $ are quasi-Banach spaces; cf.\ \cite[Lemma~2.1]{ysy}.  Furthermore, we observe $\mathcal{S}(\mathbb{R}^{d}) \hookrightarrow  B^{s,\tau}_{p,q}(\R)  \hookrightarrow   \mathcal{S}'(\mathbb{R}^{d})$, see \cite[Proposition~2.3]{ysy}. If $\tau = 0$, we recover the original Besov spaces: $B^{s,0}_{p,q}(\R) = B^{s}_{p,q}(\R)$. Also for $ \tau > \frac{1}{p}$ the spaces~$B^{s,\tau}_{p,q}(\R)$ are well-known. 
In that case the Besov-type spaces coincide with classical H\"older-Zygmund spaces, namely by \cite[Theorem~2(ii)]{yy2013} we have $B^{s,\tau}_{p,q}(\R) = B^{s + d( \tau - \frac{1}{p})}_{\infty,\infty}(\R)$ in the sense of equivalent quasi-norms. The borderline case $\tau = \frac{1}{p}$ has been investigated in \cite[Proposition 3]{yy2013}. Because of these observations in what follows we will concentrate on  the range $ 0 \leq \tau <  \frac{1}{p}  $. Here, for $\tau \neq 0$, in most cases the spaces $  B^{s,\tau}_{p,q}(\R)  $ do not coincide with Besov-Morrey spaces, see \cite[Corollary 3.3(ii)]{ysy} and \cite[Theorem~1.1]{SawYaYu} for some more details. However, for $ \tau <  \frac{1}{p}$ there exists an alternative simplified definition; see \cite[Proposition~3.1]{Si1}. Later on, it will be interesting to know under which restrictions on the parameters a space~$ B^{s,\tau}_{p,q}(\R) $ does or does not contain 
singular distributions. In this respect, the following result can be found in the literature.

\begin{lemma}[{\cite[Theorem~3.6]{HaMoSk}}]\label{l_BT_bp2}
 Let $ s \in \mathbb{R}  $, $ 0 < p <\infty $,  $ 0 \leq \tau < \frac{1}{p} $ and $ 0 < q \leq \infty$. Then the following assertions are true.

\begin{itemize}
\item[(i)] If either $ s > 0 $ and $ p \geq 1 $, or $ s > d  ( \frac{1}{p}  - 1  ) - d \tau (1 - p)$ and $ p < 1 $, then we have $  B^{s,\tau}_{p,q}(\R) \subset L_{1}^{\loc}(\R)  $.

\item[(ii)] If either $ s < 0 $ and $ p \geq 1 $, or $ s < d  ( \frac{1}{p}  - 1  ) - d \tau (1 - p)$ and $ p < 1 $, then we find $  B^{s,\tau}_{p,q}(\R) \not \subset L_{1}^{\loc}(\R)  $.
\end{itemize}
\end{lemma}

Moreover, there is the following simple embedding. 

\begin{lemma}\label{l_BF_bp1}
    Let $ s \in \mathbb{R} $, $ 0 < p < \infty $, $ 0 < q \leq \infty   $, $ 0 \leq \tau < \frac{1}{p} $ and define $u:=1/(\frac{1}{p}-\tau)$. 
    Then $s>\sigma_p$ implies $B^{s,\tau}_{p,q}(\R) \hookrightarrow \mathcal{M}^u_p(\R)$.
\end{lemma} 

\begin{proof}
For the proof we may choose $\widetilde{s}\in\re$ such that $s>\widetilde{s}> \sigma_p$. 
    Then \cite[Proposition~2.1]{ysy} together with \cite[Proposition 2.1(iii)]{ysy} and \cite[Lemma 13]{H21} yield
    \begin{align*}
        B^{s,\tau}_{p,q}(\R)
        \hookrightarrow B^{\widetilde{s},\tau}_{p,\min\{p,q\}}(\R)
        \hookrightarrow F^{\widetilde{s},\tau}_{p,q}(\R)
        = \mathcal{E}^{\widetilde{s}}_{u,p,q}(\R)
    \end{align*}
    with $F^{\widetilde{s},\tau}_{p,q}(\R)$ and $\mathcal{E}^{\widetilde{s}}_{u,p,q}(\R)$ being Triebel-Lizorkin-type and Triebel-Lizorkin-Morrey spaces, respectively; see, e.g., \cite[Chapter 1.3.3]{ysy} for a definition. Now the assertion follows from \cite[Lemma 5(iv)]{HoWe23}.
\end{proof}

\subsection{Function Spaces on Domains}
It is possible to define Morrey, Besov-Morrey, and Besov-type spaces on domains as well. For that purpose we employ the following general definition. 

\begin{defi}[Function Spaces on Domains]\label{defi:space_domain}
For $d\in\N$ let $X ( \mathbb{R}^{d} )$ be a quasi-normed space of tempered distributions such that $X( \mathbb{R}^{d} ) \hookrightarrow \cs'( \mathbb{R}^{d} )$.
Further, let $\Omega $ denote an open, nontrivial subset of $ \mathbb{R}^{d} $.
Then  $X(\Omega)$ is defined as the collection of all $f \in \cd' (\Omega)$ such that
there exists a distribution $g \in X( \mathbb{R}^{d} )$ satisfying
\begin{align*}
f (\varphi) = g (E\varphi) \qquad \mbox{for all} \qquad \varphi \in \cd (\Omega) \, .
\end{align*}
Here $E\varphi \in \mathcal{S}(\R)$ denotes the extension of $\varphi \in \cd (\Omega)$ by zero on $ \mathbb{R}^{d} \setminus \Omega$.
We put
\begin{align*}
    \norm{ f \sep X(\Omega) } := \inf \Big\{\norm{ g \sep X( \mathbb{R}^{d} ) }:\, g\in X(\R) \text{ with } g|_\Omega =f  \Big\} \, .
\end{align*}
\end{defi}
For Morrey spaces on domains we can use the following (alternative but equivalent) direct approach; cf.\ \cite[Lemma~3]{HoWe23}.
Therein, $M(\Omega)$ denotes the set of all measurable complex-valued functions on $\Omega\subseteq\R$. 
\begin{defi}[Morrey Space $\mathcal{M}^{u}_{p}(\Omega)$]
\label{def:morrey_mod}
    For $d\in\N$ let $\Omega\subseteq\R$ be a domain.
    Then for $0<p \leq u < \infty$ the Morrey space $\mathcal{M}^{u}_{p}(\Omega)$ is given by
    $$
        \mathcal{M}^{u}_{p}(\Omega) := \Big\{ f\in M(\Omega) \,:\, \norm{f \sep \mathcal{M}^{u}_{p}(\Omega)}<\infty \Big\}
    $$
    endowed with the (quasi-)norm
    $$
        \norm{f \sep \mathcal{M}^{u}_{p}(\Omega)} := \sup_{y\in\Omega, r>0} r^{d(\frac{1}{u}- \frac{1}{p})} \left( \int_{\Omega\cap B(y,r)} \abs{f(x)}^p \d x \right)^{\frac{1}{p}}.
    $$
\end{defi}

Let us remark that for $\Omega = \mathbb{R}^d$ this definition is equivalent (modulo some multiplicative constant) to \autoref{def_mor} above. 
On the other hand, for $\Omega \subsetneq \mathbb{R}^d$ we note that the weight does not change for balls close to~$\partial\Omega$ in order to avoid restrictions on~$\Omega$ (such as the so-called measure density condition). Further explanations concerning the usability of \autoref{def:morrey_mod} (in particular in the context of Lipschitz domains) can be found in \cite[Remark 2.2(iii)]{HarSchSkr18}. 
Finally, note that also in \autoref{def:morrey_mod} it is possible to replace balls by dyadic cubes; see also \cite[Chapter 1.2.1 and Remark 1]{DifazHaSaVolI} for the case $\Omega = \mathbb{R}^d$.

\begin{lemma}\label{lem:discrete_morrey}
    Let $\Omega\subseteq\R$ be a domain or $\R$ itself. For $0<p<\infty$ and $0\leq \tau < \frac{1}{p}$ let $u:=1/(\frac{1}{p}-\tau)$. Then
    $$
        \sup_{P\in\mathcal{Q}} \frac{1}{\abs{P}^\tau} \norm{f\sep L_p(P\cap \Omega)} 
        \sim \norm{f \sep \mathcal{M}^u_p(\Omega)}, \qquad f\in M(\Omega). 
    $$
\end{lemma}
\begin{proof}
    Note that the definition of $u$ ensures $p\leq u < \infty$ and $\abs{P}^{-\tau}=\abs{P}^{\frac{1}{u}-\frac{1}{p}}$. Given $P\in\mathcal{Q}$ with $P\cap\Omega\neq \emptyset$, we may choose $y\in P\cap \Omega$ and $r:=\ell(P) \sqrt{d}$ such that $P\subset B(y,r)$ while $\abs{P}\sim \abs{B(y,r)}\sim r^d$. This shows ``$\lesssim$''. Conversely, there exists $N\in\N$ such that every ball $B(y,r)$ centered in $y\in\Omega$ can be covered by at most $N$ dyadic cubes~$P_k$ with $\abs{P_k}\sim r^d$. To this end, given $y$ and $r$, choose $j\in\mathbb{Z}$ with $2^{j-1}<r\leq 2^{j}$ and let $I:=\big\{k\in\Z \,:\, \abs{2^j k - y} \leq 2\sqrt{d}\, 2^{j}\big\}$. 
    Then all $P_k:=Q_{-j,k}\in\mathcal{Q}$ with $k\in I$ satisfy $\ell(P_k)=2^j$ as well as $B(y,r) \subseteq \bigcup_{k\in I} P_k$, while $\# I$ is bounded independent of $y$ and~$r$.
    Hence,
    $$
        r^{d(\frac{1}{u}-\frac{1}{p})} \norm{f \sep L_p(B(y,r)\cap \Omega)} 
        \lesssim \sum_{k\in I} \abs{P_k}^{\frac{1}{u}-\frac{1}{p}} \norm{f \sep L_p(P_k \cap \Omega)}
        \lesssim \sup_{P\in\mathcal{Q}} \frac{1}{\abs{P}^\tau} \norm{f\sep L_p(P\cap \Omega)}
    $$
    which shows ``$\gtrsim$'' and completes the proof.
\end{proof}

In the sequel we shall need the following analogue of \cite[Lemma 5]{HoWe23}. 
Therein, for $\varepsilon>0$ the $\varepsilon$-neighbourhood of $S\subseteq\R$ is given by $S_\varepsilon:=\{x\in\R \,:\, \dist(x,S)<\varepsilon\}$.
\begin{lemma}\label{lem:tools}
    Let $ s \in \mathbb{R}$, $0 < p \leq u < \infty$, $0 < q \leq \infty$, $0\leq \tau < \frac{1}{p}$, and let $\Omega,\Omega_1,\Omega_2\subseteq\R$ be arbitrary domains.
    \begin{enumerate}
        \item For all $G\in \mathcal{D}(\R)$ there exist $c_G, c_G'>0$ such that
        \begin{align*}
            \norm{ G|_\Omega \cdot f \sep \mathcal{N}^{s}_{u,p,q}(\Omega)} &\leq c_G \norm{f \sep \mathcal{N}^{s}_{u,p,q}(\Omega)}, \qquad f\in \mathcal{N}^{s}_{u,p,q}(\Omega),\\
            \norm{ G|_\Omega \cdot f \sep B^{s,\tau}_{p,q}(\Omega)} &\leq c_G' \norm{f \sep B^{s,\tau}_{p,q}(\Omega)}, \;\qquad\; f\in B^{s,\tau}_{p,q}(\Omega).
        \end{align*}
        
        \item Every affine-linear diffeomorphism $\Phi\colon \R\to\R$ yields isomorphisms $T_\Phi\colon f\mapsto f\circ\Phi$ of $\mathcal{N}^{s}_{u,p,q}(\Phi(\Omega))$ onto $\mathcal{N}^{s}_{u,p,q}(\Omega)$ and of $B^{s,\tau}_{p,q}(\Phi(\Omega))$ onto $B^{s,\tau}_{p,q}(\Omega)$.

        \item Let $S\subseteq\R$ be such that $S_\varepsilon\cap \Omega_1= S_\varepsilon\cap \Omega_2$. 
        Then for all $F\in \mathcal{S}'(\R)$ with $\supp F \subseteq S$
        \begin{align*}
            \norm{F|_{\Omega_1} \sep \mathcal{N}^{s}_{u,p,q}(\Omega_1)} &\sim \norm{F|_{\Omega_2} \sep \mathcal{N}^{s}_{u,p,q}(\Omega_2)}, \\
            \norm{F|_{\Omega_1} \sep B^{s,\tau}_{p,q}(\Omega_1)} &\sim \norm{F|_{\Omega_2} \sep B^{s,\tau}_{p,q}(\Omega_2)}.
        \end{align*}
    \end{enumerate}
\end{lemma}
\begin{proof}
    Here we can argue exactly as in the proof of \cite[Lemma 5(i)--(iii)]{HoWe23} with $\mathcal{E}^{s}_{u,p,q}$ being replaced by $\mathcal{N}^{s}_{u,p,q}$ resp.\ $B^{s,\tau}_{p,q}$. For the needed assertions in spaces on $\R$, we refer to \cite[Theorem 2.6]{HaSkSM} and \cite[Theorem 6.1]{ysy} for (i) and to \cite[Theorem 1.7]{Saw10} and \cite[Theorem~6.7]{ysy} for (ii). 
\end{proof}

Nearly every domain that appears in this paper is a so-called Lipschitz domain. For the definition we follow Stein \cite[VI.3.2]{Stein}.

\begin{defi}[Lipschitz Domain]
\label{lipdo}
By a Lipschitz domain in $\R$, $d\in\N$, we mean either a special or a bounded Lipschitz
domain. 
\begin{itemize}
\item[(i)]
A special Lipschitz domain is an open set $\Omega \subset
\mathbb{R}^{d} $ lying above the graph of a Lipschitz function $\omega:\ \mathbb{R}^{d-1} \to \mathbb{R}$, namely $ \Omega:=\big\{(x',x_d)\in \mathbb{R}^{(d-1)+1}  :\ x_d > \omega (x')\big\} $, where $\omega$ satisfies that for all $x', y'\in \mathbb{R}^{d-1}$ there holds $\abs{\omega(x')-\omega(y')}\le c \abs{x'-y'}$ with some constant~$c>0$ independent of $x'$ and $y'$.

\item[(ii)]
A bounded  Lipschitz domain is a bounded domain $\Omega \subset \mathbb{R}^{d} $
whose boundary $\partial \Omega$ can be covered by a finite number of open balls $B_k$
such that for each $k \in \mathbb{N} $ (after a suitable rotation) $\partial\Omega\cap B_k$
is a part of the graph of a Lipschitz function.
\end{itemize}
For simplicity we shall use the convention that a bounded Lipschitz domain in $ \mathbb{R} $ is just a bounded interval.
\end{defi}

\section{Quasi-Norms of Interest}\label{sec_3_norms}
In this paper it is our main goal to describe Besov-Morrey and Besov-type spaces in terms of equivalent quasi-norms based on higher order differences and local oscillations. 
To do so, in what follows we introduce some additional notation. Let $d\in\N$. 
Given a function $ f \colon \mathbb{R}^d \rightarrow \mathbb{C}$ its first order difference of step length $h\in\R$ is given by the mapping $x\mapsto \Delta_{h}^{1}f (x) := f ( x + h ) - f (x)$ on $\R$. Moreover, for $N\in\N\setminus\{1\}$ we put
\begin{align}\label{eq_diff_def}
    \Delta_{h}^{N}f (x) := \left  (\Delta_{h} ^1 \left ( \Delta_{h} ^{N-1}f \right  )\right ) (x) \, , \qquad x \in \R\, . 
\end{align}
Then it is easily seen that for every $N\in\N$ and $h\in\R$ there holds
\begin{align}\label{eq:Delta}
    \Delta_{h}^{N}f (x) = \sum_{k=0}^N (-1)^{N-k} \binom{N}{k} \,f(x+kh), \qquad x\in\R.
\end{align}
In addition, for $N\in\N$ let $\mathcal{P}_{N-1}$ denote the set of all $d$-variate polynomials with total degree strictly less than $N$ and let $\Omega \subseteq \R$ be a domain or $\R$ itself.
Then for $0<v\leq\infty$ the local oscillation with radius $t>0$ of $f \in L_v^{\loc}(\Omega)$ in $x\in\Omega$ is given by
\begin{align*}
    \osc_{v,\Omega}^{N-1}f(x,t) 
    := \begin{cases}
        \displaystyle \inf_{\pi \in \mathcal{P}_{N-1}} \Big( t^{-d} \int_{B(x,t) \cap \Omega} \abs{f(y) - \pi (y)}^{v} \d y \Big)^{\frac{1}{v}} & \text{if}\quad v<\infty,\\
        \displaystyle\inf_{\pi \in\mathcal{P}_{N-1}} \esssup_{y\in B(x,t)\cap \Omega} \abs{f(y)-\pi (y)} & \text{if}\quad v=\infty.
    \end{cases}
\end{align*}
If $\Omega=\R$, we simply write $\osc_{v}^{N-1}f:=\osc_{v,\R}^{N-1}f$. Now we are well-prepared to introduce the quasi-(semi-)norms which will turn out to characterize Besov-Morrey spaces $\mathcal{N}^{s}_{u,p,q}(\Omega)$ and Besov-type spaces $B^{s,\tau}_{p,q}(\Omega)$, respectively. 
They naturally generalize analogous quantities for Besov spaces $B^s_{p,q}(\Omega)$; cf.~\cite[Section 5.2]{Tr92}.

\begin{defi}[New quasi-(semi-)norms on $\mathcal{N}^{s}_{u,p,q}$]
\label{defi:norms_N}
    For $d\in\N$ let $\Omega \subseteq \R$ be a domain or~$\R$ itself.
    Then for $N\in\N$, $0<p\leq u < \infty$, $0 < q,T,v \leq \infty$, $0<R<\infty$, and $f\in L_{\max\{1,p,v\}}^\loc(\Omega)$ we define
    \begin{align*}
        \abs{f}^{(T,v,N)}_{\osc,\Omega} 
        := \bigg( \int_{0}^{T} t^{-sq} \norm{ \osc_{v,\Omega}^{N-1} f(\cdot,t) \sep \mathcal{M}^{u}_{p}(\Omega)}^{q} \frac{\d t}{t} \bigg)^{\frac{1}{q}}
    \end{align*}
    (if $\Omega=\R$, we simply write $\abs{f}^{(T,v,N)}_{\osc}:=\abs{f}^{(T,v,N)}_{\osc,\R}$)
    as well as
    \begin{align*}
	   \norm{f \sep \mathcal{N}^{s}_{u,p,q}(\Omega)}^{(T,v,N)}_\osc 
	   &:= \norm{ f \sep \mathcal{M}^{u}_{p}(\Omega)} + \abs{f}^{(T,v,N)}_{\osc,\Omega},\\
	   \norm{f \sep \mathcal{N}^{s}_{u,p,q}(\Omega)}^{(R,T,v,N)}_\osc
	   &:= \norm{\Big( \int_{B(\,\cdot\,,R)\cap \Omega} \abs{f(y)}^{v} \d y \Big )^{\frac{1}{v}} \sep \mathcal{M}^{u}_{p}(\Omega)} + \abs{f}^{(T,v,N)}_{\osc,\Omega}.
    \end{align*}
    Likewise, using $V^{N}(x,t) := \{ h \in \mathbb{R}^d \,:\, \abs{h} < t \ \text{and} \ x + \ell h \in \Omega \ \text{for} \ 0 \leq \ell \leq N \}$, we let
    \begin{align*}
        \abs{f}^{(T,v,N)}_{\Delta,\Omega} 
        := \bigg( \int_{0}^{T} t^{-sq} \norm{ \Big( t^{-d} \int_{V^{N}(\,\cdot\,,t)} \abs{\Delta^{N}_{h}f(\cdot)}^v \d h \Big)^{\frac{1}{v}} \sep \mathcal{M}^{u}_{p}(\Omega) }^q \frac{\d t}{t} \bigg)^{\frac{1}{q}},
    \end{align*}    ($\abs{f}^{(T,v,N)}_{\Delta}:=\abs{f}^{(T,v,N)}_{\Delta,\R}$)
    and define $\norm{f \sep \mathcal{N}^{s}_{u,p,q}(\Omega)}^{(T,v,N)}_\Delta$ as well as $\norm{f \sep \mathcal{N}^{s}_{u,p,q}(\Omega)}^{(R,T,v,N)}_\Delta$ correspondingly.
    In any case, for $ q = \infty $ and/or $ v = \infty$, the usual modifications are made.
\end{defi}

\begin{defi}[New quasi-(semi-)norms on $B^{s,\tau}_{p,q}$]
\label{defi:norms_B}
    For $d\in\N$ let $\Omega \subseteq \R$ be a domain or~$\R$ itself.
    Then for $N\in\N$, $0<p<\infty$, $0 \leq \tau < \frac{1}{p}$, $0 < q,T,v \leq \infty$, $0<R<\infty$, and $f\in L_{\max\{1,p,v\}}^\loc(\Omega)$ we define
    \begin{align*}
        \abs{f}^{(T,v,N)}_{\osc,\tau,\Omega} 
        := \sup_{P\in\mathcal{Q}} \frac{1}{\abs{P}^\tau} \left(\int_{0}^T t^{-sq} \norm{ \osc_{v,\Omega}^{N-1} f(\cdot, t) \sep L_p(P\cap \Omega) }^{q} \frac{\d t}{t} \right)^{\frac{1}{q}}
    \end{align*}
    (if $\Omega=\R$, we simply write $\abs{f}^{(T,v,N)}_{\osc,\tau}:=\abs{f}^{(T,v,N)}_{\osc,\tau,\R}$),
    as well as
    \begin{align*}
	   \norm{f \sep B^{s,\tau}_{p,q}(\Omega)}^{(T,v,N)}_\osc 
	   &:= \sup_{P\in\mathcal{Q}} \frac{1}{\abs{P}^\tau} \norm{ f \sep L_p(P\cap \Omega)} + \abs{f}^{(T,v,N)}_{\osc,\tau,\Omega},\\
	   \norm{f \sep B^{s,\tau}_{p,q}(\Omega)}^{(R,T,v,N)}_\osc
	   &:= \sup_{P\in\mathcal{Q}} \frac{1}{\abs{P}^\tau} \norm{ \Big( \int_{B(\cdot,R)\cap\Omega} \abs{f(y)}^v\d y \Big)^{\frac{1}{v}} \sep L_p(P \cap \Omega )} + \abs{f}^{(T,v,N)}_{\osc,\tau,\Omega}.
    \end{align*}
    Likewise, let
    \begin{align*}
        \abs{f}^{(T,v,N)}_{\Delta,\tau,\Omega} 
        := \sup_{P\in\mathcal{Q}} \frac{1}{\abs{P}^\tau} \bigg( \int_{0}^{T} t^{-sq} \norm{ \Big( t^{-d} \int_{V^{N}(\,\cdot\,,t)} \abs{\Delta^{N}_{h}f(\cdot)}^v \d h \Big)^{\frac{1}{v}} \sep L_{p}(P\cap \Omega) }^q \frac{\d t}{t} \bigg)^{\frac{1}{q}},
    \end{align*}    ($\abs{f}^{(T,v,N)}_{\Delta,\tau}:=\abs{f}^{(T,v,N)}_{\Delta,\tau,\R}$)
    and define $\norm{f \sep B^{s,\tau}_{p,q}(\Omega)}^{(T,v,N)}_\Delta$ as well as $\norm{f \sep B^{s,\tau}_{p,q}(\Omega)}^{(R,T,v,N)}_\Delta$ correspondingly.
    In any case, for $ q = \infty $ and/or $ v = \infty$, the usual modifications are made.
\end{defi}

\section{Characterizations on \texorpdfstring{$\mathbb{R}^d$}{Rd}}\label{sec_char_Rd}

\subsection{Characterizations via Local Oscillations for \texorpdfstring{$\mathcal{N}^{s}_{u,p,q}(\mathbb{R}^d)$}{Nsupq(Rd)}}\label{sec_char_Rd_1}

In this first subsection our main goal is to describe the Besov-Morrey spaces $\mathcal{N}^{s}_{u,p,q}(\mathbb{R}^{d})$ in terms of local oscillations. 
In the literature there already exist some first characterizations of this type, e.g., based on the theory developed by Hedberg and Netrusov~\cite{HN}.
Similarly, we obtain the following result.

\begin{theorem}\label{thm_HN_N_osc_1}
    Let $d,N \in \mathbb{N} $, $ 0 < p \leq u < \infty$, and $ 0 < q,v \leq \infty $. 
    Further let $s\in\mathbb{R}$ with
    \begin{align*}
		d\, \max \left\{  0, \frac{1}{p} - 1, \frac{1}{p} - \frac{1}{v} \right\} < s < N.
    \end{align*}
    Then a function $ f\in L_{p}^{\loc}(\R)$ belongs to $ \mathcal{N}^{s}_{u,p,q}(\R)$ if and only if $ f \in L_{v}^{\loc}(\R)$ and
    \begin{align*}
        \norm{ f \sep \mathcal{N}^{s}_{u,p,q}(\R)}_{\osc}^{(\clubsuit)} &
        :=   \Bigg ( \norm{ \Big ( \int_{B(\,\cdot\,,1)}\vert f(y) \vert^{v} \d y \Big )^{\frac{1}{v}}   \sep  \mathcal{M}^{u}_{p}( \mathbb{R}^d) }^{q} \\
        &+ \sum_{j = 1}^{\infty} 2^{jq(s + \frac{d}{v})} \norm{ \inf_{\pi \in \mathcal{P}_{N-1}} \!\!\Big( \int_{B(\,\cdot\,,2^{-j}) } \abs{f(y) - \pi (y)}^v\d y \Big)^{\frac{1}{v}} \sep  \mathcal{M}^{u}_{p}( \mathbb{R}^d) }^{q} \Bigg )^{\frac{1}{q}}
    \end{align*}
    is finite (with usual modifications if $ q = \infty $ and/or $ v = \infty $). 
    Moreover, the quasi-norms $\norm{\cdot \sep \mathcal{N}^{s}_{u,p,q}(\R)}$ and $\norm{\cdot \sep \mathcal{N}^{s}_{u,p,q}(\R) }_{\osc}^{(\clubsuit)}$  are equivalent on $L_{p}^{\loc}(\R)$.
\end{theorem}

\begin{proof}
 In order to prove the equivalence, we can use the theory developed in \cite[Sections~1.1--1.3]{HN}. It is not difficult to see that the Besov-Morrey spaces $\mathcal{N}^s_{u,p,q}(\mathbb{R}^d)$ fit into the setting described in \cite{HN}. 
 A proof for that can be found in \cite[Proposition~1]{HoN}, see also \cite[Proposition~1]{HoN2}.
A similar observation was made also in \cite[Section~4.5.1]{ysy}. Now the above equivalence is a simple consequence of \cite[Theorem~1.1.14(iii)]{HN}. We refer to the proof of \cite[Proposition~1]{HoN} for some more details.
\end{proof}

Notice that a similar result can also be found in \cite[Corollary~4.15]{ysy}. 
In what follows, we will use \autoref{thm_HN_N_osc_1} as a starting point to prove advanced characterizations of Besov-Morrey spaces in terms of local oscillations.

\begin{theorem}[Oscillations in Besov-Morrey Spaces]\label{thm_osc_Rda=2}
    Let $d,N \in \mathbb{N}$, $s\in\mathbb{R}$ as well as $0< p \leq u < \infty$, $0 < q,T,v \leq \infty$, and $0<R<\infty$ be such that
    \begin{align*}
		d\, \max \left\{  0, \frac{1}{p} - 1, \frac{1}{p} - \frac{1}{v} \right\} < s < N.
    \end{align*}
    Then 
    \begin{align*}
        \mathcal{N}^{s}_{u,p,q}(\R) 
        &= \left\{ f\in L_{\max\{1,p,v\}}^\loc(\R) \,:\, \norm{f \sep \mathcal{N}^{s}_{u,p,q}(\R)}^{(T,v,N)}_\osc < \infty \right\} \\
        &= \left\{ f\in L_{\max\{1,p,v\}}^\loc(\R) \,:\, \norm{f \sep \mathcal{N}^{s}_{u,p,q}(\R)}^{(R,T,v,N)}_\osc < \infty \right\}
    \end{align*}
    and the quasi-norms $\norm{\,\cdot \sep \mathcal{N}^{s}_{u,p,q}(\mathbb{R}^{d})}$, $\norm{\,\cdot \sep \mathcal{N}^{s}_{u,p,q}(\mathbb{R}^{d})}^{(T,v,N)}_\osc$, and $\norm{\,\cdot \sep \mathcal{N}^{s}_{u,p,q}(\R)}^{(R,T,v,N)}_\osc$ are mutually equivalent on $L_{\max\{1,p,v\}}^\loc(\R)$. 
\end{theorem}
\begin{proof}
    Throughout this proof w.l.o.g.\ we may assume that $q,v<\infty$. In addition, due to \cite[Lemma~4]{HoWe23} it is sufficient to consider $R=1$.
    As a preparation, we further note that
    \begin{align*}
        &\norm{\inf_{\pi \in \mathcal{P}_{N-1}} \!\Big( \int_{B(\,\cdot\,,2^{-j}) } \abs{f(y) - \pi (y)}^v\d y \Big)^{\frac{1}{v}} \sep \mathcal{M}^{u}_{p}( \mathbb{R}^d)} = 2^{-j \frac{d}{v}} \norm{ \osc_v^{N-1} f(\cdot,2^{-j}) \sep \mathcal{M}^{u}_{p}( \mathbb{R}^d)}.
    \end{align*}
    
    \emph{Step 1. }
    We first show the assertion for $\norm{\,\cdot \sep \mathcal{N}^{s}_{u,p,q}(\R)}^{(1,T,v,N)}_\osc$. To this end, in view of \autoref{thm_HN_N_osc_1}, it suffices to prove that
    \begin{align}\label{eq:proof_normeq}
		\norm{f \sep \mathcal{N}^{s}_{u,p,q}(\R)}^{(1,T,v,N)}_\osc 
		\sim \norm{f \sep \mathcal{N}^{s}_{u,p,q}(\R)}^{(\clubsuit)}_\osc,
		\qquad f\in L_{\max\{1,p,v\}}^\loc(\R).
    \end{align}

    \emph{Substep 1a. } Let us start by proving \eqref{eq:proof_normeq} in the special case $0<T\leq \frac{1}{2}$.
    We have
    \begin{align*}
		\int_{0}^T t^{-sq} \norm{ \osc_v^{N-1} f(\cdot, t) \sep \mathcal{M}^{u}_{p}( \mathbb{R}^d)}^{q} \frac{\d t}{t}
		&\leq \sum_{j=1}^{\infty} \int_{2^{-j-1}}^{2^{-j}} t^{-sq} \norm{ \osc_v^{N-1} f(\cdot,t) \sep \mathcal{M}^{u}_{p}( \mathbb{R}^d)}^{q} \frac{\d t}{t} \\
        &\lesssim \sum_{j=1}^{\infty} \norm{ \osc_v^{N-1} f(\cdot,2^{-j}) \sep \mathcal{M}^{u}_{p}( \mathbb{R}^d)}^{q} \int_{2^{-j-1}}^{2^{-j}} t^{-sq-1} \d t \\
        &\sim \sum_{j=1}^{\infty} 2^{jsq} \norm{ \osc_v^{N-1} f(\cdot,2^{-j}) \sep \mathcal{M}^{u}_{p}( \mathbb{R}^d)}^{q}
    \end{align*}
    since $\osc_{v}^{N-1} f(x,t) \lesssim \osc_{v}^{N-1} f(x,2^{-j})$ for $2^{-j-1}<t < 2^{-j}$; see, e.g.\ \cite[p.124]{ysy} or \cite[p.10]{Yab}. This shows
    \begin{align*}
        \abs{f}^{(T,v,N)}_{\osc} 
        &\lesssim \left [ \norm{ \Big( \int_{B(\cdot,1)} \abs{f(y)}^v\d y \Big)^{\frac{1}{v}} \sep \mathcal{M}^{u}_{p}(\mathbb{R}^d)}^{q} + \sum_{j=1}^{\infty} 2^{jsq} \norm{ \osc_v^{N-1} f(\cdot,2^{-j}) \sep \mathcal{M}^{u}_{p}( \mathbb{R}^d)}^{q} \right ]^{\frac{1}{q}} \\
        &=\norm{f \sep \mathcal{N}^{s}_{u,p,q}(\R)}_\osc^{(\clubsuit)},
    \end{align*}
    i.e., there holds ``$\lesssim$'' in \eqref{eq:proof_normeq}. In order to show the converse estimate, we observe that for $J\in\N$ with $2^{-J}\leq T$ there holds
    \begin{align*}
        \sum_{j=1}^{J} 2^{jsq} \norm{ \osc_v^{N-1} f(\cdot,2^{-j}) \sep \mathcal{M}^{u}_{p}( \mathbb{R}^d)}^{q}
        &\leq \sum_{j=1}^{J} 2^{j(s+\frac{d}{v})q} \norm{ \Big( \int_{B(\,\cdot\,,2^{-j}) } \abs{f(y)}^v\d y \Big)^{\frac{1}{v}} \sep \mathcal{M}^{u}_{p}(\mathbb{R}^d)}^{q} \\
        &\lesssim \norm{ \Big( \int_{B(\,\cdot\,,1) } \abs{f(y)}^v\d y \Big)^{\frac{1}{v}} \sep \mathcal{M}^{u}_{p}(\mathbb{R}^d)}^{q}
    \end{align*}
    while, on the other hand,
    \begin{align*}
       & \sum_{j=J+1}^{\infty} 2^{jsq} \norm{ \osc_v^{N-1} f(\cdot,2^{-j}) \sep \mathcal{M}^{u}_{p}( \mathbb{R}^d)}^{q} \\
        & \qquad \qquad \sim \sum_{j=J+1}^{\infty} \int_{2^{-j}}^{2^{-j+1}} t^{-sq-1} \d t \norm{ \osc_v^{N-1} f(\cdot,2^{-j}) \sep \mathcal{M}^{u}_{p}( \mathbb{R}^d)}^{q} \\
        & \qquad \qquad  \lesssim \sum_{j=J+1}^{\infty} \int_{2^{-j}}^{2^{-(j-1)}} t^{-sq} \norm{ \osc_v^{N-1} f(\cdot,t) \sep \mathcal{M}^{u}_{p}( \mathbb{R}^d)}^{q} \frac{\d t}{t} \\
        & \qquad \qquad  \leq \int_{0}^T t^{-sq} \norm{ \osc_v^{N-1} f(\cdot,t) \sep \mathcal{M}^{u}_{p}( \mathbb{R}^d)}^{q} \frac{\d t}{t}.
    \end{align*}
    Consequently,
    \begin{align*}
        \sum_{j=1}^{\infty} 2^{jsq} \norm{ \osc_v^{N-1} f(\cdot,2^{-j}) \sep \mathcal{M}^{u}_{p}( \mathbb{R}^d)}^{q}
        \lesssim \norm{ \Big( \int_{B(\,\cdot\,,1) } \abs{f(y)}^v\d y \Big)^{\frac{1}{v}} \sep \mathcal{M}^{u}_{p}(\mathbb{R}^d)}^q + \big( \abs{f}^{(T,v,N)}_{\osc} \big)^q
    \end{align*}
    which implies ``$\gtrsim$'' in \eqref{eq:proof_normeq} such that the proof of \eqref{eq:proof_normeq} for $T\leq \frac{1}{2}$ is complete.
    
    \emph{Substep 1b. } 
     Due to the monotonicity of $\norm{\,\cdot \sep \mathcal{N}^{s}_{u,p,q}(\R)}^{(1,T,v,N)}_\osc$ in $T$ it remains to show
    \begin{align}\label{eq:proof_normeq1b}
		\norm{f \sep \mathcal{N}^{s}_{u,p,q}(\R)}^{(1,\infty,v,N)}_\osc 
		\lesssim \norm{f \sep \mathcal{N}^{s}_{u,p,q}(\R)}^{(1,\frac{1}{2},v,N)}_\osc,
		\qquad f\in L_{\max\{1,p,v\}}^\loc(\R),
    \end{align}
    in order to complete Step 1. To do so, we note that similar to Substep 1a we find
    \begin{align*}
		\int_{\frac{1}{2}}^{\infty} t^{-sq} \norm{ \osc_v^{N-1} f(\cdot,t) \sep \mathcal{M}^{u}_{p}( \mathbb{R}^d)}^{q} \frac{\d t}{t} 
		&= \sum_{j=0}^{\infty} \int_{2^{j-1}}^{2^{j}} t^{-sq} \norm{ \osc_v^{N-1} f(\cdot,t) \sep \mathcal{M}^{u}_{p}( \mathbb{R}^d)}^{q} \frac{\d t}{t} \\
		&\lesssim \sum_{j=0}^{\infty} 2^{-jsq} \norm{ \osc_v^{N-1} f(\cdot,2^j) \sep \mathcal{M}^{u}_{p}( \mathbb{R}^d)}^{q}.
    \end{align*}
    Furthermore note that independent of $x\in\R$ there exist appropriate displacement vectors $w_k\in\Z$, $k\in\N$, such that with $K_j:=(2^{j+1}+1)^d\sim 2^{jd}$ there holds
    \begin{equation}\label{eq_displace_vec1}
	    B(x,2^j) \subset \bigcup_{k=1}^{K_j} B(x+w_k,1), \qquad j\in\N_0.
    \end{equation}
    Hence, choosing $\pi :\equiv 0$ for the polynomial, we have
    \begin{align*}
        \osc_{v}^{N-1} f(x,2^j) 
	&\leq \left( 2^{-jd} \int_{B(x,2^j)} \abs{f(y)}^v \d y \right)^{\frac{1}{v}} \\
	&\leq \left( \sum_{k=1}^{K_j} 2^{-jd} \int_{B(x+w_k,1)} \abs{f(y)}^v \d y \right)^{\frac{1}{v}}.
    \end{align*}
    Setting $\mu:=\min\{p,v\}$, we can use $\mu/v \leq 1$ to conclude
    \begin{align*}
        \norm{ \osc_v^{N-1} f(\cdot,2^j) \sep \mathcal{M}^{u}_{p}( \mathbb{R}^d)}
        &\leq 2^{-j \frac{d}{v}} \norm{ \left( \sum_{k=1}^{K_j} \int_{B(\,\cdot\,+w_k,1)} \abs{f(y)}^v \d y \right)^{\frac{\mu}{v}} \sep \mathcal{M}^{\frac{u}{\mu}}_{\frac{p}{\mu}}( \mathbb{R}^d)}^{\frac{1}{\mu}} \\
        &\leq 2^{-j \frac{d}{v}} \left( \norm{ \sum_{k=1}^{K_j} \left(\int_{B(\,\cdot\,+w_k,1)} \abs{f(y)}^v \d y \right)^{\frac{\mu}{v}} \sep \mathcal{M}^{\frac{u}{\mu}}_{\frac{p}{\mu}}( \mathbb{R}^d)} \right)^{\frac{1}{\mu}} .
    \end{align*}
    Due to $p/\mu \geq 1$ the space $\mathcal{M}^{\frac{u}{\mu}}_{\frac{p}{\mu}}( \mathbb{R}^d)$ is a translation invariant Banach space which together with $K_j\sim 2^{jd}$ further yields
    \begin{align*}
        \norm{ \osc_v^{N-1} f(\cdot,2^j) \sep \mathcal{M}^{u}_{p}( \mathbb{R}^d)}
        &\leq 2^{-j \frac{d}{v}} \left( \sum_{k=1}^{K_j} \norm{ \left(\int_{B(\,\cdot\,+w_k,1)} \abs{f(y)}^v \d y \right)^{\frac{\mu}{v}} \sep \mathcal{M}^{\frac{u}{\mu}}_{\frac{p}{\mu}}( \mathbb{R}^d)} \right)^{\frac{1}{\mu}} \\
        &= 2^{-j \frac{d}{v}} \, K_j^{\frac{1}{\mu}} \norm{ \left(\int_{B(\,\cdot\,,1)} \abs{f(y)}^v \d y \right)^{\frac{\mu}{v}} \sep \mathcal{M}^{\frac{u}{\mu}}_{\frac{p}{\mu}}( \mathbb{R}^d)}^{\frac{1}{\mu}} \\
        &\sim 2^{jd [\frac{1}{\mu}-\frac{1}{v}]} \norm{ \left(\int_{B(\,\cdot\,,1)} \abs{f(y)}^v \d y \right)^{\frac{1}{v}} \sep \mathcal{M}^{u}_{p}( \mathbb{R}^d)}.
    \end{align*}
    Since we assumed that $ d [ \frac{1}{\mu} - \frac{1}{v} ]  = d \,\max\!\big\{ 0, \frac{1}{p} - \frac{1}{v}\big\} < s $ we can conclude
    \begin{align*}
		& \int_{\frac{1}{2}}^{\infty} t^{-sq} \norm{ \osc_v^{N-1} f(\cdot,t) \sep \mathcal{M}^{u}_{p}( \mathbb{R}^d)}^{q} \frac{\d t}{t} \\
		& \qquad \qquad  \lesssim \sum_{j=0}^{\infty} 2^{-jsq} \norm{ \osc_v^{N-1} f(\cdot,2^j) \sep \mathcal{M}^{u}_{p}( \mathbb{R}^d)}^{q} \\
        & \qquad \qquad   \lesssim \sum_{j=0}^{\infty} 2^{j (d [\frac{1}{\mu}-\frac{1}{v}]-s)q} \norm{ \left(\int_{B(\,\cdot\,,1)} \abs{f(y)}^v \d y \right)^{\frac{1}{v}} \sep \mathcal{M}^{u}_{p}( \mathbb{R}^d)}^q \\
        & \qquad \qquad   \sim \norm{ \left(\int_{B(\,\cdot\,,1)} \abs{f(y)}^v \d y \right)^{\frac{1}{v}} \sep \mathcal{M}^{u}_{p}( \mathbb{R}^d)}^q.
    \end{align*}
    Together this shows 
    $$
        \abs{f}^{(\infty,v,N)}_\osc 
        \lesssim \abs{f}^{(\frac{1}{2},v,N)}_\osc + \norm{ \left(\int_{B(\,\cdot\,,1)} \abs{f(y)}^v \d y \right)^{\frac{1}{v}} \sep \mathcal{M}^{u}_{p}( \mathbb{R}^d)}
        = \norm{f \sep \mathcal{N}^{s}_{u,p,q}(\R)}^{(1,\frac{1}{2},v,N)}_\osc
    $$ 
   which implies \eqref{eq:proof_normeq1b} and hence \eqref{eq:proof_normeq}.
    So Step 1 of this proof is complete.

    \emph{Step 2. } 
    We use the findings from the previous step to show that also
    \begin{align*}
		\norm{f \sep \mathcal{N}^{s}_{u,p,q}(\mathbb{R}^{d})} 
		\sim \norm{f \sep \mathcal{N}^{s}_{u,p,q}(\R)}^{(T,v,N)}_\osc,
		\qquad f\in L_{\max\{1,p,v\}}^\loc(\R).
    \end{align*}

    \emph{Substep 2a (Lower Bound). } 
    Step 1 ensures that
    $$
	    \norm{f \sep \mathcal{N}^{s}_{u,p,q}(\mathbb{R}^{d})} 
		\sim \norm{f \sep \mathcal{N}^{s}_{u,p,q}(\R)}^{(1,T,v,N)}_\osc
		\geq \abs{f}^{(T,v,N)}_\osc
    $$
    while \autoref{l_el_em_EN} yields
    $$
        \norm{f \sep \mathcal{N}^{s}_{u,p,q}(\mathbb{R}^{d})} 
		\gtrsim \norm{f \sep \mathcal{M}^{u}_{p}(\R)}.
    $$

    \emph{Substep 2b (Upper Bound). } 
    To complete the proof we  distinguish two cases. If $p<v$, we have
    $$
	    d\, \max\! \left\{ 0, \frac{1}{p} - 1 \right \} \leq d\, \max\! \left \{ 0, \frac{1}{p} - 1, \frac{1}{p} - \frac{1}{v} \right\}.
    $$
    In other words, we can apply Step 1 (with $v:=p$) to obtain
    \begin{align*}
	    \norm{f \sep \mathcal{N}^{s}_{u,p,q}(\mathbb{R}^{d})}
	    &\lesssim \norm{f \sep \mathcal{N}^{s}_{u,p,q}(\R)}^{(1,T,p,N)}_\osc \\
	    &=\norm{\Big( \int_{B(\,\cdot\,,1)} \abs{f(y)}^p \d y \Big )^{\frac{1}{p}} \sep \mathcal{M}^{u}_{p}(\R)} + \abs{f}^{(T,p,N)}_\osc.
    \end{align*}
    Then H\"older's inequality yields $\osc_p^{N-1} f(\cdot,t)\lesssim \osc_v^{N-1} f(\cdot,t)$ on $\R$ such that the properties of Morrey spaces allow to upper bound the second summand (up to constants) by~$\abs{f}^{(T,v,N)}_\osc$.
    In order to estimate the first term as well, we use \cite[Lemma~3(vii)]{HoWe23} to see that
    \begin{align}\label{eq:avg_f}
        \norm{\Big( \int_{B(\,\cdot\,,1)} \abs{f(y)}^p \d y \Big )^{\frac{1}{p}} \sep \mathcal{M}^{u}_{p}(\R)}
        &\lesssim \norm{f \sep \mathcal{M}^{u}_{p}(\R)}.
    \end{align}
    If otherwise $v\leq p$, Step 1 yields
    \begin{align*}
	    \norm{f \sep \mathcal{N}^{s}_{u,p,q}(\mathbb{R}^{d})}
	    &\lesssim \norm{\Big( \int_{B(\,\cdot\,,1)} \abs{f(y)}^v \d y \Big )^{\frac{1}{v}} \sep \mathcal{M}^{u}_{p}(\R)} + \abs{f}^{(T,v,N)}_\osc,
    \end{align*}
    where according to \cite[Lemma~3(vii)]{HoWe23} there holds
    \begin{align*}
	    \norm{\Big( \int_{B(\,\cdot\,,1)} \abs{f(y)}^v \d y \Big )^{\frac{1}{v}} \sep \mathcal{M}^{u}_{p}(\R)}
	    &\lesssim \norm{f \sep \mathcal{M}^{u}_{p}(\R)}.
    \end{align*}
    Hence, the proof is complete.
\end{proof}

\subsection{Difference Characterizations of \texorpdfstring{$\mathcal{N}^{s}_{u,p,q}(\mathbb{R}^d)$}{Nsupq(Rd)} }

Next, we recall and enhance characterizations in terms of higher-order differences for the spaces~$\mathcal{N}^{s}_{u,p,q}(\mathbb{R}^d)$. The following assertion slightly extends \cite[Theorem~7]{H21} and \cite[Theorem~3]{HoN},  where the case $\norm{\,\cdot \sep \mathcal{N}^{s}_{u,p,q}(\mathbb{R}^{d})}^{(T,v,N)}_\Delta$ with $T\geq 1$ already is proved. 

\begin{theorem}[Differences in Besov-Morrey Spaces]\label{thm_diff_Rda=2}
    Let $d,N \in \mathbb{N}$, $0< p \leq u < \infty$, $0 < q,T,v \leq \infty$, $0<R<\infty$, and $s\in\mathbb{R}$ satisfy
    \begin{align*}
		d\, \max \left\{  0, \frac{1}{p} - 1, \frac{1}{p} - \frac{1}{v} \right\} < s < N.
    \end{align*}
    Then 
    \begin{align*}
        \mathcal{N}^{s}_{u,p,q}(\R) 
        &= \left\{ f\in L_{\max\{1,p,v\}}^\loc(\R) \,:\, \norm{f \sep \mathcal{N}^{s}_{u,p,q}(\R)}^{(T,v,N)}_\Delta < \infty \right\} \\
        &= \left\{ f\in L_{\max\{1,p,v\}}^\loc(\R) \,:\, \norm{f \sep \mathcal{N}^{s}_{u,p,q}(\R)}^{(R,T,v,N)}_\Delta < \infty \right\}
    \end{align*}
    and the quasi-norms $\norm{\,\cdot \sep \mathcal{N}^{s}_{u,p,q}(\mathbb{R}^{d})}$, $\norm{\,\cdot \sep \mathcal{N}^{s}_{u,p,q}(\mathbb{R}^{d})}^{(T,v,N)}_\Delta$, and $\norm{\,\cdot \sep \mathcal{N}^{s}_{u,p,q}(\R)}^{(R,T,v,N)}_\Delta$ are mutually equivalent on $L_{\max\{1,p,v\}}^\loc(\R)$. 
\end{theorem}
\begin{proof}
    \emph{Step 1 (Lower Bounds). }
    From \cite[Theorem~7]{H21} and \cite[Theorem 3]{HoN} it particularly follows that
    $$
        \norm{f \sep \mathcal{N}^{s}_{u,p,q}(\mathbb{R}^{d})} \gtrsim         \abs{f}^{(\infty,v,N)}_\Delta
    $$
    which, due to monotonicity w.r.t.\ $T$ together with \autoref{thm_osc_Rda=2}, proves the needed lower bounds, namely
    $$
        \norm{f \sep \mathcal{N}^{s}_{u,p,q}(\mathbb{R}^{d})} 
        \gtrsim \norm{ f \sep \mathcal{M}^{u}_{p}(\R)} + \norm{\Big( \int_{B(\,\cdot\,,R)} \abs{f(y)}^{v} \d y \Big )^{\frac{1}{v}} \sep \mathcal{M}^{u}_{p}(\R)} + \abs{f}^{(T,v,N)}_\Delta.
    $$

    \emph{Step 2 (Upper Bounds). }
    The corresponding upper bounds can be shown as follows.
    
    \emph{Substep 2a. } Here we employ Whitney's estimate~\cite[Theorem~A.1]{HN} to derive that for all $t>0$, $N\in\N$, and $w \leq p \leq u$ there holds
    \begin{align}\label{eq:osc_diff}
        \norm{\osc_w^{N-1}f(\cdot,t) \sep \mathcal{M}^{u}_{p}(\R)} 
        \lesssim \norm{ \Big( t^{-d} \int_{\abs{h}<t} \abs{\Delta^{N}_{h}f(\cdot)}^w \d h \Big)^{\frac{1}{w}} \sep \mathcal{M}^{u}_{p}(\R) } .
    \end{align}
    To this end, let $Q(x,t) \supseteq B(x,t)$ denote the cube of side-length $a:=2t$ centered at $x\in\R$. Then Whitney's estimate implies
    \begin{align*}
        \osc_w^{N-1}f(x,t)
        &\leq \inf_{\pi\in\mathcal{P}_{N-1}} \left( t^{-d} \int_{Q(x,t)} \abs{f(y)-\pi(y)}^w \d y \right)^{\frac{1}{w}} \\
        &\lesssim \left( t^{-d} \, (2t)^{-d} \int_{B(0,t)} \int_{Q(x,t)} \abs{\Delta^{N}_{h}f(y)}^w \d y \d h\right)^{\frac{1}{w}} \\
        &\sim \left( t^{-d} \int_{Q(x,t)} F_t(y)^w \d y\right)^{\frac{1}{w}},
%        &\sim \left( t^{-d} \int_{Q(x,t)} t^{-d}  \int_{B(0,t)} \abs{\Delta^{N}_{h}f(y)}^w \d h \d y\right)^{\frac{1}{w}}.
    \end{align*}
    where we set 
%    Next, for fixed $t>0$ we define 
    $$
        F_t(y) := \left( t^{-d}  \int_{\abs{h}<t} \abs{\Delta^{N}_{h}f(y)}^w \d h \right)^{\frac{1}{w}}, \qquad t>0,\, y\in\R.
    $$
%    where due $\Omega = \mathbb{R}^d$ we have that $V^N(y,t)=B(0,t)$. 
    Then $R:=t\sqrt{d}$ shows
    $$
        \osc_w^{N-1}f(x,t) 
        \lesssim \left( t^{-d} \int_{Q(x,t)} \abs{F_t(y)}^w \d y\right)^{\frac{1}{w}}
        \lesssim \left( R^{-d} \int_{B(x,R)} \abs{F_t(y)}^w \d y\right)^{\frac{1}{w}}
    $$
    such that we can follow the lines of the proof of \cite[Lemma~3(vii)]{HoWe23} to obtain
    \begin{align*}
        \norm{\osc_w^{N-1}f(\cdot,t) \sep \mathcal{M}^{u}_{p}(\R)}
        &\lesssim \norm{\left( R^{-d} \int_{B(\,\cdot\,,R)} \abs{F_t(y)}^w \d y\right)^{\frac{1}{w}} \sep \mathcal{M}^{u}_{p}(\R)} \\
        &\leq \left( R^{-d} \int_{B(0,R)} \norm{ F_t(\cdot-h) \sep \mathcal{M}^{u}_{p}(\R)}^w \d h \right)^{\frac{1}{w}} \\
        &\lesssim \norm{ F_t \sep \mathcal{M}^{u}_{p}(\R)}.
    \end{align*}
%    Now the translation invariance of $\mathcal{M}^{u}_{p}(\R)$ implies \eqref{eq:osc_diff}.

    \emph{Substep 2b. }     
    To complete the proof, we distinguish two cases. If $p<v$, we have
    $$
	    d\, \max\! \left\{ 0, \frac{1}{p} - 1 \right \} \leq d\, \max\! \left \{ 0, \frac{1}{p} - 1, \frac{1}{p} - \frac{1}{v} \right\}.
    $$
    In other words, we can apply \autoref{thm_osc_Rda=2} (with $v:=p$) to obtain
    \begin{align}
	    \norm{f \sep \mathcal{N}^{s}_{u,p,q}(\mathbb{R}^{d})}
	    &\lesssim \norm{f \sep \mathcal{N}^{s}_{u,p,q}(\R)}^{(R,T,p,N)}_\osc \nonumber\\
	    &=\norm{\Big( \int_{B(\,\cdot\,,R)} \abs{f(y)}^p \d y \Big )^{\frac{1}{p}} \sep \mathcal{M}^{u}_{p}(\R)} + \abs{f}^{(T,p,N)}_\osc. \label{eq:proof_NRd}
    \end{align}
    Further H\"older's inequality yields 
    $$
        \Big( t^{-d} \int_{V^{N}(x,t)} \abs{\Delta^{N}_{h}f(x)}^p \d h \Big)^{\frac{1}{p}}
        \lesssim \Big( t^{-d} \int_{V^{N}(x,t)} \abs{\Delta^{N}_{h}f(x)}^v \d h \Big)^{\frac{1}{v}}, \qquad x\in\R,\,t>0,
    $$
    which together with \eqref{eq:osc_diff} for $w:=p$ shows $\abs{f}^{(T,p,N)}_\osc\lesssim \abs{f}^{(T,p,N)}_\Delta\lesssim \abs{f}^{(T,v,N)}_\Delta$ for the second summand.
    Similarly, using H\"older's inequality again, the main term can be estimated by
    \begin{align*}
        \norm{\Big( \int_{B(\,\cdot\,,R)} \abs{f(y)}^p \d y \Big )^{\frac{1}{p}} \sep \mathcal{M}^{u}_{p}(\R)}
        &\lesssim \norm{\Big( \int_{B(\,\cdot\,,R)} \abs{f(y)}^v \d y \Big )^{\frac{1}{v}} \sep \mathcal{M}^{u}_{p}(\R)}.
    \end{align*}
    When $\norm{f \sep \mathcal{N}^{s}_{u,p,q}(\R)}^{(T,p,N)}_\Delta$ is concerned, we may instead use \cite[Lemma~3(vii)]{HoWe23} to bound the first summand in \eqref{eq:proof_NRd} by $\norm{f \sep \mathcal{M}^{u}_{p}(\R)}$; cf.\ \eqref{eq:avg_f}. 
    
    If otherwise $v\leq p$, then \autoref{thm_osc_Rda=2} and~\eqref{eq:osc_diff} for $w:=v$ yield
    \begin{align*}
	    \norm{f \sep \mathcal{N}^{s}_{u,p,q}(\mathbb{R}^{d})}
	    &\lesssim \norm{\Big( \int_{B(\,\cdot\,,R)} \abs{f(y)}^v \d y \Big )^{\frac{1}{v}} \sep \mathcal{M}^{u}_{p}(\R)} + \abs{f}^{(T,v,N)}_\Delta \\
      & =\norm{f \sep \mathcal{N}^{s}_{u,p,q}(\R)}^{(R,T,v,N)}_\Delta,
    \end{align*}
    where for $\norm{f \sep \mathcal{N}^{s}_{u,p,q}(\R)}^{(T,v,N)}_\Delta$ we additionally use \cite[Lemma~3(vii)]{HoWe23} to bound
    \begin{align*}
	    \norm{\Big( \int_{B(\,\cdot\,,R)} \abs{f(y)}^v \d y \Big )^{\frac{1}{v}} \sep \mathcal{M}^{u}_{p}(\R)}
	    &\lesssim \norm{f \sep \mathcal{M}^{u}_{p}(\R)}.
    \end{align*}
    Hence, the proof is complete.
\end{proof}

\subsection{Characterizations via Local Oscillations for \texorpdfstring{$B^{s,\tau}_{p,q}(\mathbb{R}^d)$}{Bstaupq(Rd)}}

In this subsection, we prove characterizations of Besov-type spaces~$B^{s,\tau}_{p,q}(\mathbb{R}^d)$ in terms of oscillations. For that purpose, again we employ the theory developed by Hedberg and Netrusov~\cite{HN} as starting point. 

\begin{theorem}
\label{thm_HN_B_osc1}
    Let $d,N\in\N$, $ 0 < p < \infty$, $0 \leq \tau < \frac{1}{p}$, $ 0 < q,v \leq \infty $, and $s\in\mathbb{R}$ with
    \begin{align*}
		d\, \max \left\{  0, \frac{1}{p} - 1, \frac{1}{p} - \frac{1}{v} \right\} < s < N.
    \end{align*}
    Then a function $ f\in L_{p}^{\loc}(\R)$ belongs to $ B^{s,\tau}_{p,q}(\R) $ if and only if  $ f \in L_{v}^{\loc}(\R)$ and
    \begin{align*}
        \norm{f \sep B^{s,\tau}_{p,q}(\R) }^{(\clubsuit )}
        &:= \sup_{P \in \mathcal{Q} } \frac{1}{\abs{P}^{\tau}} \Bigg[ \Big ( \int_{P} \Big ( \int_{B(x,1)}\vert f(y) \vert^{v} \d y \Big )^{\frac{p}{v}} \d x \Big )^{\frac{q}{p}} \\
        &\quad+ \sum_{j = 1}^{\infty} 2^{jsq}  \Big ( \int_{P} 2^{\frac{djp}{v}} \Big (     \inf_{\pi \in \mathcal{P}_{N-1}} \!\!\Big( \int_{B(x,2^{-j}) } \abs{f(y) - \pi (y)}^v\d y \Big)^{\frac{1}{v}}      \Big )^{p} \d x \Big )^{\frac{q}{p}} \Bigg]^{\frac{1}{q}}   \end{align*}
    is finite (modifications if $ q = \infty $ and/or $ v = \infty $). 
    Moreover, the quasi-norms $\norm{\cdot\sep B^{s,\tau}_{p,q}(\R)}$ and $\norm{\cdot\sep B^{s,\tau}_{p,q}(\R)}^{(\clubsuit )}$  are equivalent on $L_{p}^{\loc}(\R)$.
\end{theorem}

\begin{proof}
    In order to prove this statement, again we can use the machinery developed in \cite[Sections~1.1--1.3]{HN}. We find that Besov-type spaces $B^{s,\tau}_{p,q}(\mathbb{R}^d)$ fit into the setting described in~\cite{HN}. 
A detailed proof for that can be found in \cite[Section~3]{HoSi20}. 
Now the above equivalence is a simple consequence of \cite[Theorem~1.1.14(iii)]{HN}. Here we also can refer to the proof of \cite[Proposition~3.3]{HoSi20} for some more details.   
\end{proof}

Now we can apply \autoref{thm_HN_B_osc1} to deduce advanced characterizations for $B^{s,\tau}_{p,q}(\mathbb{R}^d)$ in terms of oscillations.

\begin{theorem}[Oscillations in Besov-Type Spaces]\label{thm_osc_B_Rda=2}
    Let $d,N \in \mathbb{N}$, $0< p < \infty$, $0 \leq \tau < \frac{1}{p}$, $0 < q,T,v \leq \infty$, $0<R<\infty$, and $s\in\mathbb{R}$ be such that
    \begin{align*}
		d\, \max \left\{  0, \frac{1}{p} - 1, \frac{1}{p} - \frac{1}{v} \right\} < s < N.
    \end{align*}
    Then 
    \begin{align*}
        B^{s,\tau}_{p,q}(\R) 
        &= \left\{ f\in L_{\max\{1,p,v\}}^\loc(\R) \,:\, \norm{f \sep B^{s,\tau}_{p,q}(\R) }^{(T,v,N)}_\osc < \infty \right\} \\
        &= \left\{ f\in L_{\max\{1,p,v\}}^\loc(\R) \,:\, \norm{f \sep B^{s,\tau}_{p,q}(\R) }^{(R,T,v,N)}_\osc < \infty \right\}
    \end{align*}
    and the quasi-norms $\norm{\,\cdot \sep B^{s,\tau}_{p,q}(\R)}$, $\norm{\,\cdot \sep B^{s,\tau}_{p,q}(\R)}^{(T,v,N)}_\osc$, and $\norm{\,\cdot \sep B^{s,\tau}_{p,q}(\R)}^{(R,T,v,N)}_\osc$ are mutually equivalent on $L_{\max\{1,p,v\}}^\loc(\R)$. 
\end{theorem}

\begin{remark}\label{rem:main_term}
    Note that, according to Step 3 of the subsequent proof, in \autoref{thm_osc_B_Rda=2} the main term of our new quasi-norms can equivalently be replaced by
    $$
        \norm{ f \sep \mathcal{M}^u_p(\R)}
        \quad\text{or} \quad
        \norm{ \Big( \int_{B(\,\cdot\,,R) } \abs{f(y)}^v\d y \Big)^{\frac{1}{v}} \sep \mathcal{M}^u_p(\R)}, \quad R>0,
    $$
    where $u:=1/(\frac{1}{p}-\tau)$.
\end{remark}

\begin{proof}[Proof (of \autoref{thm_osc_B_Rda=2})]
    W.l.o.g.\ we again assume that $q,v<\infty$. In addition, note that for every $P:=Q_{j,k}\in\mathcal{Q}$ we have $\abs{P}=2^{-jd}$ and
    \begin{align*}
        \Big ( \int_{P} 2^{\frac{djp}{v}} \Big (  \inf_{\pi \in \mathcal{P}_{N-1}} \!\!\Big( \int_{B(x,2^{-j}) } \abs{f(y) - \pi (y)}^v\d y \Big)^{\frac{1}{v}}  \Big )^{p} \d x \Big )^{\frac{1}{p}}  = \norm{ \osc_v^{N-1} f(\cdot,2^{-j}) \sep L_p(P)}.
    \end{align*}

    \emph{Step 1. } We first show the assertion for $\norm{\,\cdot \sep B^{s,\tau}_{p,q}(\R)}^{(1,T,v,N)}_\osc$, i.e.
    \begin{align}\label{eq:proof_normeq_tau}
		\norm{f \sep B^{s,\tau}_{p,q}(\R)}^{(1,T,v,N)}_\osc 
		\sim \norm{f \sep B^{s,\tau}_{p,q}(\R)}^{(\clubsuit)} ,
		\qquad f\in L_{\max\{1,p,v\}}^\loc(\R).
    \end{align}
    Exactly as in Substep 1a of the proof of \autoref{thm_osc_Rda=2}, for $0<T\leq \frac{1}{2}$ and $P\in\mathcal{Q}$ we find
    \begin{align*}
		\int_{0}^T t^{-sq} \norm{ \osc_v^{N-1} f(\cdot, t) \sep L_p(P) }^{q} \frac{\d t}{t}
        &\lesssim \sum_{j=1}^{\infty} 2^{jsq} \norm{ \osc_v^{N-1} f(\cdot,2^{-j}) \sep L_p(P)}^{q}
    \end{align*}
    as well as
    \begin{align*}
        \sum_{j=1}^{\infty} 2^{jsq} \norm{ \osc_v^{N-1} f(\cdot,2^{-j}) \sep L_p(P)}^{q}
        &\lesssim \norm{ \Big( \int_{B(\,\cdot\,,1) } \abs{f(y)}^v\d y \Big)^{\frac{1}{v}} \sep L_p(P)}^q \\
        &\qquad \qquad + \int_{0}^T t^{-sq} \norm{ \osc_v^{N-1} f(\cdot, t) \sep L_p(P) }^{q} \frac{\d t}{t}.
    \end{align*}
    Together this shows that for $T\leq \frac{1}{2}$ indeed
    \begin{align*}
        &\norm{f \sep B^{s,\tau}_{p,q}(\R)}^{(\clubsuit )} \\
        &\quad= \sup_{P\in\mathcal{Q}} \frac{1}{\abs{P}^\tau} \left( 
        \norm{ \Big( \int_{B(\,\cdot\,,1) } \abs{f(y)}^v\d y \Big)^{\frac{1}{v}} \sep L_p(P)}^q
        + \sum_{j=1}^{\infty} 2^{jsq} \norm{ \osc_v^{N-1} f(\cdot,2^{-j}) \sep L_p(P)}^{q} \right)^{\frac{1}{q}}\\
        &\quad\sim \sup_{P\in\mathcal{Q}} \frac{1}{\abs{P}^\tau} \left(  \norm{ \Big( \int_{B(\,\cdot\,,1) } \abs{f(y)}^v\d y \Big)^{\frac{1}{v}} \sep L_p(P)}^q
        + \int_{0}^T t^{-sq} \norm{ \osc_v^{N-1} f(\cdot, t) \sep L_p(P) }^{q} \frac{\d t}{t} \right)^{\frac{1}{q}} \\
        &\quad\sim \sup_{P\in\mathcal{Q}} \frac{1}{\abs{P}^\tau} \norm{ \Big( \int_{B(\,\cdot\,,1) } \abs{f(y)}^v\d y \Big)^{\frac{1}{v}} \sep L_p(P)}
        + \abs{f}^{(T,v,N)}_{\osc,\tau}.
    \end{align*}
    To extend this to $T>\frac{1}{2}$, we adapt Substep 1b from the proof of \autoref{thm_osc_Rda=2} and show that
    \begin{align}\label{eq:proof_normeq1b_B}
		\norm{f \sep B^{s,\tau}_{p,q}(\R)}^{(1,\infty,v,N)}_\osc 
		\lesssim \norm{f \sep B^{s,\tau}_{p,q}(\R)}^{(1,\frac{1}{2},v,N)}_\osc,
		\qquad f\in L_{\max\{1,p,v\}}^\loc(\R).
    \end{align}
    For this purpose, we note that now for every $P\in\mathcal{Q}$
    \begin{align*}
		\int_{\frac{1}{2}}^{\infty} t^{-sq} \norm{ \osc_v^{N-1} f(\cdot,t) \sep L_{p}(P)}^{q} \frac{\d t}{t} 
		&\lesssim \sum_{j=0}^{\infty} 2^{-jsq} \norm{ \osc_v^{N-1} f(\cdot,2^j) \sep L_{p}(P)}^{q},
    \end{align*}
    where the oscillation can be estimated using $K_j\sim 2^{jd}$ displacements $w_k\in\Z$, see \eqref{eq_displace_vec1}, by
    $$
	    \osc_{v}^{N-1} f(x,2^j) 
	    \leq \left( \sum_{k=1}^{K_j} 2^{-jd} \int_{B(x+w_k,1)} \abs{f(y)}^v \d y \right)^{\frac{1}{v}}, \qquad x\in\R, j\in\N_0.
    $$
    Similar as above with $\mu:=\min\{p,v\}$ this yields
    \begin{align*}
        \norm{ \osc_v^{N-1} f(\cdot,2^j) \sep L_{p}(P)}
        &\leq 2^{-j \frac{d}{v}} \left( \sum_{k=1}^{K_j} \norm{ \left(\int_{B(\cdot\,+w_k,1)} \abs{f(y)}^v \d y \right)^{\frac{\mu}{v}} \sep L_{\frac{p}{\mu}}(P)} \right)^{\frac{1}{\mu}} \\
        &= 2^{-j \frac{d}{v}} \left( \sum_{k=1}^{K_j} \norm{ \left(\int_{B(\cdot\,+w_k,1)} \abs{f(y)}^v \d y \right)^{\frac{1}{v}} \sep L_{p}(P)}^\mu \right)^{\frac{1}{\mu}}.
    \end{align*}
    However, now for each $k$ we can estimate
    \begin{align*}
         \norm{ \left(\int_{B(\,\cdot\,+w_k,1)} \abs{f(y)}^v \d y \right)^{\frac{1}{v}} \sep L_{p}(P)}
         \leq \abs{P}^\tau \cdot \sup_{P'\in\mathcal{Q}} \frac{1}{\abs{P'}^\tau}  \norm{ \left(\int_{B(\cdot,1)} \abs{f(y)}^v \d y \right)^{\frac{1}{v}} \sep L_{p}(P')}
    \end{align*}
    which together with $K_j\sim 2^{jd}$ gives
    \begin{align*}
        \norm{ \osc_v^{N-1} f(\cdot,2^j) \sep L_{p}(P)}
        &\lesssim 2^{jd [\frac{1}{\mu}-\frac{1}{v}]} \cdot \abs{P}^\tau \cdot \sup_{P'\in\mathcal{Q}} \frac{1}{\abs{P'}^\tau}  \norm{ \left(\int_{B(\cdot,1)} \abs{f(y)}^v \d y \right)^{\frac{1}{v}} \sep L_{p}(P')}.
    \end{align*}
    As before, this allows to conclude
    \begin{align*}
		&\int_{\frac{1}{2}}^{\infty} t^{-sq} \norm{ \osc_v^{N-1} f(\cdot,t) \sep L_{p}(P)}^{q} \frac{\d t}{t} \\
        &\qquad \lesssim \sum_{j=0}^{\infty} 2^{j (d [\frac{1}{\mu}-\frac{1}{v}]-s)q} 
        \cdot \abs{P}^{\tau q} \cdot \left( \sup_{P'\in\mathcal{Q}} \frac{1}{\abs{P'}^\tau}  \norm{ \left(\int_{B(\cdot,1)} \abs{f(y)}^v \d y \right)^{\frac{1}{v}} \sep L_{p}(P')} \right)^q
    \end{align*}
    and since $ \sum_{j=0}^{\infty} 2^{j (d [\frac{1}{\mu}-\frac{1}{v}]-s)q} \sim 1 $ we find
    \begin{align*}
        \abs{f}^{(\infty,v,N)}_{\osc,\tau} 
        &\lesssim \abs{f}^{(\frac{1}{2},v,N)}_{\osc,\tau} + \sup_{P\in\mathcal{Q}} \frac{1}{\abs{P}^\tau}  \norm{ \left(\int_{B(\,\cdot\,,1)} \abs{f(y)}^v \d y \right)^{\frac{1}{v}} \sep L_{p}(P)} 
        = \norm{f \sep B^{s,\tau}_{p,q}(\R)}^{(1,\frac{1}{2},v,N)}_\osc
    \end{align*}
   which implies \eqref{eq:proof_normeq1b_B} and thus \eqref{eq:proof_normeq_tau} for all $T$.

   \emph{Step 2. } Let us extend \eqref{eq:proof_normeq_tau} to arbitrary $R>0$. This follows from \autoref{lem:discrete_morrey} and \cite[Lemma~4]{HoWe23}.

   \emph{Step 3. } Finally, we have to show that
    \begin{align*}
		\norm{f \sep B^{s,\tau}_{p,q}(\R)}
  \sim \norm{f \sep B^{s,\tau}_{p,q}(\R)}^{(T,v,N)}_\osc,
		\qquad f\in L_{\max\{1,p,v\}}^\loc(\R).
    \end{align*}
    Here we can argue exactly as in Step 2 of the proof of \autoref{thm_osc_Rda=2} with \autoref{l_el_em_EN} being replaced by \autoref{l_BF_bp1}, since \autoref{lem:discrete_morrey} implies that for $u:=1/(\frac{1}{p}-\tau)$
    $$
        \sup_{P\in\mathcal{Q}} \frac{1}{\abs{P}^\tau} \norm{ \Big( \int_{B(\,\cdot\,,R) } \abs{f(y)}^v\d y \Big)^{\frac{1}{v}} \sep L_p(P)} 
        \sim \norm{ \Big( \int_{B(\,\cdot\,,R) } \abs{f(y)}^v\d y \Big)^{\frac{1}{v}} \sep \mathcal{M}^u_p(\R)}, \quad R>0,
    $$
    and
    $$
        \sup_{P\in\mathcal{Q}} \frac{1}{\abs{P}^\tau} \norm{ f \sep L_p(P)} 
        \sim \norm{ f \sep \mathcal{M}^u_p(\R)}.
    $$
    Thus, the proof is complete.
\end{proof}

\subsection{Difference Characterizations of \texorpdfstring{$B^{s,\tau}_{p,q}(\R)$}{Bstaupq(Rd)} }

Here we recall and generalize existing characterizations of $B^{s,\tau}_{p,q}(\mathbb{R}^d)$ in terms of higher order differences. Some first results concerning this topic already can be found in the literature; see \cite[Subsection~4.3.2]{ysy}, \cite{Dri1}, and \cite{HoSi20}. The subsequent statement slightly extends \cite[Theorem~3.1]{HoSi20} which already proves the case $\norm{\,\cdot \sep B^{s,\tau}_{p,q}(\R)}^{(T,v,N)}_\Delta$ with $T\geq 1$. \autoref{rem:main_term} applies likewise.

\begin{theorem}\label{thm_diff_B_Rda=2}
    Let $d,N \in \mathbb{N}$, $0< p < \infty$, $0 \leq \tau < \frac{1}{p}$, $0 < q,T,v \leq \infty$, $0<R<\infty$, and $s\in\mathbb{R}$ with
    \begin{align*}
		d\, \max \left\{  0, \frac{1}{p} - 1, \frac{1}{p} - \frac{1}{v} \right\} < s < N.
    \end{align*}
    Then 
    \begin{align*}
        B^{s,\tau}_{p,q}(\R) 
        &= \left\{ f\in L_{\max\{1,p,v\}}^\loc(\R) \,:\, \norm{f \sep B^{s,\tau}_{p,q}(\R) }^{(T,v,N)}_\Delta < \infty \right\} \\
        &= \left\{ f\in L_{\max\{1,p,v\}}^\loc(\R) \,:\, \norm{f \sep B^{s,\tau}_{p,q}(\R) }^{(R,T,v,N)}_\Delta < \infty \right\}
    \end{align*}
    and the quasi-norms $\norm{\,\cdot \sep B^{s,\tau}_{p,q}(\R)}$, $\norm{\,\cdot \sep B^{s,\tau}_{p,q}(\R)}^{(T,v,N)}_\Delta$ as well as $\norm{\,\cdot \sep B^{s,\tau}_{p,q}(\R)}^{(R,T,v,N)}_\Delta$ are mutually equivalent on $L_{\max\{1,p,v\}}^\loc(\R)$. 
\end{theorem}
\begin{proof}
    Once more w.l.o.g.\ we assume $q,v<\infty$.
    
    \emph{Step 1 (Lower Bounds). }
    From \cite[Theorem~3.1]{HoSi20} it particularly follows that
    $$
        \norm{f \sep B^{s,\tau}_{p,q}(\R)} \gtrsim         \abs{f}^{(\infty,v,N)}_{\Delta,\tau}
    $$
    which due to monotonicity w.r.t.\ $T$ together with \autoref{thm_osc_B_Rda=2} proves the needed lower bounds, exactly as in Step 1 in the proof of \autoref{thm_diff_Rda=2}.

    \emph{Step 2 (Upper Bound for $T\leq \frac{1}{2}$ and $R:=N$). } Here we show that for $0<T\leq \frac{1}{2}$ there holds
    \begin{align}\label{eq:proof_tau}
        \norm{f \sep B^{s,\tau}_{p,q}(\R)} \lesssim \norm{f \sep B^{s,\tau}_{p,q}(\R)}^{(N,T,v,N)}_{\Delta}, \qquad f\in L_{\max\{1,p,v\}}^\loc(\R).
    \end{align}
    To this end, let $J\in\N$ with $2^{-J}\leq T$ and note that for $f\in L_{\max\{1,p,v\}}^\loc(\R)$, $x\in\R$, and $h>0$ there holds
    \begin{align*}
        \abs{\Delta^{N}_{h}f(x)} 
        &\leq \abs{\Delta^{N}_{\frac{h}{2^{J+1}}}f(x)} + \abs{\Delta^{N}_{h}f(x) - \Delta^{N}_{\frac{h}{2^{J+1}}}f(x)} \\
        &= \abs{\Delta^{N}_{\frac{h}{2^{J+1}}}f(x)} + \abs{\sum_{k=1}^N (-1)^{N-k} \binom{N}{k} \, \left[ f(x+kh) - f\left(x+k\frac{h}{2^{J+1}}\right)\right]} \\
        &\lesssim \abs{\Delta^{N}_{\frac{h}{2^{J+1}}}f(x)} + \sum_{k=1}^N \abs{f(x+kh)} + \sum_{k=1}^N \abs{f\left(x+k\frac{h}{2^{J+1}}\right)}.
    \end{align*}
    Here we used \eqref{eq:Delta}. Thus for each $j\in\{1,\ldots,J\}$ we obtain
    \begin{align*}
        \Big( 2^{jd} \int_{V^{N}(x,2^{-j})} \abs{\Delta^{N}_{h}f(x)}^v \d h \Big)^{\frac{1}{v}}
        &\lesssim \Big( 2^{jd} \int_{V^{N}(x,2^{-j})} \abs{\Delta^{N}_{\frac{h}{2^{J+1}}}f(x)}^v \d h \Big)^{\frac{1}{v}} \\
        & \qquad + \sum_{k=1}^N \Big( 2^{jd} \int_{V^{N}(x,2^{-j})} \abs{f(x+kh)}^v \d h \Big)^{\frac{1}{v}} \\
        & \qquad + \sum_{k=1}^N \Big( 2^{jd} \int_{V^{N}(x,2^{-j})} \abs{f\left(x+k\frac{h}{2^{J+1}}\right)}^v \d h \Big)^{\frac{1}{v}}.
    \end{align*}
    Substituting $\sigma:=\frac{h}{2^{J+1}}$ we conclude
    \begin{align*}
        &\Big( 2^{jd} \int_{V^{N}(x,2^{-j})} \abs{\Delta^{N}_{h}f(x)}^v \d h \Big)^{\frac{1}{v}} \\
        &\quad\lesssim \Big( 2^{(J+1+j)d} \int_{V^{N}(x,2^{-(J+1+j)})} \abs{\Delta^{N}_{\sigma}f(x)}^v \d \sigma \Big)^{\frac{1}{v}} \\
        &\quad\qquad+ \sum_{k=1}^N \Big( 2^{jd} \int_{V^{N}(x,2^{-j})} \abs{f(x+kh)}^v \d h \Big)^{\frac{1}{v}} \\
        &\quad\qquad+ \sum_{k=1}^N \Big( 2^{(J+1+j)d} \int_{V^{N}(x,2^{-(J+1+j)})} \abs{f\left(x+k\sigma\right)}^v \d \sigma \Big)^{\frac{1}{v}} \\
        &\quad\lesssim \Big( 2^{(J+1+j)d} \int_{V^{N}(x,2^{-(J+1+j)})} \abs{\Delta^{N}_{\sigma}f(x)}^v \d \sigma \Big)^{\frac{1}{v}} + 2^{(J+1+j)\frac{d}{v}} \Big( \int_{B(x,N)} \abs{f(y)}^v \d y \Big)^{\frac{1}{v}}
    \end{align*}
    and hence for every $P\in\mathcal{Q}$
    \begin{align*}
        &\sum_{j=1}^{J} 2^{jsq} \norm{ \Big( 2^{jd} \int_{V^{N}(\,\cdot\,,2^{-j})} \abs{\Delta^{N}_{h}f(\cdot)}^v \d h \Big)^{\frac{1}{v}} \sep L_p(P)}^{q} \\
        & \qquad \qquad   \lesssim \sum_{j=1}^{J} 2^{jsq} \norm{ \Big( 2^{(J+1+j)d} \int_{V^{N}(\,\cdot\,,2^{-(J+1+j)})} \abs{\Delta^{N}_{\sigma}f(\cdot)}^v \d \sigma \Big)^{\frac{1}{v}} \sep L_p(P)}^{q} \\
        &\qquad \qquad \qquad \qquad    + \sum_{j=1}^{J} 2^{jsq} 2^{(J+1+j)\frac{d}{v}q} \norm{ \Big( \int_{B(\,\cdot\,,N)} \abs{f(y)}^v \d y \Big)^{\frac{1}{v}} \sep L_p(P)}^{q} \\
        & \qquad \qquad    \lesssim \norm{ \Big( \int_{B(\,\cdot\,,N)} \abs{f(y)}^v \d y \Big)^{\frac{1}{v}} \sep L_p(P)}^{q} \\
        & \qquad \qquad \qquad \qquad    + \sum_{k=J+1}^{\infty} 2^{ksq} \norm{ \Big( 2^{kd} \int_{V^{N}(\,\cdot\,,2^{-k})} \abs{\Delta^{N}_{h}f(\cdot)}^v \d h \Big)^{\frac{1}{v}} \sep L_p(P)}^{q}.
    \end{align*}
    Now \cite[Proposition~3.3]{HoSi20} in combination with the estimate above implies
    \begin{align*}
        \norm{f \sep B^{s,\tau}_{p,q}(\R)} 
        &\sim \sup_{P\in\mathcal{Q}} \frac{1}{\abs{P}^\tau} \Bigg( 
        \norm{ \Big( \int_{B(\,\cdot\,,1) } \abs{f(y)}^v\d y \Big)^{\frac{1}{v}} \sep L_p(P)}^q \\
        &\qquad\qquad\quad + \sum_{j=1}^{\infty} 2^{jsq} \norm{ \Big( 2^{jd} \int_{V^{N}(\,\cdot\,,2^{-j})} \abs{\Delta^{N}_{h}f(\cdot)}^v \d h \Big)^{\frac{1}{v}} \sep L_p(P)}^{q} \Bigg)^{\frac{1}{q}} \\
        &\lesssim \sup_{P\in\mathcal{Q}} \frac{1}{\abs{P}^\tau} \Bigg( 
        \norm{ \Big( \int_{B(\,\cdot\,,N)} \abs{f(y)}^v \d y \Big)^{\frac{1}{v}} \sep L_p(P)} \\
        &\qquad\qquad\quad + \left[\sum_{k=J+1}^{\infty} 2^{ksq} \norm{ \Big( 2^{kd} \int_{V^{N}(\,\cdot\,,2^{-k})} \abs{\Delta^{N}_{h}f(\cdot)}^v \d h \Big)^{\frac{1}{v}} \sep L_p(P)}^{q}\right]^{\frac{1}{q}} \Bigg),
    \end{align*}
    where similar to Step 1 in the proof of \autoref{thm_osc_Rda=2} we have
    \begin{align*}
        &\sum_{k=J+1}^{\infty} 2^{ksq} \norm{ \Big( 2^{kd} \int_{V^{N}(\,\cdot\,,2^{-k})} \abs{\Delta^{N}_{h}f(\cdot)}^v \d h \Big)^{\frac{1}{v}} \sep L_p(P)}^{q} \\
        &\qquad \qquad \lesssim \int_0^T t^{-sq} \norm{ \Big( t^{-d} \int_{V^{N}(\,\cdot\,,t)} \abs{\Delta^{N}_{h}f(\cdot)}^v \d h \Big)^{\frac{1}{v}} \sep L_p(P)}^{q} \frac{\d t}{t}  \leq \abs{P}^{\tau q} \left(\abs{f}^{(T,v,N)}_{\Delta,\tau} \right)^q.
    \end{align*}
    This verifies \eqref{eq:proof_tau}.

    \emph{Step 3. } Using the monotonicity of $\abs{f}^{(T,v,N)}_{\Delta,\tau}$ in $T$ together with \autoref{lem:discrete_morrey} and \cite[Lemma~4]{HoWe23} allows to extend \eqref{eq:proof_tau} to 
    \begin{align*}
        \norm{f \sep B^{s,\tau}_{p,q}(\R)} 
        \lesssim \norm{f \sep B^{s,\tau}_{p,q}(\R)}^{(R,T,v,N)}_{\Delta}, \qquad f\in L_{\max\{1,p,v\}}^\loc(\R),
    \end{align*}
    with arbitrary $0<R<\infty$ and $0<T \leq \infty$. 
    In combination with Step 1 this shows the claim for $\norm{\cdot \sep B^{s,\tau}_{p,q}(\R)}^{(R,T,v,N)}_{\Delta}$.

    \emph{Step 4. } Literally the same arguments as in Step 3 of the proof of \autoref{thm_osc_B_Rda=2} can now be used to derive the assertion for $\norm{\cdot \sep B^{s,\tau}_{p,q}(\R)}^{(T,v,N)}_{\Delta}$ which completes the proof.
\end{proof}

\section{Characterizations on Domains}\label{sect:characterizations_domains}
In this section we shall derive the main results of this paper, namely the intrinsic characterizations of Besov-Morrey spaces $\mathcal{N}^{s}_{u,p,q}(\Omega)$ and Besov-type spaces $B^{s,\tau}_{p,q}(\Omega)$ on Lipschitz domains $\Omega \subset \mathbb{R}^d$ in terms of local oscillations and differences.

\subsection{General Lower Bounds}
Let us start by proving the lower bounds. From Theorems \ref{thm_osc_Rda=2} and \ref{thm_diff_Rda=2} the following assertion (valid for general domains) can be obtained easily.
\begin{prop}[Lower Bound in Besov-Morrey Spaces on Domains]\label{prop:N_Omega_lower}
    For $d\in\N$ let $\Omega\subsetneq\R$ be any domain.
    Further, let $0< p \leq u < \infty$, $0 < q,T,v \leq \infty$, $0<R<\infty$, $N \in \mathbb{N}$, and $s\in\mathbb{R}$ be such that
	\begin{align*}
		d\, \max\! \left\{ 0, \frac{1}{p} - 1, \frac{1}{p} - \frac{1}{v} \right\} < s < N.
	\end{align*}
	Then for $f\in \mathcal{N}^{s}_{u,p,q}(\Omega)$ there holds $f \in L_{\max\{1,p,v\}}^\loc(\Omega)$ as well as
	$$
	    \norm{f\sep \mathcal{N}^{s}_{u,p,q}(\Omega)} 
	    \gtrsim \norm{\Big( \int_{B(\,\cdot\,,R)\cap \Omega} \abs{f(y)}^{v} \d y \Big)^{\frac{1}{v}} \sep \mathcal{M}^{u}_{p}(\Omega)} 
        + \norm{ f \sep \mathcal{M}^{u}_{p}(\Omega)}
        + \abs{f}^{(T,v,N)}_{\osc,\Omega}
        + \abs{f}^{(T,v,N)}_{\Delta,\Omega} .
	$$
\end{prop}
\begin{proof}
    Let $f\in \mathcal{N}^{s}_{u,p,q}(\Omega)$. 
    Then by \autoref{defi:space_domain} there exists $F\in \mathcal{N}^{s}_{u,p,q}(\R)$ such that $F\vert_\Omega = f$ on~$\Omega$ in $\mathcal{D}'(\Omega)$ and
    $$
        \norm{f\sep \mathcal{N}^{s}_{u,p,q}(\Omega)} 
        \geq \frac{1}{2} \norm{F\sep \mathcal{N}^{s}_{u,p,q}(\R)}.
    $$
    Using \autoref{thm_osc_Rda=2} we can conclude that this extension satisfies $F\in L_{\max\{1,p,v\}}^\loc(\R)$ and $\norm{F\sep \mathcal{N}^{s}_{u,p,q}(\R)} \sim \norm{F\sep \mathcal{N}^{s}_{u,p,q}(\R)}^{(R,T,v,N)}_\osc$. 
    Therefore, $F|_\Omega \in L_{\max\{1,p,v\}}^\loc(\Omega)$ equals $f$ pointwise a.e.\ in $\Omega$ and
    \begin{align*}
        \norm{\Big( \int_{B(\,\cdot\,,R)} \abs{F(y)}^{v} \d y \Big)^{\frac{1}{v}} \sep \mathcal{M}^{u}_{p}(\R)}
        &\geq \norm{\Big( \int_{B(\,\cdot\,,R)\cap \Omega} \Big| F|_\Omega(y) \Big|^{v} \d y \Big)^{\frac{1}{v}} \sep \mathcal{M}^{u}_{p}(\R)} \\
        &\gtrsim \norm{\Big( \int_{B(\,\cdot\,,R)\cap \Omega} \abs{f(y)}^{v} \d y \Big)^{\frac{1}{v}} \sep \mathcal{M}^{u}_{p}(\Omega)},
    \end{align*}
    whereby we have used \cite[Lemma~3(ii)]{HoWe23}.
    Further note that for every $x\in\Omega$ and $t>0$
    \begin{align}\label{eq:proof_osc_restriction}
        \osc_{v}^{N-1} F(x,t) \geq \osc_{v,\Omega}^{N-1} [F|_\Omega](x,t) = \osc_{v,\Omega}^{N-1} f(x,t)
    \end{align}
    such that
    $$
        \norm{ \osc_{v}^{N-1} F(\cdot,t) \sep \mathcal{M}^{u}_{p}(\R)}
        \gtrsim \norm{ \osc_{v}^{N-1} F(\cdot,t) \sep \mathcal{M}^{u}_{p}(\Omega)}
        \geq \norm{ \osc_{v,\Omega}^{N-1} f(\cdot,t) \sep \mathcal{M}^{u}_{p}(\Omega)}
    $$
    and hence also
    $$
        \abs{F}^{(T,v,N)}_{\osc} = \bigg( \int_{0}^{T} t^{-sq} \norm{ \osc_{v}^{N-1} F(\cdot,t) \sep \mathcal{M}^{u}_{p}(\R)}^{q} \frac{\d t}{t} \bigg)^{\frac{1}{q}} \gtrsim \abs{f}^{(T,v,N)}_{\osc,\Omega}.
    $$
    Together this yields $\norm{f\sep \mathcal{N}^{s}_{u,p,q}(\Omega)} \gtrsim \norm{F\sep \mathcal{N}^{s}_{u,p,q}(\R)}^{(R,T,v,N)}_\osc \gtrsim \norm{f\sep \mathcal{N}^{s}_{u,p,q}(\Omega)}^{(R,T,v,N)}_\osc$.
    Similarly we obtain $\norm{f\sep \mathcal{N}^{s}_{u,p,q}(\Omega)} \gtrsim \norm{F\sep \mathcal{N}^{s}_{u,p,q}(\R)}^{(T,v,N)}_\osc \gtrsim \norm{f\sep \mathcal{N}^{s}_{u,p,q}(\Omega)}^{(T,v,N)}_\osc$ since in this case the main term satisfies
    $$
        \norm{F \sep \mathcal{M}^{u}_{p}(\R)} \geq \norm{F|_\Omega \sep \mathcal{M}^{u}_{p}(\Omega)} = \norm{f \sep \mathcal{M}^{u}_{p}(\Omega)}.
    $$
    
    Finally, exactly the same arguments can be used for ball means of differences instead of oscillations. We only have to replace \eqref{eq:proof_osc_restriction} by
    \begin{align}
        \Big( t^{-d} \int_{B(x,t)} \abs{\Delta^{N}_{h} F(x)}^v \d h \Big)^{\frac{1}{v}}
        &\geq \Big( t^{-d} \int_{V^{N}(x,t)} \abs{\Delta^{N}_{h}[F|_\Omega](x)}^v \d h \Big)^{\frac{1}{v}} \nonumber\\
        &= \Big( t^{-d} \int_{V^{N}(x,t)} \abs{\Delta^{N}_{h}f(x)}^v \d h \Big)^{\frac{1}{v}}, \qquad x\in\Omega, t>0,\label{eq:proof_Delta_restriction}
    \end{align}
    and employ \autoref{thm_diff_Rda=2} in place of \autoref{thm_osc_Rda=2}. 
\end{proof}

For the Besov-type spaces the corresponding result follows from Theorems \ref{thm_osc_B_Rda=2} and \ref{thm_diff_B_Rda=2}.

\begin{prop}[Lower Bound in Besov-Type Spaces on Domains]\label{prop:B_Omega_lower}
    For $d\in\N$ let $\Omega\subsetneq\R$ be any domain.
    Further, let $0< p < \infty$, $0\leq\tau<\frac{1}{p}$, $0 < q,T,v \leq \infty$, $0<R<\infty$, $N \in \mathbb{N}$, and $s\in\mathbb{R}$ be such that
	\begin{align*}
		d\, \max\! \left\{ 0, \frac{1}{p} - 1, \frac{1}{p} - \frac{1}{v} \right\} < s < N.
	\end{align*}
	Then for $f\in B^{s,\tau}_{p,q}(\Omega)$ there holds $f \in L_{\max\{1,p,v\}}^\loc(\Omega)$ as well as
	\begin{align*}
	    \norm{f\sep B^{s,\tau}_{p,q}(\Omega)} 
	    &\gtrsim \sup_{P\in\mathcal{Q}} \frac{1}{\abs{P}^\tau} \norm{\Big( \int_{B(\,\cdot\,,R)\cap \Omega} \abs{f(y)}^{v} \d y \Big)^{\frac{1}{v}} \sep L_{p}(P \cap \Omega)} \\
        &\qquad \qquad + \sup_{P\in\mathcal{Q}} \frac{1}{\abs{P}^\tau} \norm{ f \sep L_{p}(P\cap \Omega)}
        + \abs{f}^{(T,v,N)}_{\osc,\tau,\Omega}
        + \abs{f}^{(T,v,N)}_{\Delta,\tau,\Omega} .
	\end{align*}
\end{prop}
\begin{proof}
    Again for every $f\in B^{s,\tau}_{p,q}(\Omega)$ we find an extension $F\in B^{s,\tau}_{p,q}(\R) \subset L^\loc_{\max\{1,p,v\}}(\R)$ with
    $$
        \norm{f\sep B^{s,\tau}_{p,q}(\Omega)} 
        \geq \frac{1}{2} \norm{F\sep B^{s,\tau}_{p,q}(\R)}.
    $$
    Hence, like in the proof of \autoref{prop:N_Omega_lower}, it suffices to lower bound all components of the quasi-norms of $F$ given in Theorems \ref{thm_osc_B_Rda=2} and \ref{thm_diff_B_Rda=2} by the corresponding expressions of~$f$ on~$\Omega$. 
    For each $P\in\mathcal{Q}$ Formula \eqref{eq:proof_osc_restriction} yields
    $$
        \norm{ \osc_{v}^{N-1} F(\cdot,t) \sep L_{p}(P)}
        \gtrsim \norm{ \osc_{v}^{N-1} F(\cdot,t) \sep L_{p}(P\cap \Omega)}
        \geq \norm{ \osc_{v,\Omega}^{N-1} f(\cdot,t) \sep L_{p}(P\cap \Omega)}
    $$
    which proves $\abs{F}^{(T,v,N)}_{\osc,\tau} \gtrsim \abs{f}^{(T,v,N)}_{\osc,\tau,\Omega}$.
    Likewise \eqref{eq:proof_Delta_restriction} shows $\abs{F}^{(T,v,N)}_{\Delta,\tau} \gtrsim \abs{f}^{(T,v,N)}_{\Delta,\tau,\Omega}$.
    Moreover, for the main terms the claim is obvious since for each $P\in\mathcal{Q}$ we have
    \begin{align*}
        \Big( \int_{B(x,R)} \abs{F(y)}^{v} \d y \Big)^{\frac{1}{v}}
        &\geq \Big( \int_{B(x,R)\cap \Omega} \Big| F|_\Omega(y) \Big|^{v} \d y \Big)^{\frac{1}{v}}
        = \Big( \int_{B(x,R)\cap \Omega} \abs{f(y)}^{v} \d y \Big)^{\frac{1}{v}}, \quad x\in P,
    \end{align*}
    as well as $\norm{F \sep L_p(P)} \geq \norm{F \sep L_p(P\cap\Omega)} = \norm{f \sep L_p(P\cap\Omega)}$.
\end{proof}

\subsection{Differences on Special Lipschitz Domains}\label{Subsec_Diff1_splido}

Now let us turn to the corresponding upper estimates. For that purpose we will start with differences on special Lipschitz domains.
On the one hand, the resulting characterization is of interest on its own sake. On the other hand, it will also provide a technical tool which is used to derive the assertions concerning oscillations later on. 

\begin{prop}\label{prop:diff_special_Nsupq}
    For $d\in\N$ let $ \Omega \subset \R$ be a special Lipschitz domain. Moreover, let $0< p \leq u < \infty$, $0 < q,T \leq \infty$, $1 \leq v \leq \infty$, $0<R<\infty$, $ N \in \N$, and $s>0$.
    Then for $f \in L^\loc_{\max\{p,v\}}(\Omega)$ there holds
    \begin{align*}
        \norm{ f \sep \mathcal{N}^{s}_{u,p,q}(\Omega) } 
        & \lesssim  \norm{ \bigg(\int_{B(\,\cdot\,, R) \cap \Omega}   \abs{f(y)}^v \d y\bigg)^{\frac{1}{v}} \sep  \mathcal{M}^{u}_{p}( \Omega)} + \abs{f}^{(T,v,N)}_{\Delta,\Omega} . 
    \end{align*}
    If additionally $p\geq 1$, then also $\norm{ f \sep \mathcal{N}^{s}_{u,p,q}(\Omega) } \lesssim  \norm{ f \sep  \mathcal{M}^{u}_{p}( \Omega)} + \abs{f}^{(T,v,N)}_{\Delta,\Omega} $.
    In both cases the implied constants are independent of $f$.
\end{prop}

\begin{proof}
    W.l.o.g.\ we can assume that $T\leq 1$ and $q<\infty$.

 \emph{Step 1. }  In this first step, we will make use of so-called intrinsic Littlewood-Paley type characterizations for the spaces $\mathcal{N}^{s}_{u,p,q}(\Omega)$. 
 For that purpose, let us recall the definition of a Littlewood-Paley family associated to a special Lipschitz domain $\Omega$ from \cite{Yao}. A sequence $\phi = ( \phi_{j} )_{j = 0}^{\infty} \subset \mathcal{S}(\mathbb{R}^d)$ of Schwartz functions is called a Littlewood-Paley family associated with $\Omega$ if the following conditions are fulfilled:
    \begin{enumerate}
        \item For all multi-indices $\alpha \in \mathbb{N}_{0}^{d}$ we have $ \int_{\mathbb{R}^d} x^{\alpha} \,\phi_{1}(x) \d x = 0 $.

        \item For $j \in\N$ we have $\phi_{j} = 2^{(j-1)d}\, \phi_{1}(2^{j-1} \,\cdot\,)$.
    
        \item There holds $ \sum_{j = 0}^{\infty} \phi_{j} = \delta $ with convergence in $\mathcal{S}'(\mathbb{R}^d)$, whereby $\delta$ denotes the Dirac delta distribution.

        \item For all $j \in \mathbb{N}_{0}$ we have
        \begin{align*}
            \supp (\phi_{j}) \subset \left\{ x = (x',x_{d}) \in \mathbb{R}^{(d-1)+1} \,:\, x_{d} < - \norm{ \abs{\nabla \omega} \sep L_{\infty}(\mathbb{R}^{d-1}) } \cdot \abs{x'} \right\}=:-K.
        \end{align*}
        Notice that $-K$ can be interpreted as reflected, narrow, vertically directed cone; see also \cite[Section~2]{Ry99} and \cite[Section~4]{ZHS}. 
    \end{enumerate}
    Now let $\phi = ( \phi_{j} )_{j = 0}^{\infty} \subset \mathcal{S}(\mathbb{R}^d)$ be a Littlewood-Paley family associated with the special Lipschitz domain $\Omega \subset \mathbb{R}^d$. 
    Further let $0 < p \leq u < \infty$, $0 < q \leq \infty$, and $s \in \mathbb{R}$. 
    Then \cite[Theorem~1]{Yao} shows that there exists a constant $C_{1} >0$ independent of $f \in \mathcal{D}'(\Omega)$ such that
    \begin{align}\label{eq_LPTC1_new}
        \norm{ f \sep \mathcal{N}^{s}_{u,p,q}(\Omega)} 
        \leq  C_{1}  \bigg( \sum_{j = 0}^{\infty}  2^{jsq} \norm{  \phi_{j} \ast f  \sep   \mathcal{M}^{u}_{p}(\Omega)  }^{q}  \bigg)^{\frac{1}{q}} ,
    \end{align}
    where as usual $\ast$ denotes the convolution.

    \emph{Step 2. } To continue the proof we choose a particular Littlewood-Paley family. 
    For that purpose, we closely follow the proof of \cite[Proposition 2]{HoWe23} and employ some ideas of Triebel as given in Step~3 of his proof of \cite[Theorem~1.118]{Tr06}. 
    That is, we work with the distinguished kernels constructed in \cite[Section~3.3]{Tr92} and estimate them from above as described in \cite[Formula (1.392)]{Tr06}.
    Let us briefly recall some necessary notations. 
    Fix $N \in \mathbb{N}$ and let $\varphi \in C^{\infty}_{0}(\mathbb{R})$ as well as $\psi \in C^{\infty}_{0}(\mathbb{R})$ be such that 
    \begin{align*}
        \int_{\mathbb{R}} \varphi(x) \d x = 1 
        \qquad \mbox{and} \qquad 
        \varphi(x) - \frac{1}{2}\, \varphi\!\left(\frac{x}{2}\right) = \psi^{(N)}(x), \qquad x\in\mathbb{R}.
    \end{align*}
    Based on that we define $\Phi \in C^{\infty}_{0}(\mathbb{R}^d)$ by setting
    \begin{align*}
        \Phi(x) := \prod_{j=1}^{d} \varphi(x_{j}), \qquad x=(x_{1} , \ldots , x_{d})  \in \mathbb{R}^d.
    \end{align*}
    Using this function, for $x\in\R$ we put
    \begin{align*}
        k_{0}(x) := \frac{(-1)^{N+1}}{N!} \sum_{r=1}^{N} \sum_{m=1}^{N} (-1)^{r+m} \binom{N}{r} \binom{N}{m}\,  \frac{m^{N-d}}{r^d}\, \Phi\!\left(\frac{x}{rm}\right),
    \end{align*}
    $ k(x) := k_{0}(x) - 2^{-d} k_{0}(\frac{x}{2} ) $ as well as $ k_{j}(x) := 2^{jd} k(2^{j}x) $ for $ j\in\N$. 
    Note that $\varphi$ and $\psi$ can be chosen in a way such that $  \supp (k_{j}) \subset B(0, 2^{-j} N) \cap -K $, whereby (depending on $\Omega$) perhaps a rotation of either $\Phi$ or $\Omega$ becomes necessary. 
    For details let us refer to Step 2 of the proof of \cite[Proposition 2]{HoWe23}. 
    To continue, for $j \in \mathbb{N}_{0}$ and $f \in L^\loc_{\max\{p,v\}}(\Omega)$ we define
    \begin{equation}\label{eq_def_fj_late}
        f_{j}(x) 
        := (k_{j} \ast f)(x) = \int_{B(x,2^{-j} N) \,\cap\, x + K} k_{j}(x-z) f(z) \d z .
    \end{equation} 
    In \cite[Formula (10) on p.175]{Tr92} it was observed that for $j \in \mathbb{N}$ we can write
    \begin{align}\label{eq_kern_diff1_new}
        f_{j}(x) = \sum_{| \alpha | = N} \int_{\mathbb{R}^d} D^{\alpha} k_{\alpha}(-y) \, \Delta^{N}_{2^{-j}y} f(x) \d y,
    \end{align}
    with appropriate $k_{\alpha} \in C^{\infty}_{0}(\mathbb{R}^d)$.  
    Recall that $( k_{j} )_{j = 0}^{\infty}$ can be interpreted as a Littlewood-Paley family; cf.\ \cite[Proposition 2]{HoWe23}. 
    Using these $ f_{j} $ Formula~\eqref{eq_LPTC1_new} becomes
    \begin{equation}\label{eq_LPTC2_new}
     \norm{ f \sep \mathcal{N}^{s}_{u,p,q}(\Omega)} 
        \lesssim    \bigg ( \sum_{j = 0}^{\infty}  2^{jsq} \norm{ f_{j}  \sep \mathcal{M}^{u}_{p}(\Omega) }^{q} \bigg )^{\frac{1}{q}} .   
    \end{equation}

    \emph{Step 3. } 
    Next we prove that the functions $\abs{f_{j}}$ can be estimated from above by integrals of higher order differences. This observation already has been made in \cite[Formula~(1.392)]{Tr06}, see also \cite[Proposition 2]{HoWe23}.
    Therefore we only recall the main ideas and refer to \cite{HoWe23} for the details. 
    Let $J\in\N$ be such that $2^{-J} \leq T$ and consider $j\leq J$. For $x \in \Omega$ the definition of $k_{j}$ (given in Step 2 above) yields
    $$
        \abs{k_j(x)} \leq 2^{jd} \abs{k_0(2^j x)} + 2^{(j-1)d} \abs{k_0(2^{j-1} x)}. 
    $$
    Consequently for $j=0,\ldots,J$ we find $\abs{k_j(x)} \lesssim \sum_{\ell=0}^J \abs{k_0(2^{\ell} x)}$ with an implicit constant depending on $d$ and $J$.
    Hence, for $x\in\Omega$ using the definitions of $f_{j}$ and $k_{0}$ we get
    \begin{equation}\label{eq_domdif_j<J1}
        \abs{f_j(x)} \lesssim \int_{B(x, 2^{-j} N) \,\cap\, \Omega} \abs{k_{j}(x-y)} \abs{f(y)} \d y  \lesssim \int_{B(x, N) \,\cap\, \Omega} \abs{f(y)} \d y.
    \end{equation}
    
    Now let us turn to the case $j > J$. Here we want to use \eqref{eq_kern_diff1_new}. Recall that the functions $k_{\alpha} \in C^{\infty}_{0}(\mathbb{R}^d) $ are defined in terms of $\Phi$, see \cite[Formula (26) in Chapter 3.3.2]{Tr92}. 
    Hence they can be chosen such that $\supp k_{\alpha} \subset B(0,1)  \cap -K $. For $x \in \Omega$ the narrow vertically directed cone property of $\Omega$ yields
    \begin{align}
        \abs{f_{j}(x)} 
        &\leq \sum_{\abs{\alpha} = N} \int_{B(0,1) \,\cap\, \{ y\in\R \;:\; x + \ell 2^{-j}  y \in \Omega \; \text{ for } \; 0 \leq \ell \leq N  \}} \abs{D^{\alpha} k_{\alpha}(-y)} \abs{\Delta^{N}_{2^{-j}y} f(x)} \d y \nonumber\\
        & \lesssim 2^{jd} \int_{B(0,2^{-j}) \,\cap\, \{ h\in\R \;:\; x + \ell h \in \Omega \; \text{ for } \; 0 \leq \ell \leq N \}} \abs{\Delta^{N}_{h} f(x)} \d h.\label{eq_domdif_j>J2}
    \end{align}
    Again we refer to  \cite[Proposition 2]{HoWe23} for some more details. Formula \eqref{eq_domdif_j>J2} will be important for us later on. Let us remark that it can also be found in \cite[Formula (1.392)]{Tr06}. 

    \textit{Step 4. } We proceed and combine our previous findings to prove the claim for $v=1$. A combination of \eqref{eq_LPTC2_new} and \eqref{eq_domdif_j<J1} yields 
    \begin{align*}
      & \norm{ f \sep \mathcal{N}^{s}_{u,p,q}(\Omega)} \\
       & \qquad \lesssim  \bigg ( \sum_{j = 0}^{J}  2^{jsq} \norm{ f_{j} \sep \mathcal{M}^{u}_{p}(\Omega)  }^{q} \bigg)^{\frac{1}{q}} 
       + \bigg( \sum_{j = J+1}^{\infty}  2^{jsq} \norm{  f_{j}  \sep   \mathcal{M}^{u}_{p}(\Omega)  }^{q}     \bigg)^{\frac{1}{q}} \\
       & \qquad \lesssim  \bigg ( \sum_{j = 0}^{J}  2^{jsq} \norm{ \int_{B(\,\cdot\,, N) \cap \Omega} \abs{f(y)} \d y \sep   \mathcal{M}^{u}_{p}(\Omega) }^{q} \bigg)^{\frac{1}{q}} 
       + \bigg( \sum_{j = J+1}^{\infty}  2^{jsq} \norm{  f_{j}  \sep \mathcal{M}^{u}_{p}(\Omega) }^{q}     \bigg)^{\frac{1}{q}} .
    \end{align*}
    For the first term an application of \cite[Lemma 4]{HoWe23} leads to
    \begin{align*}
        \bigg ( \sum_{j = 0}^{J}  2^{jsq} \norm{ \int_{B(\,\cdot\,, N) \cap \Omega} \abs{f(y)} \d y \sep   \mathcal{M}^{u}_{p}(\Omega) }^{q} \bigg)^{\frac{1}{q}} 
%& \lesssim \Big ( \sum_{j = 0}^{J}   \Big \Vert  \int_{B(x, N) \cap \Omega} \abs{f(y)} \d y \Big  \vert   \mathcal{M}^{u}_{p}(\Omega) \Big  \Vert^{q} \Big )^{\frac{1}{q}} \\
%& \lesssim   \Big \Vert  \int_{B(x, N) \cap \Omega} \abs{f(y)} \d y \Big  \vert   \mathcal{M}^{u}_{p}(\Omega) \Big  \Vert \\
        & \lesssim   \norm{ \int_{B(\,\cdot\,, R) \cap \Omega} \abs{f(y)} \d y \sep   \mathcal{M}^{u}_{p}(\Omega) }
    \end{align*}
    as $J=J(T)$ was chosen fixed.
    For the second term we can use \eqref{eq_domdif_j>J2} and $2^{-J} \leq T$ to get 
    \begin{align*}
        & \bigg( \sum_{j = J+1}^{\infty}  2^{jsq} \norm{  f_{j}  \sep \mathcal{M}^{u}_{p}(\Omega) }^{q}     \bigg)^{\frac{1}{q}} \\
        & \qquad \lesssim \bigg ( \sum_{j = J+1}^{\infty} 2^{jsq} \norm{ 2^{jd} \int_{B(0,2^{-j}) \,\cap\, \{ h\in\R \;:\; x + \ell h \in \Omega \; \text{ for } \; 0 \leq \ell \leq N \}} \abs{\Delta^{N}_{h} f(\cdot)} \d h \sep   \mathcal{M}^{u}_{p}(\Omega) }^{q}     \bigg )^{\frac{1}{q}} \\  
        & \qquad \lesssim \bigg ( \sum_{j = J+1}^{\infty} \int_{2^{-j}}^{2^{-j+1}}\, 2^{j}  2^{jsq} \norm{ 2^{jd} \int_{ V^{N}(\,\cdot\,,2^{-j}) } \abs{\Delta^{N}_{h} f(\cdot)} \d h \sep \mathcal{M}^{u}_{p}(\Omega) }^{q}  \d t \bigg)^{\frac{1}{q}} \\ 
        & \qquad \lesssim \bigg ( \int_{0}^{T} t^{-sq} \norm{ t^{-d} \int_{ V^{N}(\,\cdot\,,t) } \abs{\Delta^{N}_{h} f(\cdot)} \d h \sep   \mathcal{M}^{u}_{p}(\Omega) }^{q}  \frac{\d t}{t}   \bigg)^{\frac{1}{q}}
        =\abs{f}^{(T,1,N)}_{\Delta,\Omega}.  
    \end{align*}
    Consequently we have shown the claim for $v=1$, namely
    \begin{align*}
        \norm{ f \sep \mathcal{N}^{s}_{u,p,q}(\Omega)}  \lesssim \norm{ \int_{B(\,\cdot\,, R) \,\cap\, \Omega} \abs{f(y)} \d y \sep  \mathcal{M}^{u}_{p}(\Omega) }  + \abs{f}^{(T,1,N)}_{\Delta,\Omega}.
    \end{align*}

    \emph{Step 5. } 
    In order to complete the proof we can now use H\"older's inequality with $v\geq 1$ and $\abs{V^N(x,t)} \leq \abs{B(0,t)}\lesssim t^d$ to observe that  $\abs{f}^{(T,1,N)}_{\Delta,\Omega}\lesssim \abs{f}^{(T,v,N)}_{\Delta,\Omega}$ as well as
    \begin{align*}
        \norm{ \int_{B(\,\cdot\,, R) \,\cap\, \Omega} \abs{f(y)} \d y \sep  \mathcal{M}^{u}_{p}(\Omega) }
        \lesssim \norm{ \bigg( \int_{B(\,\cdot\,, R) \,\cap\, \Omega} \abs{f(y)}^{v} \d y  \bigg)^{\frac{1}{v}} \sep \mathcal{M}^{u}_{p}(\Omega)}.
    \end{align*}
    If in addition $ p \geq 1 $, we may instead employ \cite[Lemma 3]{HoWe23} to find
    \begin{align*}
        \norm{ \int_{B(\,\cdot\,, R) \,\cap\, \Omega} \abs{f(y)} \d y \sep  \mathcal{M}^{u}_{p}(\Omega) } 
        \lesssim \norm{ f \sep  \mathcal{M}^{u}_{p}( \Omega)}  .
    \end{align*}
    For the details we refer to Step 5 of the proof of \cite[Proposition 2]{HoWe23}.
\end{proof}

Likewise, we can prove a corresponding assertion for  Besov-type spaces $B^{s, \tau}_{p,q}(\Omega)$ defined on special Lipschitz domains $\Omega$.

\begin{prop}\label{prop:diff_special_Bstpq}
    For $d\in\N$ let $ \Omega \subset \R$ be a special Lipschitz domain. Moreover, let $0< p  < \infty$, $ 0 \leq \tau < \frac{1}{p} $, $0 < q,T \leq \infty$, $1 \leq v \leq \infty$, $0<R<\infty$, $ N \in \N$, and $s>0$.
    Then for $f \in L^\loc_{\max\{p,v\}}(\Omega)$ we have
    \begin{align*}
        \norm{ f \sep B^{s, \tau}_{p,q}(\Omega) } 
        & \lesssim \sup_{P\in\mathcal{Q}} \frac{1}{\abs{P}^\tau} \norm{\bigg( \int_{B(\,\cdot\,,R)\,\cap\, \Omega} \abs{f(y)}^{v} \d y \bigg)^{\frac{1}{v}} \sep L_{p}(P \cap \Omega)}  + \abs{f}^{(T,v,N)}_{\Delta, \tau , \Omega}
    \end{align*}
    with an implied constant independent of $f$.
    If additionally $p\geq 1$, then also
    \begin{align*}
    \norm{ f \sep B^{s, \tau}_{p,q}(\Omega) } \lesssim  \sup_{P\in\mathcal{Q}} \frac{1}{\abs{P}^\tau} \norm{ f \sep L_{p}(P\cap \Omega)} + \abs{f}^{(T,v,N)}_{\Delta, \tau ,  \Omega}. 
    \end{align*}
%    In both cases the implied constants are independent of $f$.
\end{prop}

\begin{proof}
    The proof is quite similar to that of \autoref{prop:diff_special_Nsupq}. Therefore, we shall only indicate the necessary modifications.

    Again, we start with a Littlewood-Paley type characterization, this time for the spaces~$B^{s, \tau}_{p,q}(\Omega)$. In \cite[Theorem~1]{Yao} Yao showed that for any Littlewood-Paley family $\phi = ( \phi_{j} )_{j = 0}^{\infty} \subset \mathcal{S}(\mathbb{R}^d)$ associated to the special Lipschitz domain $\Omega$ as well as for $ 0 < p < \infty $, $ 0 \leq \tau < \frac{1}{p} $, $ 0 < q \leq \infty$, and $ s > 0 $, there exists a constant $C_{1} >0$ such that
    \begin{align*}
        \norm{ f \sep B^{s, \tau}_{p,q}(\Omega)} 
        \leq  C_{1} \sup_{P\in\mathcal{Q}} \frac{1}{\abs{P}^\tau}  \bigg( \sum_{j = 0}^{\infty}  2^{jsq} \norm{\phi_{j} \ast f  \sep  L_{p}(P \cap \Omega) }^{q} \bigg)^{\frac{1}{q}}, \qquad f \in \mathcal{D}'(\Omega).
    \end{align*}
    To continue, we once more work with the specific distinguished kernels $\phi:=( k_{j} )_{j = 0}^{\infty}$ due to Triebel~\cite{Tr92}; see the proof of \autoref{prop:diff_special_Nsupq}.
    Given $f\in L^\loc_{\max\{p,v\}}(\Omega)\subset \mathcal{D}'(\Omega)$, for $j\in\N_0$ we let $f_j:=k_j \ast f$, see \eqref{eq_def_fj_late} and \eqref{eq_kern_diff1_new}, such that the latter formula becomes
    \begin{align*}
        &\norm{ f \sep B^{s, \tau}_{p,q}(\Omega)} \\
%        &\quad\lesssim \sup_{P\in\mathcal{Q}} \frac{1}{\abs{P}^\tau}  \bigg( \sum_{j = 0}^{\infty}  2^{jsq} \norm{f_{j}  \sep L_{p}(P \cap \Omega)  }^{q} \bigg)^{\frac{1}{q}} \\
        &\;\lesssim   \sup_{P\in\mathcal{Q}} \frac{1}{\abs{P}^\tau}  \bigg( \sum_{j = 0}^{J}  2^{jsq} \norm{f_{j} \sep  L_{p}(P \cap \Omega)  }^{q} \bigg)^{\frac{1}{q}} %\\
        %&\qquad \qquad 
        +  \sup_{P\in\mathcal{Q}} \frac{1}{\abs{P}^\tau}   \bigg(  \sum_{j = J + 1}^{\infty} 2^{jsq} \norm{ f_{j} \sep L_{p}(P \cap \Omega) }^{q} \bigg)^{\frac{1}{q}}, 
    \end{align*}
    where $J\in \N$ is chosen such that $2^{-J}\leq T$.
    Therein, the first term can be estimated using~\eqref{eq_domdif_j<J1}, \cite[Lemma 4]{HoWe23} as well as \autoref{lem:discrete_morrey} which together yield that
    \begin{align*}
        \sup_{P\in\mathcal{Q}} \frac{1}{\abs{P}^\tau}  \bigg( \sum_{j = 0}^{J}  2^{jsq} \norm{f_{j} \sep  L_{p}(P \cap \Omega)  }^{q} \bigg)^{\frac{1}{q}} 
        \lesssim \sup_{P\in\mathcal{Q}} \frac{1}{\abs{P}^\tau} \norm{ \int_{B(\,\cdot\,, R) \,\cap\, \Omega} \abs{f(y)} \d y \sep  L_{p}(P \cap \Omega) }.
    \end{align*}
    For the second term we note that \eqref{eq_domdif_j>J2} and $2^{-J} \leq T$ imply
    \begin{align*}
        &\sum_{j = J + 1}^{\infty} 2^{jsq} \norm{ f_{j}   \sep L_{p}(P \cap \Omega)  }^{q} \\
        &\qquad \lesssim  \sum_{j = J + 1}^{\infty} 2^{jsq} \norm{ 2^{jd} \int_{B(0,2^{-j}) \,\cap\, \{ h\in\R \;:\; x + \ell h \in \Omega \; \text{ for } \; 0 \leq \ell \leq N \}} \abs{\Delta^{N}_{h} f(\cdot)} \d h  \sep L_{p}(P \cap \Omega) }^{q}    \\
        & \qquad \lesssim  \int_{0}^{T} t^{-sq} \norm{  t^{-d} \int_{ V^{N}(\,\cdot\,, t) } \abs{\Delta^{N}_{h} f(\cdot)} \d h  \sep L_{p}(P \cap \Omega) }^{q} \frac{\d t}{t},
    \end{align*}
    i.e.\
    \begin{align*}
        \sup_{P\in\mathcal{Q}} \frac{1}{\abs{P}^\tau}   \bigg(  \sum_{j = J + 1}^{\infty} 2^{jsq} \norm{ f_{j} \sep L_{p}(P \cap \Omega) }^{q} \bigg)^{\frac{1}{q}}
        \lesssim \abs{f}^{(T,1,N)}_{\Delta,\tau,\Omega}.
    \end{align*}
    Together this shows the claim for the special case $v=1$ and, following the lines of the proof of \autoref{prop:diff_special_Nsupq}, the general statement can be derived easily using H\"older's inequality.
    
    Furthermore, again for $p\geq 1$ the main term can be bounded with the help of \cite[Lemma 3]{HoWe23}, now combined with \autoref{lem:discrete_morrey}, which yields
    \begin{align*}
        \sup_{P\in\mathcal{Q}} \frac{1}{\abs{P}^\tau} \norm{ \int_{B(\,\cdot\,, R) \,\cap\, \Omega} \abs{f(y)} \d y \sep  L_{p}(P \cap \Omega) }
        \lesssim \sup_{P\in\mathcal{Q}} \frac{1}{\abs{P}^\tau}   \norm{ f \sep L_{p}(P \cap \Omega) }.
    \end{align*}
    Therefore, the proof is complete.
\end{proof}

\subsection{Oscillations on Special Lipschitz Domains}
Here we shall derive characterizations in terms of local oscillations for Besov-Morrey and Besov-type spaces defined on special Lipschitz domains. 
In view of Propositions~\ref{prop:N_Omega_lower} and~\ref{prop:B_Omega_lower}, it suffices to discuss upper bounds.
For that purpose, we first recall the following result on quasi-optimal polynomials.
\begin{lemma}[{\cite[Lemma 6]{HoWe23}}]
\label{lem:opt_pol}
    Let $d,N\in\N$ and $\Omega \subset \mathbb{R}^d$ be a special Lipschitz domain. 
    For $x\in\Omega$ and $t>0$ consider the Hilbert spaces $\mathcal{H}_{x,t}^{N-1} := \big( \mathcal{P}_{N-1}, \distr{\cdot}{\cdot}_{x,t}\big)$ with inner product
    $$
        \distr{f}{g}_{x,t} := t^{-d} \int_{B(x,t)\cap \Omega} f(z)\, \overline{g(z)} \d z.
    $$
    Further, for an appropriate index set $\mathcal{I}$ let $\{p_{i,x,t} \sep i\in \mathcal{I}\}$ be an orthonormal basis of $\mathcal{H}_{x,t}^{N-1}$ with $  \sup_{y\in B(x,t)\cap \Omega} \abs{p_{i,x,t}(y)} \leq C $ for some $C>0$ independent of $x$ and $t$. Then there exist constants $c_1=c_1(C^2,\#\mathcal{I})>0$ and $c_2=c_2(c_1, d)>0$ such that the projection operators
    $$
        \Pi_{x,t}^{N-1} \colon L_1(B(x,t)\cap \Omega) \to \mathcal{H}_{x,t}^{N-1}, \qquad f\mapsto \Pi_{x,t}^{N-1} f := \sum_{i\in \mathcal{I}} \distr{f}{p_{i,x,t}}_{x,t} \, p_{i,x,t},
    $$
    satisfy 
    \begin{enumerate}
        \item for all $y\in B(x,t)\cap \Omega $ the pointwise bound $\abs{(\Pi_{x,t}^{N-1}f )(y)} \leq c_1 \, t^{-d} \int_{B(x,t)\cap \Omega} \abs{f(z)} \d z$,

        \item for all $ \pi \in\mathcal{P}_{N-1}$ the equation $ \Pi_{x,t}^{N-1}[f- \pi ] = \Pi_{x,t}^{N-1}f- \pi  \in \mathcal{P}_{N-1}$,
        
        \item the quasi-optimality $  t^{-d} \int_{B(x,t)\cap \Omega} \abs{f(z) - (\Pi_{x,t}^{N-1}f )(z)} \d z 
            \leq c_2 \, \osc_{1,\Omega}^{N-1}f(x,t) $,
        
        \item for a.e.\ $x\in\Omega$ the limit property $\lim_{t\to 0}\limits \big(\Pi_{x,t}^{N-1}f\big)(x) = f(x)$.
    \end{enumerate}
\end{lemma}

Now we are well-prepared to prove the following result concerning the Besov-Morrey spaces $\mathcal{N}^{s}_{u,p,q}(\Omega) $.

\begin{prop}\label{prop:osc_special}
    For $d\in\N$ let $ \Omega \subset \R$ be a special Lipschitz domain, $0< p \leq u < \infty$, $0 < q,T \leq \infty$, $1 \leq v \leq \infty$, $0<R<\infty$, $ N \in \N$, and $s>0$.
    Then
    \begin{align*}
        \norm{ f \sep \mathcal{N}^{s}_{u,p,q}(\Omega) } 
        & \lesssim  \norm{ \bigg(\int_{B(\,\cdot\,, R) \cap \Omega}   \abs{f(y)}^v \d y\bigg)^{\frac{1}{v}} \sep  \mathcal{M}^{u}_{p}( \Omega)} + \abs{f}^{(T,v,N)}_{\osc,\Omega}, \qquad  f \in L^\loc_{\max\{p,v\}}(\Omega).
    \end{align*}
    If additionally $p\geq 1$, then also $\norm{ f \sep \mathcal{N}^{s}_{u,p,q}(\Omega) } \lesssim  \norm{ f \sep  \mathcal{M}^{u}_{p}( \Omega)} + \abs{f}^{(T,v,N)}_{\osc,\Omega}$.
    In both cases the implied constants are independent of $f$.
\end{prop}

\begin{proof}
    It clearly suffices to prove the claim for $ T < N $. To this end, let $ f \in L^\loc_{\max\{p,v\}}(\Omega)$. 
    Then we can use \autoref{prop:diff_special_Nsupq} with $T=v=1$ to get
\begin{align*}
        \norm{ f \sep \mathcal{N}^{s}_{u,p,q}(\Omega) } 
        & \lesssim  \norm{ \int_{B(\,\cdot\,, R) \cap \Omega}   \abs{f(y)} \d y \sep  \mathcal{M}^{u}_{p}( \Omega)} + \abs{f}^{(1,1,N)}_{\Delta,\Omega} . 
    \end{align*}
In what follows we want to estimate the averaged differences in $\abs{f}^{(1,1,N)}_{\Delta,\Omega}$ from above in terms of polynomials and local oscillations of $f$. We use some ideas from \cite[Lemma~4.10]{ysy}, see also the proof of \cite[Theorem~1]{See1}. At first we observe that a change of measure yields
\begin{align*}
        \abs{f}^{(1,1,N)}_{\Delta,\Omega}  \lesssim \bigg( \int_{0}^{N} t^{-sq} \norm{ t^{-d} \int_{V^{N}(\,\cdot\,,\frac{t}{N})} \abs{\Delta^{N}_{h}f(\cdot)} \d h  \sep \mathcal{M}^{u}_{p}(\Omega) }^q \frac{\d t}{t} \bigg)^{\frac{1}{q}} .
    \end{align*} 
    Next we apply an estimate given in Step~1 of the proof of \cite[Proposition 3]{HoWe23}. 
    It yields that for all $x \in \Omega$, $t>0$, and each $ \pi \in\mathcal{P}_{N-1}$ we have
    \begin{equation}\label{eq_diff_p_osc1}
        t^{-d} \int_{V^{N}\!\left(x,\frac{t}{N}\right)} \abs{\Delta^{N}_{h} f(x)} \d h
        \lesssim \abs{[f- \pi ](x)} + t^{-d} \int_{B(x,t)\cap \Omega} \abs{[f- \pi ](y)} \d y.
    \end{equation}
    To continue, we can choose $ \pi :=\Pi_{x,t}^{N-1}f$ as defined in \autoref{lem:opt_pol} above. We thus find
    \begin{align*}
    \abs{f}^{(1,1,N)}_{\Delta,\Omega} &  \lesssim \bigg( \int_{0}^{N} t^{-sq} \norm{ f- \Pi_{(\cdot),t}^{N-1}f  \sep \mathcal{M}^{u}_{p}(\Omega) }^q \frac{\d t}{t} \bigg)^{\frac{1}{q}} \\
    &  \qquad \qquad + \bigg( \int_{0}^{N} t^{-sq} \norm{  t^{-d} \int_{B(\,\cdot\,,t)\cap \Omega} \abs{[f- \Pi_{(\cdot),t}^{N-1}f ](y)} \d y  \sep \mathcal{M}^{u}_{p}(\Omega) }^q \frac{\d t}{t} \bigg)^{\frac{1}{q}} .
    \end{align*} 
    For the second term we may use \autoref{lem:opt_pol}(iii) to get
    \begin{align*}
    \bigg( \int_{0}^{N} t^{-sq} \norm{  t^{-d} \int_{B(\,\cdot\,,t)\cap \Omega} \abs{[f- \Pi_{(\cdot),t}^{N-1}f ](y)} \d y  \sep \mathcal{M}^{u}_{p}(\Omega) }^q \frac{\d t}{t} \bigg)^{\frac{1}{q}}  \lesssim  \abs{f}^{(N,1,N)}_{\osc,\Omega} .
    \end{align*}
    Consequently we obtain
    \begin{align*}
    \abs{f}^{(1,1,N)}_{\Delta,\Omega} &  \lesssim \bigg( \int_{0}^{N} t^{-sq} \norm{ f- \Pi_{(\cdot),t}^{N-1}f  \sep \mathcal{M}^{u}_{p}(\Omega) }^q \frac{\d t}{t} \bigg)^{\frac{1}{q}}  +  \abs{f}^{(N,1,N)}_{\osc,\Omega}  .
    \end{align*} 
    In order to estimate the remaining integral, we use the following result which is shown in Step~2 of the proof of \cite[Proposition 3]{HoWe23} by means of \autoref{lem:opt_pol}:
    \begin{align*}%\label{eq_fp_sumosc_inf}
       \abs{\big[f-\Pi_{x,t}^{N-1}f\big](x)} 
        \lesssim \sum_{\ell=0}^{\infty} \osc_{1,\Omega}^{N-1}f(x,2^{-\ell}t), \qquad x\in\Omega,  t>0.
    \end{align*}
   Moreover, setting $ m := \min \{ 1 , p , q \} $ standard arguments show that
    \begin{align*}
        \norm{g+h \sep \mathcal{M}^{u}_{p}(\Omega)}^m 
        &\leq \norm{g \sep \mathcal{M}^{u}_{p}(\Omega)}^m + \norm{h \sep \mathcal{M}^{u}_{p}(\Omega)}^m, \qquad g, h \in \mathcal{M}^{u}_{p}(\Omega).
    \end{align*}
    Hence we observe
    \begin{align}
    & \bigg( \int_{0}^{N} t^{-sq} \norm{ f- \Pi_{(\cdot),t}^{N-1}f    \sep \mathcal{M}^{u}_{p}(\Omega) }^q \frac{\d t}{t} \bigg)^{\frac{m}{q}} \nonumber\\
    & \qquad \lesssim \Bigg( \int_{0}^{N} t^{-sq} \Bigg( \norm{  \sum_{\ell=0}^{\infty} \osc_{1,\Omega}^{N-1}f(\cdot,2^{-\ell}t)  \sep \mathcal{M}^{u}_{p}(\Omega) }^m \Bigg)^{\frac{q}{m}} \frac{\d t}{t} \Bigg)^{\frac{m}{q}} \nonumber\\
    & \qquad \lesssim \sum_{\ell=0}^{\infty} \bigg( \int_{0}^{N} t^{-sq} \norm{   \osc_{1,\Omega}^{N-1}f(\cdot,2^{-\ell}t)  \sep \mathcal{M}^{u}_{p}(\Omega) }^q \frac{\d t}{t} \bigg)^{\frac{m}{q}} \nonumber\\
    & \qquad \lesssim \sum_{\ell=0}^{\infty} 2^{- \ell sm}  \bigg( \int_{0}^{2^{- \ell} N} \tau^{-sq} \norm{   \osc_{1,\Omega}^{N-1}f(\cdot,\tau)  \sep \mathcal{M}^{u}_{p}(\Omega) }^q \frac{\d \tau}{\tau} \bigg)^{\frac{m}{q}} \nonumber\\
    & \qquad \lesssim   \bigg( \int_{0}^{ N} t^{-sq} \norm{   \osc_{1,\Omega}^{N-1}f(\cdot,t)  \sep \mathcal{M}^{u}_{p}(\Omega) }^q \frac{\d t}{t} \bigg)^{\frac{m}{q}}, \label{eq:proof_osc_N}
    \end{align}
    where in the last step we used that the series converges, as $ sm > 0$. Consequently,
    \begin{align*}
    \abs{f}^{(1,1,N)}_{\Delta,\Omega} \lesssim \abs{f}^{(N,1,N)}_{\osc,\Omega}  .
    \end{align*} 
    Since we assumed $T < N$, the definition of the local oscillation finally yields
    \begin{align*}
        \abs{f}^{(N,1,N)}_{\osc,\Omega}  
        & \lesssim  \abs{f}^{(T,1,N)}_{\osc,\Omega} +\bigg( \int_{T}^{N} t^{-sq} \norm{  \inf_{\pi \in \mathcal{P}_{N-1}}  t^{-d} \int_{B(\,\cdot\,,t) \cap \Omega} \abs{f(y) - \pi (y)} \d y   \sep \mathcal{M}^{u}_{p}(\Omega) }^q \frac{\d t}{t} \bigg)^{\frac{1}{q}} \\
        &\lesssim  \abs{f}^{(T,1,N)}_{\osc,\Omega} + \bigg( \int_{T}^{N} t^{-sq} \norm{   t^{-d} \int_{B(\,\cdot\,,t) \cap \Omega} \abs{f(y) } \d y  \sep \mathcal{M}^{u}_{p}(\Omega) }^q \frac{\d t}{t} \bigg)^{\frac{1}{q}} \\
        & \lesssim  \abs{f}^{(T,1,N)}_{\osc,\Omega} + \norm{   \int_{B(\,\cdot\,,N) \cap \Omega} \abs{f(y) } \d y  \sep \mathcal{M}^{u}_{p}(\Omega) }
    \end{align*}
    which together with \cite[Lemma 4]{HoWe23} proves the statement for $v=1$. Now the proof can be completed exactly as in Step 5 of \autoref{prop:diff_special_Nsupq}.
\end{proof}

In a next step we prove the corresponding result for the Besov-type spaces $ B^{s, \tau}_{p,q}(\Omega) $.

\begin{prop}\label{prop:osc_special_Bstpq}
    For $d\in\N$ let $ \Omega \subset \R$ be a special Lipschitz domain. 
    Let $0< p  < \infty$, $ 0 \leq \tau < \frac{1}{p} $, $0 < q,T \leq \infty$, $1 \leq v \leq \infty$, $0<R<\infty$, $ N \in \N$, and $s>0$.
    Then for $f \in L^\loc_{\max\{p,v\}}(\Omega)$ there holds
    \begin{align*}
        \norm{ f \sep B^{s, \tau}_{p,q}(\Omega) } 
        & \lesssim \sup_{P\in\mathcal{Q}} \frac{1}{\abs{P}^\tau} \norm{\bigg( \int_{B(\,\cdot\,,R)\cap \Omega} \abs{f(y)}^{v} \d y \bigg)^{\frac{1}{v}} \sep L_{p}(P \cap \Omega)}  + \abs{f}^{(T,v,N)}_{\osc,\tau,\Omega}  . 
    \end{align*}
    If additionally $p\geq 1$, then also
    \begin{align*}
    \norm{ f \sep B^{s, \tau}_{p,q}(\Omega) } \lesssim  \sup_{P\in\mathcal{Q}} \frac{1}{\abs{P}^\tau} \norm{ f \sep L_{p}(P\cap \Omega)} + \abs{f}^{(T,v,N)}_{\osc,\tau,\Omega}  . 
    \end{align*}
    In both cases the implied constants are independent of $f$.
\end{prop}

\begin{proof}
We proceed as in the proof of \autoref{prop:osc_special} and w.l.o.g.\ assume that $ T < N $ and $q<\infty$.
That is, given $ f \in L^\loc_{\max\{p,v\}}(\Omega)$ we apply \autoref{prop:diff_special_Bstpq} (with $T=v=1$) to get
\begin{align*}
        \norm{ f \sep B^{s, \tau}_{p,q}(\Omega) } 
        & \lesssim \sup_{P\in\mathcal{Q}} \frac{1}{\abs{P}^\tau} \norm{\int_{B(\,\cdot\,,R)\cap \Omega} \abs{f(y)} \d y  \sep L_{p}(P \cap \Omega)}  + \abs{f}^{(1,1,N)}_{\Delta, \tau , \Omega},
    \end{align*}
    where now the averaged differences in $\abs{f}^{(1,1,N)}_{\Delta, \tau , \Omega}$ will be estimated from above in terms of $\abs{f}^{(N,1,N)}_{\osc,\tau,\Omega}$.  As before a change of measure yields
    \begin{align*}
\abs{f}^{(1,1,N)}_{\Delta,\tau,\Omega}  \lesssim \sup_{P\in\mathcal{Q}} \frac{1}{\abs{P}^\tau} \bigg( \int_{0}^{N} t^{-sq} \norm{  t^{-d} \int_{V^{N}(\,\cdot\,,\frac{t}{N})} \abs{\Delta^{N}_{h}f(\cdot)} \d h  \sep L_{p}(P\cap \Omega) }^q \frac{\d t}{t} \bigg)^{\frac{1}{q}} .
    \end{align*}    
    We then apply \eqref{eq_diff_p_osc1} with $ \pi :=\Pi_{x,t}^{N-1}f$ as well as \autoref{lem:opt_pol}(iii) to obtain
\begin{align}
\abs{f}^{(1,1,N)}_{\Delta,\tau,\Omega} 
      &  \lesssim \sup_{P\in\mathcal{Q}} \frac{1}{\abs{P}^\tau} \bigg( \int_{0}^{N} t^{-sq} \norm{  f- \Pi_{(\cdot),t}^{N-1}f \sep L_{p}(P\cap \Omega) }^q \frac{\d t}{t} \bigg)^{\frac{1}{q}} \nonumber\\  
      & \quad + \sup_{P\in\mathcal{Q}} \frac{1}{\abs{P}^\tau} \bigg( \int_{0}^{N} t^{-sq} \norm{   t^{-d} \int_{B(\cdot,t)\cap \Omega} \abs{[f- \Pi_{(\cdot),t}^{N-1}f ](y)} \d y  \sep L_{p}(P\cap \Omega) }^q \frac{\d t}{t} \bigg)^{\frac{1}{q}} \nonumber\\
      &\lesssim \sup_{P\in\mathcal{Q}} \frac{1}{\abs{P}^\tau} \bigg( \int_{0}^{N} t^{-sq} \norm{  f- \Pi_{(\cdot),t}^{N-1}f    \sep L_{p}(P\cap \Omega) }^q \frac{\d t}{t} \bigg)^{\frac{1}{q}} + \abs{f}^{(N,1,N)}_{\osc,\tau,\Omega}. \label{eq:proof_oscB}
\end{align}
    To further estimate the first term, we fix $P\in \mathcal{Q}$ for a moment. Then a slight modification of \eqref{eq:proof_osc_N} applied with $m := \min \{ 1 , p , q \}$ as well as $u:=p$ (such that $\mathcal{M}_p^p=L_p$) yields
\begin{align*}
    \int_{0}^{N} t^{-sq} \norm{ (f- \Pi_{(\cdot),t}^{N-1}f) \, \chi_P  \sep  L_{p}(\Omega) }^q \frac{\d t}{t} 
    &\lesssim \int_{0}^{ N} t^{-sq} \norm{ \osc_{1,\Omega}^{N-1}f(\cdot,t) \, \chi_P(\cdot) \sep  L_{p}(\Omega) }^q \frac{\d t}{t}.
\end{align*}
In combination with \eqref{eq:proof_oscB} this implies %\begin{align*}
$\abs{f}^{(1,1,N)}_{\Delta,\tau,\Omega}  \lesssim  \abs{f}^{(N,1,N)}_{\osc,\tau,\Omega}$, as desired.
%\end{align*}
Now we again use $T < N$ and the definition of $\osc_{1,\Omega}^{N-1} f(\cdot, t)$ to find
\begin{align*}
    \abs{f}^{(N,1,N)}_{\osc,\tau,\Omega} 
    &\lesssim \abs{f}^{(T,1,N)}_{\osc,\tau,\Omega} + \sup_{P\in\mathcal{Q}} \frac{1}{\abs{P}^\tau} \left(\int_{T}^N t^{-sq} \norm{ \osc_{1,\Omega}^{N-1} f(\cdot, t) \sep L_p(P\cap \Omega) }^{q} \frac{\d t}{t} \right)^{\frac{1}{q}} \\
    &\lesssim \abs{f}^{(T,1,N)}_{\osc,\tau,\Omega} + \sup_{P\in\mathcal{Q}} \frac{1}{\abs{P}^\tau} \left(\int_{T}^N t^{-sq} \norm{   t^{-d} \int_{B(\,\cdot\,,t) \cap \Omega} \abs{f(y)} \d y   \sep L_p(P\cap \Omega) }^{q} \frac{\d t}{t} \right)^{\frac{1}{q}} \\
    &\lesssim \abs{f}^{(T,1,N)}_{\osc,\tau,\Omega} + \sup_{P\in\mathcal{Q}} \frac{1}{\abs{P}^\tau} \norm{ \int_{B(\,\cdot\,,N) \cap \Omega} \abs{f(y)} \d y   \sep L_p(P\cap \Omega) }.
\end{align*}
Thus, \cite[Lemma 4]{HoWe23} combined with \autoref{lem:discrete_morrey} concludes the proof in the case $v=1$. Once more, the general case can be derived from this using the arguments given in Step~5 of the proof of  \autoref{prop:diff_special_Nsupq}.
\end{proof}

\subsection{Oscillations on Bounded Lipschitz Domains}
In this subsection we transfer our previous findings to the case of \emph{bounded} Lipschitz domains $\Omega\subset\R$. We start with the case of Besov-Morrey spaces. 

\begin{prop}\label{prop:osc}
    For $d\in\N$ let $ \Omega \subset \R$ be a bounded Lipschitz domain, $0< p \leq u < \infty$, $0 < q,T \leq \infty$, $1 \leq v \leq \infty$, $0<R<\infty$, $ N \in \N$, and $s>\sigma_p$.
    Then
    \begin{align*}
        \norm{ f \sep \mathcal{N}^{s}_{u,p,q}(\Omega) } 
        & \lesssim  \norm{ \bigg(\int_{B(\,\cdot\,, R) \cap \Omega}   \abs{f(y)}^v \d y\bigg)^{\frac{1}{v}} \sep  \mathcal{M}^{u}_{p}( \Omega)} + \abs{f}^{(T,v,N)}_{\osc,\Omega}, \qquad f \in L^\loc_{\max\{p,v\}}(\Omega). 
    \end{align*}
    If additionally $p\geq 1$, then also $\norm{ f \sep \mathcal{N}^{s}_{u,p,q}(\Omega) } \lesssim  \norm{ f \sep  \mathcal{M}^{u}_{p}( \Omega)} + \abs{f}^{(T,v,N)}_{\osc,\Omega}$.
    In both cases the implied constants are independent of $f$.
\end{prop}

\begin{proof}
For the proof we start with a standard localization procedure as described in detail in \cite[Section 4.2]{HoWe23}. 
First of all, there exist open balls $B_1,\ldots, B_m$ in $\R$, affine-linear diffeomorphisms $\Phi_1,\ldots,\Phi_m\colon\R\to\R$, and $[0,1]$-valued functions $\sigma_1,\ldots,\sigma_m\in\mathcal{D}(\R)$ with the following properties for $k=1,\ldots,m$:
    \begin{itemize}
        \item $B_k \cap\partial\Omega\neq\emptyset$ and $\partial\Omega\subset\bigcup_{k=1}^m B_k$,
        \item $\supp \sigma_k \subset B_k$ and $\sum_{k=1}^m\sigma_k \equiv 1$ on some neighborhood of $\partial\Omega$,
        \item $\Phi_k(B_k)\cap \Phi_k(\Omega)$ can be extended to a special Lipschitz domain $\omega_k \subset\R$.
    \end{itemize}
    Setting $\sigma_0 := (1- \sum_{k=1}^m\sigma_k) \chi_\Omega$ then yields $\sigma_0 \in\mathcal{D}(\R)$ with values in $[0,1]$ and $\supp \sigma_0\subset \Omega$ such that $\sum_{k=0}^m\sigma_k\equiv 1$ on $\Omega_\varepsilon$ with some $\varepsilon>0$.
    Following the lines of the proof of \cite[Proposition 4]{HoWe23} with \cite[Lemma 5]{HoWe23} replaced by \autoref{lem:tools}, we see that every $f\in\mathcal{D}'(\Omega)$ can be decomposed into a finite sum of suitably localized building blocks such that
    \begin{align*}
        \norm{f \sep \mathcal{N}^{s}_{u,p,q}(\Omega)} 
        &\lesssim \norm{f_0 \sep \mathcal{N}^{s}_{u,p,q}(\R)} + \sum_{k=1}^m \norm{f_k \sep \mathcal{N}^{s}_{u,p,q}(\omega_k)} .
    \end{align*}
    Moreover, we observe that $f\in L^\loc_{\max\{p,v\}}(\Omega)$ is regular and we may take
    $$
        f_0:=\sigma_0 Ef \in L^\loc_{\max\{p,v\}}(\R)
    $$
    (where again $E$ denotes the trivial extension from $\Omega$ to $\R$) as well as
    $$
         f_k:=\big[(\sigma_k Ef) \circ \Phi_k^{-1}\big]\vert_{\omega_k}\in L^\loc_{\max\{p,v\}}(\omega_k), \qquad k=1,\ldots,m.
    $$
    Now the idea is to use our previously derived results to upper bound the quasi-norms of these $f_k$ by our quantities of interest based on $f$ itself, in order to prove estimates that match the lower bounds stated in \autoref{prop:N_Omega_lower}.
    To do so, let $\widetilde{N}:=N+L-1$ with $N\in\N$ and $L>s+d$ such that in particular $s<\widetilde{N}$. 
    Since $ s> \sigma_p $ \autoref{thm_osc_Rda=2} (applied with $\widetilde{R}:=\widetilde{v}:=1$ as well as $\widetilde{N}$ and some small $0<\widetilde{T}<\infty$ to be specified later on) implies
    \begin{align*}
        \norm{ f_0 \sep \mathcal{N}^{s}_{u,p,q}(\R) }
        \sim \norm{ f_0 \sep \mathcal{N}^{s}_{u,p,q}(\R) }^{(1,\widetilde{T},1,\widetilde{N})}_\osc
        = \norm{ \int_{B(\,\cdot\,, 1)}   \abs{f_0(y)} \d y \sep  \mathcal{M}^{u}_{p}( \R)} + \abs{f_0}^{(\widetilde{T},1,\widetilde{N})}_{\osc}
    \end{align*}
    and similarly for $k=1,\ldots,m$ \autoref{prop:osc_special} yields
    \begin{align*}
        \norm{ f_k \sep \mathcal{N}^{s}_{u,p,q}(\omega_k) }
        \lesssim \norm{ \int_{B(\,\cdot\,, 1)\cap \omega_k}   \abs{f_k(y)} \d y \sep  \mathcal{M}^{u}_{p}(\omega_k)} + \abs{f_k}^{(\widetilde{T},1,\widetilde{N})}_{\osc,\omega_k}.
    \end{align*}
    It is easily shown that therein for each $k=0,\ldots,m$ the main term is bounded by
    $$
        \norm{ \int_{B(\,\cdot\,, R)\cap \Omega}  \abs{f(y)} \d y \sep  \mathcal{M}^{u}_{p}(\Omega)},
    $$
    see Steps~3b and 4 in the proof of \cite[Proposition~4]{HoWe23} for details.
    In order to handle the quasi-semi-norms for $k=1,\ldots,m$ we can use the support properties of $f_k$ to conclude that there exist constants $t_k,c_k>0$ depending on $\Omega$ such that
    for all $x\in \omega_k$ and $0\leq t\leq t_k$
    \begin{align}\label{eq:proof_oscN}
        \osc_{1,\omega_k}^{\widetilde{N}-1}f_k(x,t) 
        \lesssim \chi_{\omega_k\cap\Omega_k}(x) \, E\Big(\osc_{1,\Omega}^{\widetilde{N}-1} \big[ (\sigma_k Ef)|_\Omega\big]\big(\Phi_k^{-1}(\cdot), c_k t\big) \Big)(x).
    \end{align}
    We refer to Step~3c in the proof of \cite[Proposition~4]{HoWe23} for details.
    Setting $\Omega_k:=\Phi_{k}(\Omega)$, \cite[Lemma~3(v) and (iv)]{HoWe23} thus yields
    \begin{align}
        \norm{\osc_{1,\omega_k}^{\widetilde{N}-1}f_k(\cdot,t) \sep \mathcal{M}^u_p(\omega_k)}
        &\lesssim \norm{\chi_{\omega_k\cap\Omega_k}(\cdot) \, E\Big(\osc_{1,\Omega}^{\widetilde{N}-1} \big[ (\sigma_k Ef)|_\Omega\big]\big(\Phi_k^{-1}(\cdot), c_k t\big) \Big)\Big|_{\Omega_k}  \sep \mathcal{M}^u_p(\Omega_k)} \nonumber\\
        &\leq \norm{ \Big(\osc_{1,\Omega}^{\widetilde{N}-1} \big[ (\sigma_k Ef)|_\Omega\big]\big(\cdot, c_k t\big) \Big) \circ \Phi_k^{-1}  \sep \mathcal{M}^u_p(\Phi_k(\Omega))} \nonumber\\
        &\sim \norm{ \osc_{1,\Omega}^{\widetilde{N}-1} \big[ (\sigma_k Ef)|_\Omega\big](\cdot, c_k t) \sep \mathcal{M}^u_p(\Omega)}, \qquad 0\leq t \leq t_k. \label{eq:proof_osc_morrey_bound}
    \end{align}
    Now for fixed $x\in\Omega$ a Taylor expansion of degree $L-1=\widetilde{N}-N$ of $\sigma_k$ into
    $$
        \sigma_k(y)=T_k(y)+R_k(y) \quad \text{with} \quad \abs{T_k(y)}\lesssim 1 \quad\text{and}\quad \abs{R_k(y)}\lesssim t^L, \qquad y\in B(x,c_k t)\cap\Omega,
    $$
    allows to further simplify this expression, as
    $$
        \osc_{1,\Omega}^{\widetilde{N}-1} \big[ (\sigma_k Ef)|_\Omega\big](x, c_k t)
        \lesssim \osc_{1,\Omega}^{N-1} f(x, c_k t) + t^{L-d} \int_{B(x,c_k t)\cap \Omega} \abs{f(y)} \d y
    $$
    with implied constants independent of $x$ and $t$. Again we refer to \cite[Proposition~4]{HoWe23}, see Step 3c of the proof. Hence $L>s+d$ shows that for every $k=1,\ldots,m$
    \begin{align*}
        \abs{f_k}^{(\widetilde{T},1,\widetilde{N})}_{\osc,\omega_k} 
         &\lesssim \bigg( \int_{0}^{\widetilde{T}} t^{-sq} \norm{ \osc_{1,\Omega}^{\widetilde{N}-1} \big[ (\sigma_k Ef)|_\Omega\big](\,\cdot\,, c_k t) \sep \mathcal{M}^u_p(\Omega)}^q \frac{\d t}{t} \bigg)^{\frac{1}{q}} \\
         &\lesssim \bigg( \int_{0}^{\widetilde{T}} t^{-sq} \norm{ \osc_{1,\Omega}^{N-1} f(\,\cdot\,, c_k t) \sep \mathcal{M}^u_p(\Omega)}^q \frac{\d t}{t} \bigg)^{\frac{1}{q}} \\
         &\qquad \qquad + \bigg( \int_{0}^{\widetilde{T}} t^{(L-s-d)q} \norm{ \int_{B(\cdot,c_k\,t)\cap \Omega} \abs{f(y)} \d y \sep \mathcal{M}^u_p(\Omega)}^q \frac{\d t}{t} \bigg)^{\frac{1}{q}}\\
         &\lesssim \abs{f}^{(T,1,N)}_{\osc,\Omega} + \norm{ \int_{B(\,\cdot\,,R)\cap \Omega} \abs{f(y)} \d y \sep \mathcal{M}^u_p(\Omega)}
    \end{align*}
    if $\widetilde{T}$ was chosen smaller than each $\min\{t_k, \frac{T}{c_k}, \frac{R}{c_k}\}$.
    For $k=0$ exactly the same arguments (formally setting $\omega_0:=\R$, $\Phi_0:=\mathrm{id}$ and $c_0=1$) prove that also
    \begin{align*}
        \abs{f_0}^{(\widetilde{T},1,\widetilde{N})}_{\osc} 
         &\lesssim \abs{f}^{(T,1,N)}_{\osc,\Omega} + \norm{ \int_{B(\,\cdot\,,R)\cap \Omega} \abs{f(y)} \d y \sep \mathcal{M}^u_p(\Omega)} .
    \end{align*}
    A combination of these estimates with the technique described in Step 5 of the proof of \autoref{prop:diff_special_Nsupq} completes the proof.
\end{proof}

Similarly, we derive the corresponding assertion for Besov-type spaces on bounded Lipschitz domains.
\begin{prop}\label{prop:osc_Bstpq}
    For $d\in\N$ let $ \Omega \subset \R$ be a bounded Lipschitz domain. Let $0< p  < \infty$, $ 0 \leq \tau < \frac{1}{p} $, $0 < q,T \leq \infty$, $1 \leq v \leq \infty$, $0<R<\infty$, $ N \in \N$, and $s>\sigma_p$.
    Then for $f \in L^\loc_{\max\{p,v\}}(\Omega)$ there holds
    \begin{align*}
        \norm{ f \sep B^{s, \tau}_{p,q}(\Omega) } 
        & \lesssim \sup_{P\in\mathcal{Q}} \frac{1}{\abs{P}^\tau} \norm{\Big( \int_{B(\,\cdot\,,R)\cap \Omega} \abs{f(y)}^{v} \d y \Big)^{\frac{1}{v}} \sep L_{p}(P \cap \Omega)}  + \abs{f}^{(T,v,N)}_{\osc,\tau,\Omega}  . 
    \end{align*}
    If additionally $p\geq 1$, then also
    \begin{align*}
    \norm{ f \sep B^{s, \tau}_{p,q}(\Omega) } \lesssim  \sup_{P\in\mathcal{Q}} \frac{1}{\abs{P}^\tau} \norm{ f \sep L_{p}(P\cap \Omega)} + \abs{f}^{(T,v,N)}_{\osc,\tau,\Omega}  . 
    \end{align*}
    In both cases the implied constants are independent of $f$.
\end{prop}
\begin{proof}
    This result can be proved with similar methods as described in the proof of \autoref{prop:osc}. Let us only sketch the needed modifications.
    At first, the same localization as above together with \autoref{thm_osc_B_Rda=2} and \autoref{prop:osc_special_Bstpq} yields
    \begin{align*}
        \norm{f \sep B^{s,\tau}_{p,q}(\Omega)} 
        &\lesssim \norm{f_0 \sep B^{s,\tau}_{p,q}(\R)} + \sum_{k=1}^m \norm{f_k \sep B^{s,\tau}_{p,q}(\omega_k)},
    \end{align*}
    where
    \begin{align*}
        \norm{ f_0 \sep B^{s,\tau}_{p,q}(\R) }
        \sim  \sup_{P\in\mathcal{Q}} \frac{1}{\abs{P}^\tau} \norm{ \int_{B(\,\cdot\,, 1)}   \abs{f_0(y)} \d y \sep  L_{p}(P)} + \abs{f_0}^{(\widetilde{T},1,\widetilde{N})}_{\osc,\tau}
    \end{align*}
    as well as for $k=1,\ldots,m$
    \begin{align*}
        \norm{ f_k \sep B^{s,\tau}_{p,q}(\omega_k) }
        \lesssim \sup_{P\in\mathcal{Q}} \frac{1}{\abs{P}^\tau} \norm{ \int_{B(\,\cdot\,, 1)\cap \omega_k}   \abs{f_k(y)} \d y \sep  L_{p}(P\cap \omega_k)} + \abs{f_k}^{(\widetilde{T},1,\widetilde{N})}_{\osc,\tau,\omega_k}.
    \end{align*}
    Therein, thanks to \autoref{lem:discrete_morrey}, the main terms can be equivalently rewritten in terms of Morrey-norms and thus be estimated exactly as before. 
    
    For the quasi-semi-norms we note that there is some $K=K_k\in\N$ such that for each fixed cube $P\in\mathcal{Q}$ we find disjoint $P_1,\ldots,P_K\in \mathcal{Q}$ with
    $$
        P \subseteq \bigcup_{\ell=1}^K \Phi_k(P_\ell) \quad \text{and}\quad \abs{P}=\abs{P_1}=\ldots=\abs{P_K}. 
    $$
    Then we may multiply \eqref{eq:proof_oscN} by $\chi_P\leq \sum_{\ell=1}^K \chi_{\Phi_k(P_\ell)}$ to replace \eqref{eq:proof_osc_morrey_bound} by
    \begin{align*}
        &\norm{\osc_{1,\omega_k}^{\widetilde{N}-1}f_k(\cdot,t) \sep L_p(P\cap\omega_k)} \\
        &\qquad \lesssim \sum_{\ell=1}^K \norm{\chi_{\Phi_k(P_\ell)}(\cdot)\, \chi_{\omega_k\cap\Omega_k}(\cdot) \, E\Big(\osc_{1,\Omega}^{\widetilde{N}-1} \big[ (\sigma_k Ef)|_\Omega\big]\big(\Phi_k^{-1}(\cdot), c_k t\big) \Big)\Big|_{\omega_k}  \sep L_p(\omega_k)} \\
        &\qquad\sim \sum_{\ell=1}^K \norm{\chi_{\Phi_k(P_\ell)}(\cdot)\,\chi_{\omega_k\cap\Omega_k}(\cdot) \, E\Big(\osc_{1,\Omega}^{\widetilde{N}-1} \big[ (\sigma_k Ef)|_\Omega\big]\big(\Phi_k^{-1}(\cdot), c_k t\big) \Big)\Big|_{\Omega_k}  \sep L_p(\Omega_k)} \\
        &\qquad\leq \sum_{\ell=1}^K \norm{ \Big(\chi_{P_\ell}(\cdot)\, \osc_{1,\Omega}^{\widetilde{N}-1} \big[ (\sigma_k Ef)|_\Omega\big]\big(\cdot, c_k t\big) \Big) \circ \Phi_k^{-1}  \sep L_p(\Phi_k(\Omega))} \\
        &\qquad\sim \sum_{\ell=1}^K \norm{ \osc_{1,\Omega}^{\widetilde{N}-1} \big[ (\sigma_k Ef)|_\Omega\big](\cdot, c_k t) \sep L_p(P_\ell\cap\Omega)}, \qquad 0\leq t \leq t_k,
    \end{align*}
    since \cite[Lemma~3]{HoWe23} particularly applies for $L_p=\mathcal{M}^p_p$. Still following the lines of the proof of \autoref{prop:osc} we find
    \begin{align*}
         &\frac{1}{\abs{P}^\tau}\bigg( \int_{0}^{\widetilde{T}} t^{-sq} \norm{\osc_{1,\omega_k}^{\widetilde{N}-1}f_k(\cdot,t) \sep L_p(P\cap \omega_k)}^q \frac{\d t}{t} \bigg)^{\frac{1}{q}} \\
         &\qquad \qquad \lesssim \sum_{\ell=1}^K \frac{1}{\abs{P_\ell}^\tau} \bigg( \int_{0}^{\widetilde{T}} t^{-sq} \norm{ \osc_{1,\Omega}^{N-1} f(\,\cdot\,, c_k t) \sep L_p(P_\ell\cap\Omega)}^q \frac{\d t}{t} \bigg)^{\frac{1}{q}} \\
         &\qquad\qquad \qquad \qquad + \sum_{\ell=1}^K \frac{1}{\abs{P_\ell}^\tau} \norm{ \int_{B(\,\cdot\,,R)\cap \Omega} \abs{f(y)} \d y \sep L_p(P_\ell\cap\Omega)}\\
         &\qquad \qquad \lesssim \abs{f}^{(T,1,N)}_{\osc,\tau,\Omega} + \sup_{P'\in\mathcal{Q}} \frac{1}{\abs{P'}^\tau} \norm{ \int_{B(\,\cdot\,,R)\cap \Omega} \abs{f(y)} \d y \sep L_p(P'\cap\Omega)}
    \end{align*}
    where again $\widetilde{T}$ has to be chosen small enough compared to $T$ and $R$. Taking the sup w.r.t.~$P$ thus shows the needed bounds
    \begin{align*}
         \abs{f_k}^{(\widetilde{T},1,\widetilde{N})}_{\osc,\tau,\omega_k} 
         &\lesssim \abs{f}^{(T,1,N)}_{\osc,\tau,\Omega} + \sup_{P\in\mathcal{Q}} \frac{1}{\abs{P}^\tau} \norm{ \int_{B(\,\cdot\,,R)\cap \Omega} \abs{f(y)} \d y \sep L_p(P\cap\Omega)}, \qquad k=1,\ldots,m,
    \end{align*}
    and likewise for $k=0$ which implies the statement for $v=1$. The general case is obtained from this exactly as before; see the proof of \autoref{prop:osc}.
\end{proof}

\subsection{Differences on Bounded Convex Lipschitz Domains}\label{Subsec_Diff_conv}

Here we extend our results on differences to the case of bounded convex Lipschitz domains by proving counterparts to \cite[Proposition~5]{HoWe23} for the scales of Besov-Morrey and Besov-type spaces, respectively.

\begin{prop}\label{prop:diff_conv}
    For $d \in \N$ let $\Omega\subset\R$ be a bounded convex Lipschitz domain. 
    Further let 
    $1 < p \leq u < \infty$, $0 < q,T \leq \infty$, $0 < R < \infty$, $N \in \N$, and $s > 0$. 
    Then there holds
    \begin{align*}
        \norm{ f \sep \mathcal{N}^s_{u,p,q}(\Omega) } 
        & \lesssim  \norm{ \esssup_{y\in B(\,\cdot\,, R) \cap \Omega} \abs{f(y)} \sep  \mathcal{M}^{u}_{p}( \Omega)} + \abs{f}^{(T,\infty,N)}_{\Delta,\Omega}, \qquad f \in L^\loc_\infty(\Omega),
    \end{align*}
    as well as $\norm{ f \sep \mathcal{N}^s_{u,p,q}(\Omega) } \lesssim  \norm{ f \sep  \mathcal{M}^{u}_{p}( \Omega)} + \abs{f}^{(T,\infty,N)}_{\Delta,\Omega}$, with  constants independent of $f$.
\end{prop}
\begin{proof}
    We closely follow the lines of the proof of \cite[Proposition~5]{HoWe23} and w.l.o.g.\ assume that $q<\infty$ as well as $T\geq 2^{-J+1}$ for some $J\in\N$.
    Due to \autoref{prop:osc} there holds
    $$
        \norm{ f \sep \mathcal{N}^{s}_{u,p,q}(\Omega) } \lesssim  \norm{ f \sep  \mathcal{M}^{u}_{p}( \Omega)} + \abs{f}^{(2^{-J},1,N)}_{\osc,\Omega}
        \leq \norm{ \esssup_{y\in B(\,\cdot\,, R) \cap \Omega} \abs{f(y)} \sep  \mathcal{M}^{u}_{p}( \Omega)} + \abs{f}^{(2^{-J},1,N)}_{\osc,\Omega},
    $$
    such that it suffices to bound $\abs{f}^{(2^{-J},1,N)}_{\osc,\Omega}$ in terms of $\abs{f}^{(T,\infty,N)}_{\Delta,\Omega}$.
    For this purpose, we use a specially tailored Whitney-type estimate for convex Lipschitz domains~\cite{DekLev} (see \cite[Lemma~7]{HoWe23} for details)
    which implies that 
    \begin{align}\label{eq:proof_whitney}
        \osc_{1,\Omega}^{N-1} f (x,t) 
        \leq c\, \textit{\textbf{M}} \left( \chi_{\Omega}(\cdot) \esssup_{h\in V^N(\,\cdot\,, 2t)} \abs{\Delta_h^N f(\cdot)} \right)(x), \qquad x\in \Omega,\; 0<t\leq 2^{-J},
    \end{align}
    with $c$ being independent of $x$, $t$, and $f$;
    cf.\ \cite[Proposition~5]{HoWe23}. 
    Therein \textit{\textbf{M}} denotes the Hardy-Littlewood maximal operator which is known to be bounded on $\mathcal{M}^{u}_{p}(\R)$ if $p>1$; see, e.g., \cite[Theorem~6.19]{SawBook}.
    Thus,
    \begin{align*}
        \abs{f}^{(2^{-J},1,N)}_{\osc,\Omega} 
        &\stackrel{\eqref{eq:proof_whitney}}{\lesssim} \bigg( \int_{0}^{2^{-J}} t^{-sq} \norm{ \textit{\textbf{M}} \left( \chi_{\Omega}(\ast) \esssup_{h\in V^N(\ast\,, 2t)} \abs{\Delta_h^N f(\ast)} \right)(\cdot) \sep \mathcal{M}^{u}_{p}(\Omega)}^{q} \frac{\d t}{t} \bigg)^{\frac{1}{q}} \\
        &\lesssim \bigg( \int_{0}^{2^{-J+1}} \tau^{-sq} \norm{ \textit{\textbf{M}} \left( \chi_{\Omega}(\ast) \esssup_{h\in V^N(\ast\,, \tau)} \abs{\Delta_h^N f(\ast)} \right)(\cdot) \sep \mathcal{M}^{u}_{p}(\R)}^{q} \frac{\d \tau}{\tau} \bigg)^{\frac{1}{q}} \\
        &\lesssim \bigg( \int_{0}^{2^{-J+1}} t^{-sq} \norm{ \chi_{\Omega}(\cdot) \esssup_{h\in V^N(\cdot\,, t)} \abs{\Delta_h^N f(\cdot)} \sep \mathcal{M}^{u}_{p}(\R)}^{q} \frac{\d t}{t} \bigg)^{\frac{1}{q}} \\
%        &\lesssim \bigg( \int_{0}^{T} t^{-sq} \norm{ \esssup_{h\in V^N(\cdot\,, t)} \abs{\Delta_h^N f(\cdot)} \sep \mathcal{M}^{u}_{p}(\Omega)}^{q} \frac{\d t}{t} \bigg)^{\frac{1}{q}} \\
        &\lesssim \abs{f}^{(T,\infty,N)}_{\Delta,\Omega},
    \end{align*}
    where the last inequality is due to \cite[Lemma~3(i)]{HoWe23} and $2^{-J+1}\leq T$.
\end{proof}

The corresponding assertion for Besov-type spaces reads as follows.
\begin{prop}\label{prop:diff_conv_tau}
    For $d \in \N$ let $\Omega\subset\R$ be a bounded convex Lipschitz domain. 
    Further let $1 < p < \infty$, $0\leq\tau < \frac{1}{p}$, $0 < q,T \leq \infty$, $0 < R < \infty$, $N \in \N$, and $s > 0$.
    Then we have
    %for all $f \in L^\loc_\infty(\Omega)$ there holds
    \begin{align*}
        \norm{ f \sep B^{s,\tau}_{p,q}(\Omega) } 
        & \lesssim \sup_{P\in\mathcal{Q}} \frac{1}{\abs{P}^\tau} \norm{ \esssup_{y\in B(\,\cdot\,, R) \cap \Omega} \abs{f(y)} \sep  L_{p}(P\cap \Omega)} + \abs{f}^{(T,\infty,N)}_{\Delta,\tau,\Omega}, \qquad  f \in L^\loc_\infty(\Omega),
    \end{align*}
    and $\norm{ f \sep B^{s,\tau}_{p,q}(\Omega) } \lesssim  \sup_{P\in\mathcal{Q}}\limits \frac{1}{\abs{P}^\tau} \norm{ f \sep L_{p}(P\cap \Omega)} + \abs{f}^{(T,\infty,N)}_{\Delta,\tau,\Omega}$, with constants independent of~$f$.
\end{prop}
\begin{proof}
    We argue as in the proof of  \autoref{prop:diff_conv}.
    W.l.o.g.\ we may assume $q<\infty$ and $T\geq 2^{-J+2}$ with $J\in\N$ such that, in view of \autoref{prop:osc_Bstpq}, it suffices to 
    prove $\abs{f}^{(2^{-J},1,N)}_{\osc,\tau,\Omega}\lesssim  \abs{f}^{(T,\infty,N)}_{\Delta,\tau,\Omega}$.
    To this end, note that for all fixed $P\in\mathcal{Q}$ Formula~\eqref{eq:proof_whitney} implies
    \begin{align*}
        &\int_0^{2^{-J}} t^{-sq} \norm{\osc_{1,\Omega}^{N-1} f (\cdot,t) \sep L_p(P\cap\Omega)}^q \frac{\d t}{t} \\
        &\qquad \lesssim\sum_{j=J-1}^\infty \int_{2^{-(j+2)}}^{2^{-(j+1)}} t^{-sq} \norm{\textit{\textbf{M}} \left( \chi_{\Omega}(\ast) \esssup_{h\in V^N(\,\ast\,, 2t)} \abs{\Delta_h^N f(\ast)} \right)(\cdot) \sep L_p(P)}^q \frac{\d t}{t} \\
        &\qquad \lesssim \sum_{j=J-1}^\infty \norm{\textit{\textbf{M}} \left( 2^{js}\, \chi_{\Omega}(\ast) \esssup_{h\in V^N(\,\ast\,, 2^{-j})} \abs{\Delta_h^N f(\ast)} \right)(\cdot) \sep L_p(P)}^q.
    \end{align*}
    As $p>1$, the vector-valued maximal inequality \cite[Proposition~2.3(ii)]{ZSYY1} thus yields that
    \begin{align*}
        \abs{f}^{(2^{-J},1,N)}_{\osc,\tau,\Omega}
        &\lesssim \sup_{P\in\mathcal{Q}} \frac{1}{\abs{P}^\tau} \left( \sum_{j=J-1}^\infty \norm{\textit{\textbf{M}} \left( 2^{js}\, \chi_{\Omega}(\ast) \esssup_{h\in V^N(\,\ast\,, 2^{-j})} \abs{\Delta_h^N f(\ast)} \right)(\cdot) \sep L_p(P)}^q \right)^{\frac{1}{q}} \\
        &\lesssim \sup_{P\in\mathcal{Q}} \frac{1}{\abs{P}^\tau} \left( \sum_{j=J-1}^\infty 2^{jsq} \norm{ \chi_{\Omega}(\cdot) \esssup_{h\in V^N(\,\cdot\,, 2^{-j})} \abs{\Delta_h^N f(\cdot)} \sep L_p(P)}^q \right)^{\frac{1}{q}} \\
        &\lesssim \sup_{P\in\mathcal{Q}} \frac{1}{\abs{P}^\tau} \left( \sum_{j=J-1}^\infty \int_{2^{-j}}^{2^{-(j-1)}} t^{-sq} \norm{ \esssup_{h\in V^N(\,\cdot\,, t)} \abs{\Delta_h^N f(\cdot)} \sep L_p(P\cap \Omega)}^q \frac{\d t}{t} \right)^{\frac{1}{q}} \\
        &\leq \abs{f}^{(T,\infty,N)}_{\Delta,\tau,\Omega}
    \end{align*}
    since we assumed $T\geq 2^{-J+2}$.
\end{proof}

\section{Proofs of the Main Theorems}\label{sect:proofs}
Now we are well-prepared to prove our main results stated in Theorems~\ref{mainresult1}--\ref{thm_main_diff_2}.

\vspace{0,2 cm}

\textbf{Proof of \autoref{mainresult1}. } Part~(i) is proven in \autoref{thm_osc_Rda=2} while part~(ii) follows from \autoref{prop:N_Omega_lower} combined with Propositions~\ref{prop:osc_special} and \ref{prop:osc}.\qed

\vspace{0,2 cm}

\textbf{Proof of \autoref{mainresult2}. }  \autoref{mainresult2}(i) corresponds to \autoref{thm_osc_B_Rda=2}. For (ii) we combine
\autoref{prop:B_Omega_lower} with Propositions~\ref{prop:osc_special_Bstpq} and \ref{prop:osc_Bstpq}.\qed

\vspace{0,2 cm}

\textbf{Proof of \autoref{thm_main_diff_1}. } The first assertion is shown in \autoref{thm_diff_Rda=2}. The lower bounds for (ii) and (iii) are given in \autoref{prop:N_Omega_lower} and for the corresponding upper bounds we refer to Propositions~\ref{prop:diff_special_Nsupq} and \ref{prop:diff_conv}, respectively.\qed

\vspace{0,2 cm}

\textbf{Proof of \autoref{thm_main_diff_2}. } 
Here part~(i) corresponds to \autoref{thm_diff_B_Rda=2}. For the remaining assertions~(ii) and (iii) we combine \autoref{prop:B_Omega_lower} with Propositions~\ref{prop:diff_special_Bstpq} and~\ref{prop:diff_conv_tau}.\qed
    
\vspace{0,3 cm}
\noindent
\textbf{Acknowledgments.} 
Marc Hovemann has been supported by Deutsche Forschungsgemeinschaft (DFG), grant HO 7444/1-1 with project number 528343051.

\end{document}